\documentclass{amsart}
\usepackage{mathrsfs}
\usepackage{bbm}
\usepackage{bbding}
\usepackage[mathcal]{euscript}
\usepackage{graphicx}
\usepackage{amsthm,amscd}
\usepackage{calc}
\usepackage{mathrsfs,dsfont}
\usepackage{CJK,fancyhdr,amscd}
\usepackage{anysize}
\usepackage{amsmath,amssymb,amsfonts}
\usepackage{epsfig,enumerate}
\usepackage{latexsym}
\usepackage{indentfirst, latexsym}
\usepackage{graphics}
\usepackage[all,poly,knot]{xy}
\usepackage[colorlinks,linkcolor=blue,anchorcolor=blue,citecolor=blue,pagebackref]{hyperref}%

\providecommand{\U}[1]{\protect\rule{.1in}{.1in}}

\numberwithin{equation}{section}
\newtheorem{theorem} {Theorem} [section]
\newtheorem{proposition}[theorem]{Proposition}
\newtheorem{corollary}  [theorem]     {Corollary}
\newtheorem{lemma}  [theorem]     {Lemma}
\newtheorem{question}  [theorem]     {Question}

\newtheorem{remark}  [theorem]     {Remark}

\newtheorem{observation}  [theorem]     {Observation}
\theoremstyle{definition}
\newtheorem{definition}  [theorem]     {Definition}

\newtheorem{example}  [theorem]     {Example}

\renewcommand{\1}{\mathds{1}}

\newcommand{\btheorem}{\begin{theorem}}
\newcommand{\etheorem}{\end{theorem}}
\newcommand{\bproposition}{\begin{proposition}}
\newcommand{\eproposition}{\end{proposition}}
\newcommand{\bdefinition}{\begin{definition}}
\newcommand{\edefinition}{\end{definition}}
\newcommand{\bcorollary}{\begin{corollary}}
\newcommand{\ecorollary}{\end{corollary}}
\newcommand{\bproof}{\begin{proof}}
\newcommand{\eproof}{\end{proof}}
\newcommand{\beq}{\begin{equation}}
\newcommand{\eeq}{\end{equation}}
\newcommand{\ee}{\end{eqnarray*}}
\newcommand{\be}{\begin{eqnarray*}}

\newcommand{\elemma}{\end{lemma}}
\newcommand{\blemma}{\begin{lemma}}

\renewcommand{\>}{\rightarrow}

\newcommand{\bd}{\begin{enumerate} }
\newcommand{\ed}{\end{enumerate}}

\renewcommand*\backref[1]{}
\renewcommand*\backrefalt[4]{ \ifcase #1 \or (cited on page #2) \else (cited on pages #2) \fi}

\begin{document}

\title{Maximal nilpotent complex structures}

\begin{abstract}
Let the pair $(\mathfrak{g},J)$ be a nilpotent Lie algebra $\mathfrak{g}$ (NLA for short) endowed with a nilpotent complex structure $J$. In this paper, motivated by a question in the work of Cordero, Fern\'andez, Gray and Ugarte \cite{CFGU}, we prove that $2\leq \nu(J) \leq 3$
for $(\mathfrak{g},J)$ when $\nu(\mathfrak{g})=2$, where $\nu(\mathfrak{g})$ is the step of $\mathfrak{g}$ and $\nu(J)$ is the unique smallest integer such that $\mathfrak{a}(J)_{\nu(J)}=\mathfrak{g}$ as in the \cite[Definition 1, 8]{CFGU}. When $\nu(\mathfrak{g})=3$, for arbitrary $n \geq 3$, there exists a pair $(\mathfrak{g},J)$ such that $\nu(J)=\dim_{\mathbb{C}}\mathfrak{g}=n$, for which we call the $J$ in the pair $(\mathfrak{g},J)$, satisfying $\nu(J)=\dim_{\mathbb{C}}\mathfrak{g}=n$, a maximal nilpotent (MaxN for short) complex structure. The algebraic dimension of a nilmanifold endowed with a left invariant MaxN complex structure is discussed. Furthermore, a structure theorem is proved for the pair $(\mathfrak{g},J)$, where $\nu(\mathfrak{g})=3$ and $J$ is a MaxN complex structure.
\end{abstract}

\author{Qin Gao}
\address{Qin Gao. School of Mathematics and Economics, Hubei
University of Education, Wuhan, 430205, P.R.China.} \email{gaoqin9632112369@126.com}

\author{Quanting Zhao}
\address{Quanting Zhao. School of Mathematics and Statistics \&
Hubei Key Laboratory of Mathematical Sciences, Central China Normal
University, Wuhan, 430079, P.R.China.} \email{zhaoquanting@126.com;zhaoquanting@mail.ccnu.edu.cn}
\thanks{Gao is partially supported by National Natural Science Foundations of China
with the grant No.11901176. Zhao is partially supported by National Natural Science Foundations of China with the grant No.11801205
and China Scholarship Council to Ohio State University.}

\author{Fangyang Zheng}
\address{Fangyang Zheng. School of Mathematical Sciences, Chongqing Normal University, Chongqing 401131, China}
\email{franciszheng@yahoo.com} \thanks{}


\subjclass[2010]{57T15 22E25 (Primary), 53C30 (Secondary)}
\keywords{nilpotent Lie groups, nilpotent complex structures, maximal nilpotent, algebraic dimension, the first Betti number}

\maketitle

\tableofcontents

\section{Introduction}
A nilmanifold $M$ is the compact quotient of a simply connected nilpotent Lie Group $G$ endowed with a left invariant complex structure $J$ by a discrete lattice $\Gamma$, whose study provides a rich and wide variety of examples with unusual properties both in complex geometry and symplectic geometry. We refer the readers \cite{CF,CFGU,FGV,FRR,LUV,M,Rol09Lie,Sal} for the complex structures on nilmanifolds and their Dolbeault cohomologies, \cite{A,CFP,MPPS,Nak,Rol09Geo} for their complex deformations, \cite{EFV,FG,Uga} for the existence of special Hermitian metrics, and \cite{COUV,Rol08} for the study of the Fr\"olicher spectral sequence of nilmanifolds, with more discussions of related topics and the references therein.

In this paper, we focus on a special class of left invariant complex structures on nilpotent Lie groups, namely nilpotent complex structures initiated by Cordero-Fern\'andez-Gray-Ugarte \cite[Definition 1 and 8]{CFGU}, and study a question mentioned in \cite[the part between Proposition 10 and 11]{CFGU}
\begin{question}
Let the pair $(\mathfrak{g},J)$ be a nilpotent Lie algebra $\mathfrak{g}$ of finite positive dimension, which admits a nilpotent complex structure $J$.
Then does the following inequality hold
\begin{equation}\label{ineq_main}
\nu(\mathfrak{g}) \leq \nu(J) \leq \nu(\mathfrak{g})+1,
\end{equation}
where $\nu(\mathfrak{g})$ is the step of $\mathfrak{g}$ and $\nu(J)$ is the unique smallest integer such that $\mathfrak{a}(J)_{\nu(J)}=\mathfrak{g}$,
as in \cite[Definition 1, 8, Lemma 2 and Theorem 12]{CFGU}?
\end{question}
The nilpotent Lie algebra will be abbreviated as NLA throughout the paper. It is rather clear that $\nu(J)=1$ holds for any complex structure $J$ on an abelian Lie algebra $\mathfrak{g}$, that is, a Lie algebra satisfying $\nu(\mathfrak{g})=1$, and thus \eqref{ineq_main} trivially holds when $\nu(\mathfrak{g})=1$.

Details of the related definitions above will be recalled in Section \ref{ncs_mcn} and several equivalent statements of nilpotent complex structures are proved after Cordero-Fern\'andez-Gray-Ugarte \cite[Definition 1,8 and Theroem 12,13]{CFGU}.
\begin{theorem}\label{equivalence}
Let the pair $(\mathfrak{g},J)$ be a NLA $\mathfrak{g}$ endowed with a complex structure $J$ and $t$ is a nonnegative integer.
Then the following statements are equivalent
\begin{enumerate}
\item\label{i} $\mathfrak{a}_t = \mathfrak{g}$ and $\mathfrak{a}_{t-1} \subsetneq \mathfrak{g}$.
\item\label{ii} $W^t = \mathfrak{g}^{1,0}$ and $W^{t-1} \subsetneq \mathfrak{g}^{1,0}$.
\item\label{iii} $\mathfrak{h}^t = 0$ and $\mathfrak{h}^{t-1} \neq 0$.
\end{enumerate}
where $\{\mathfrak{a}_i\}_{i\in\mathbb{N}}\footnotemark$\footnotetext{Here $\mathbb{N}$ denotes the set of nonnegative integers throughout the paper}, $\{W^i\}_{i\in\mathbb{N}}$ and $\{\mathfrak{h}^i\}_{i\in\mathbb{N}}$ follow the definitions \ref{nil_cplx}, \ref{another_nil_cplx} and \ref{annihilator}.
\end{theorem}

The existence of a finite number $t$ so that statement \eqref{i} holds is exactly the definition of the complex structure being nilpotent as in \cite[Definition 1 and 8]{CFGU}. Hence, by abuse of notations, we may call the integer $t$ in Theorem \ref{equivalence} the \emph{step of the nilpotent complex structure} $J$, denoted by $\nu(J)$, as in  definition \ref{newdef_nil_cplx}, if it wouldn't cause any ambiguity.

The upmost interest of the question above lies in the upper bound of $\nu(J)$, which is independent of the dimension of $\mathfrak{g}$,
when the step $\nu(\mathfrak{g})$ goes up higher. When $\nu(\mathfrak{g})=2$, recall the fact which is contained in \cite[Proposition 3.3]{Rol09Geo} that,
every complex structure $J$ on a 2-step NLA $\mathfrak{g}$ is necessarily nilpotent. For 2-step NLA, we have the following statement whose proof will be given in Section \ref{pfstep2}:
\begin{theorem}\label{step2}
Let $(\mathfrak{g},J)$ be a 2-step NLA $\mathfrak{g}$, which admits a complex structure $J$.
Then the following inequality holds
\begin{equation}\label{ineq_2}
2 \leq \nu(J) \leq 3.
\end{equation}
\end{theorem}
So the answer to Question \eqref{ineq_main} is affirmative when the NLA $\mathfrak{g}$ is 2-step. However, this will no longer be the case when $\nu(\mathfrak{g})=3$, a little surprisingly, where the upper bound of \eqref{ineq_main} is no longer valid, due to the example \ref{ex_max} below. Note that on $3$-step NLAs, non-nilpotent complex structures may appear, for instance $\mathfrak{h}_{19}^-$ and $\mathfrak{h}_{26}^{+}$, as proved in \cite[Theorem 8]{Uga}.

For the sake of completeness and simplicity, several notations will be delivered here before the example, with a more clear description in Section \ref{ncs_mcn}. A complex structure $J$ on a Lie algebra $\mathfrak{g}$ is an endomorphism $J:\mathfrak{g}\rightarrow \mathfrak{g}$ of the Lie algebra $\mathfrak{g}$ such that $J^2=-\1$, satisfying the Nijenhuis condition \eqref{NJH}. Let $\mathfrak{g}_\mathbb{C}$ and $\mathfrak{g}_\mathbb{C}^{\ast}$ be the complexification of the Lie algebra $\mathfrak{g}$ and its dual $\mathfrak{g}^{\ast}$ respectively, namely $\mathfrak{g} \otimes_{\mathbb{R}} \mathbb{C}$ and $\mathfrak{g}^* \otimes_{\mathbb{R}} \mathbb{C}$, where $J$ naturally extends to both spaces. The eigenspace corresponding to the eigenvalue $\sqrt{-1}$ of $J$ as an endomorphism of $\mathfrak{g}_\mathbb{C}^*$ is denoted by $\mathfrak{g}^{1,0}$, with its dual being $\mathfrak{g}_{1,0}$, which is the $\sqrt{-1}$-eigenspace of $J$ considered as an endomorphism of $\mathfrak{g}_\mathbb{C}$. The Nijenhuis condition amounts to the the integrability condition \eqref{integrability} in terms of differential forms.

\begin{example}\label{ex_max}
For any given positive integer $n\geq3$, there exists a pair $(\mathfrak{g},J)$, where $\mathfrak{g}$ is a $3$-step NLA endowed with a nilpotent complex structure $J$, such that \[\nu(J)=\dim_{\mathbb{C}}\mathfrak{g}=n.\]
The followings are two types of systems of equations, for $k \geq 2$,
\begin{align}
(\mathrm{I})\quad &\begin{cases}
d\omega^1=0,\\
d\omega^2=\omega^1 \wedge \overline{\omega}^1=-d\overline{\omega}^2,\\
d\omega^3= \sqrt{-1}(\omega^2 \wedge \overline{\omega}^1 - \omega^1 \wedge \overline{\omega}^2) = -d\overline{\omega}^3,\\
d\omega^4=\omega^1 \wedge (\omega^3+\overline{\omega}^3),\\
\cdots\cdots\\
d\omega^{2k-1}=\sqrt{-1}(\omega^{2k-2}\wedge\overline{\omega}^1-\omega^1\wedge\overline{\omega}^{2k-2})=-d\overline{\omega}^{2k-1},\\
d\omega^{2k}=\omega^1 \wedge (\omega^{2k-1}+\overline{\omega}^{2k-1}).\\
\end{cases} \label{type_1} \\
(\mathrm{II})\quad&\begin{cases}
d\omega^1=0,\\
d\omega^2=\omega^1 \wedge \overline{\omega}^1=-d\overline{\omega}^2,\\
d\omega^3= \omega^1 \wedge (\omega^2 + \overline{\omega}^2),\\
d\omega^4=\sqrt{-1}(\omega^3 \wedge \overline{\omega}^1 - \omega^1 \wedge \overline{\omega}^3) = -d\overline{\omega}^4,\\
\cdots\cdots\\
d\omega^{2k-1}=\omega^1 \wedge (\omega^{2k-2}+\overline{\omega}^{2k-2}),\\
d\omega^{2k}=\sqrt{-1}(\omega^{2k-1}\wedge\overline{\omega}^1-\omega^1\wedge\overline{\omega}^{2k-1})=-d\overline{\omega}^{2k}.\\
\end{cases}  \label{type_2}
\end{align}
When $n=2k$, the structure equations of the two pairs $(\mathfrak{k}_{n}^{\mathrm{I}},J_{n}^{\mathrm{I}})$ and $(\mathfrak{k}_{n}^{\mathrm{II}},J_{n}^{\mathrm{II}})$ are defined by the $2k$ equations in \eqref{type_1} and \eqref{type_2} respectively, with $\{\omega^{i}\}_{i=1}^{2n}$ being bases of the spaces $(\mathfrak{k}_n^{\mathrm{I}})^{1,0}$ and $(\mathfrak{k}_n^{\mathrm{II}})^{1,0}$ accordingly; when $n=2k-1$, the structure equations of the two pairs $(\mathfrak{k}_{n}^{\mathrm{I}},J_{n}^{\mathrm{I}})$ and $(\mathfrak{k}_{n}^{\mathrm{II}},J_{n}^{\mathrm{II}})$ are referred to the first $2k-1$ equations in \eqref{type_1} and \eqref{type_2} respectively, with bases of $(\mathfrak{k}_n^{\mathrm{I}})^{1,0}$ and $(\mathfrak{k}_n^{\mathrm{II}})^{1,0}$ denoted by $\{\omega^i\}_{i=1}^{2k-1}$ accordingly.

It is easy to verify that $d^2=0$ and the integrability condition \eqref{integrability} hold for the structure equations of $(\mathfrak{k}_{n}^{\mathrm{I}},J_{n}^{\mathrm{I}})$ and $(\mathfrak{k}_{n}^{\mathrm{II}},J_{n}^{\mathrm{II}})$ when $n\geq3$, which implies that $J_n^{\mathrm{I}}$ and $J_n^{\mathrm{II}}$ are indeed complex structures on the corresponding Lie algebras $\mathfrak{k}_n^{\mathrm{I}}$ and $\mathfrak{k}_n^{\mathrm{II}}$.

For $n \geq 4$, $(\mathfrak{k}_{n}^{\mathrm{I}},J_{n}^{\mathrm{I}})$ and $(\mathfrak{k}_{n}^{\mathrm{II}},J_{n}^{\mathrm{II}})$ are two types of nilpotent complex structures such that $\nu(J_{n}^{\mathrm{I}})=\nu(J_{n}^{\mathrm{II}})=\dim_{\mathbb{C}}\mathfrak{k}_{n}^{\mathrm{I}}=\dim_{\mathbb{C}}\mathfrak{k}_{n}^{\mathrm{II}}=n$ and $\nu(\mathfrak{k}_n^{\mathrm{I}})=\nu(\mathfrak{k}_n^{\mathrm{II}})=3$. For $n=3$, $(\mathfrak{k}_{3}^{\mathrm{I}},J_{3}^{\mathrm{I}})$ still satisfies that $\nu(J_{3}^{\mathrm{I}})=\dim_{\mathbb{C}}\mathfrak{k}_{3}^{\mathrm{I}}=3$ and $\nu(\mathfrak{k}_{3}^{\mathrm{I}})=3$,
however, $(\mathfrak{k}_{3}^{\mathrm{II}},J_{3}^{\mathrm{II}})$ fails to enjoy the property $\nu(\mathfrak{k}_{3}^{\mathrm{II}})=3$, where the equality $\nu(\mathfrak{k}_{3}^{\mathrm{II}})=2$ holds instead, although $\nu(J_{3}^{\mathrm{II}})=\dim_{\mathbb{C}}\mathfrak{k}_{3}^{\mathrm{II}}=3$. Details of the example will be shown in Section \ref{ex}.
\end{example}
The example actually confirms that the following well known inequality (cf. \cite[Proposition 10]{CFGU}) for a NLA $\mathfrak{g}$
endowed a nilpotent complex structure $J$,
\[\nu(\mathfrak{g}) \leq \nu(J) \leq \dim_{\mathbb{C}}\mathfrak{g},\]
is sharp for $\nu(\mathfrak{g})=3$ and the upper bound can be arbitrarily large (the sharpness on the left is also well known as in \cite[Corollary 5 and 7]{CFGU}), which may shed light on the sharpness for $\nu(\mathfrak{g})>3$.

Motivated by Example \ref{ex_max} above, we propose the following definition:
\begin{definition}\label{mcn}
Let the pair $(\mathfrak{g},J)$ be a NLA $\mathfrak{g}$ endowed with a complex structure $J$. The complex structure $J$ is said to be \emph{maximal nilpotent}, or MaxN for short, if $J$ is a nilpotent complex structure satisfying
\[ \nu(J)=\dim_{\mathbb{C}}\mathfrak{g}. \]
At this time, the pair $(\mathfrak{g},J)$ is also said to be \emph{maximal nilpotent}, or MaxN for short.
\end{definition}

As a direct corollary to Theorem \ref{equivalence}, we have

\begin{corollary}\label{equ_mcn}
Let the pair $(\mathfrak{g},J)$ be a (possibly not nilpotent) Lie algebra $\mathfrak{g}$ endowed with a complex structure $J$.
Then $(\mathfrak{g},J)$ is MaxN if and and only if $\mathfrak{g}^{1,0}$ admits a basis $\{\omega^i\}_{i=1}^n$, satisfying
\begin{equation}\label{streq_nil_cplx_mcn}
d\omega^k = \sum_{i<j<k}A_{ij}^k\omega^i \wedge \omega^j +\sum_{i,j<k}B_{ij}^k\omega^i\wedge\overline{\omega}^j,\quad 1 \leq k \leq n,
\end{equation}
where the coefficients $\{A_{i,k-1}^k\}_{i<k-1}$, $\{B_{k-1,j}^k\}_{j=1}^{k-1}$ and $\{B_{i,k-1}^k\}_{i<k-1}$ are not all zero. The MaxN condition implies that the inclusions \eqref{inclusion} becomes equalities.
\end{corollary}

\begin{remark}\label{d-indept}
It is clear that $\{d\omega^k\}_{k=2}^n$ is $\mathbb{C}$-linearly independent for the basis $\{\omega^i\}_{i=1}^n$ above.
\end{remark}

\begin{corollary}\label{mcn_hierarchy}
Let the pair $(\mathfrak{g},J)$ be MaxN. Then $(\mathfrak{a}_k,J)$ and $(\mathfrak{g}/\mathfrak{a}_{\nu(J)-k}, J)$ both are MaxN and $\mathfrak{a}_{\nu(J)-k} = \mathfrak{h}_k$ for $0 \leq k \leq \nu(J)$.
\end{corollary}

\begin{corollary}\label{mcn_agdim}
Let $M=(G/\Gamma,J)$ be a nilmanifold, where $G$ is a simply connected nilpotent Lie group endowed with a left invariant MaxN complex structure $J$ and $\Gamma$ is a cocompact lattice.
Then the algebraic dimension $a(M) = 1$ and the Albanese map $\Psi: M \rightarrow \mathrm{Alb}(M)$ is a smooth fibration.
\end{corollary}

The study of the MaxN pair $(\mathfrak{g},J)$ will be focused on the cases when $\nu(\mathfrak{g})\geq3$, due to the heavy restriction on $\nu(J)$ when $\nu(\mathfrak{g})=2$ as in Theorem \ref{step2}. The main goal of this paper is to establish the structure theorem for MaxN complex structures. The proofs will be given in Section \ref{structure}.

The following theorem is a starting point to study the structure of MaxN complex structures,
which can be viewed as a normal expression for Corollary \ref{equ_mcn}.
\begin{theorem}\label{equ_mcn_refined}
Let $(\mathfrak{g},J)$ be MaxN with the structure equation given in Corollary \ref{equ_mcn}. Then the coframe $\{\omega^k\}_{k=1}^n$ can be chosen such that $\{d\omega^k\}_{k=1}^n$ satisfies \eqref{streq_nil_cplx_mcn} while each $\omega^k$ satisfies one of the following two conditions:
\begin{enumerate}
\item\label{dpt} $d\omega^k + d \overline{\omega}^k =0$,
\item\label{indpt} $\forall$ $a,b\in\mathbb{C}$, it holds that
\[ a \,d\omega^k + b\,d\overline{\omega}^k \in \mathrm{span}_{\mathbb{C}}\{d\omega^{k-1},d\overline{\omega}^{k-1},\cdots,d\omega^2,d\overline{\omega}^2,d\omega^1,d\overline{\omega}^1\} \Leftrightarrow a=b=0.\]
\end{enumerate}
\end{theorem}

Theorem \ref{equ_mcn_refined} motivates the following terminology, which will be helpful in the study of the structure theorem in the section \ref{structure}.
\begin{definition}\label{DPT}
Let $(\mathfrak{g},J)$ be MaxN. A coframe of $\mathfrak{g}^{1,0}$, satisfying the property in
Theorem \ref{equ_mcn_refined}, is called an \emph{admissible coframe}. Define the following for an admissible coframe $\{\omega^i\}_{i=1}^n$ of the MaxN pair $(\mathfrak{g},J)$
\[\mathrm{Dpt}(J) \stackrel{\triangle}{=}\{ 1\leq i \leq n \, \big |\,  \omega^i\ \text{satisfies the condition}\ \eqref{dpt}\ \text{in Theorem}\  \ref{equ_mcn_refined}\}.\]
\end{definition}

It turns out that the index set $\mathrm{Dpt}(J)$ is independent of the choice of a specific admissible coframe due to the lemma below, so $\mathrm{Dpt}(J)$ is an invariant of the MaxN complex structure.

\begin{lemma}\label{invariance_dpt}
Let $\{\omega^k\}_{k=1}^n$ and $\{\tau^k\}_{k=1}^n$ be two admissible coframes of a MaxN pair $(\mathfrak{g},J)$. Write
\[\omega^k=\sum_{j=1}^n a^k_j\tau^j, \ \ 1\leq k\leq n,\]
then $a^k_j=0$ for $j\geq k+1$ and $a^k_k\neq0$. Furthermore, for each $1\leq k \leq n$,
\[d\omega^k+d\overline{\omega}^k=0 \, \Longleftrightarrow \, d\tau^k+d\overline{\tau}^k=0.\]
\end{lemma}

The following structure theorem for a MaxN pair $(\mathfrak{g},J)$ with $\nu(\mathfrak{g})=3$ was inspired by
the obvious pattern exhibited in Example \ref{ex_max}, where in the structure equations of $(\mathfrak{k}^{\mathrm{I}}_{n},J^{\mathrm{I}}_n)$ and $(\mathfrak{k}^{\mathrm{II}}_{n},J^{\mathrm{II}}_n)$, the linear dependence and independence of $\{d\omega^i,d\overline{\omega}^i\}$ interlace for each $i\geq3$. This turns out to be almost a general phenomenon, as we have the following result in terms of $\mathrm{Dpt}(J)$:

\begin{theorem}\label{strthmn-2}
Let $(\mathfrak{g},J)$ be MaxN, where $\nu(\mathfrak{g})=3$ and $\dim_{\mathbb{C}}\mathfrak{g}=n \geq 5$. Then
\begin{enumerate}
\item\label{n-2_dept} If $n-2 \in \mathrm{Dpt}(J)$, then for each $3 \leq k \leq n-2 $, $k \in \mathrm{Dpt}(J)$ $\, \Longleftrightarrow \,$  $k \equiv n-2 \mod 2$.
\item\label{n-2_indept} If $n-2 \notin \mathrm{Dpt}(J)$, then for each $3 \leq k \leq n-2 $,  $k \in \mathrm{Dpt}(J)$ $\, \Longleftrightarrow \,$  $k \not \equiv n-2 \mod 2$.
\end{enumerate}
There exists an admissible coframe $\{\omega^k\}_{k=1}^n$ such that \[\begin{cases} d\omega^1 =0,\\ d\omega^2 = \omega^1 \wedge \overline{\omega}^1,\end{cases}\]
and for each $3 \leq k \leq n-2$,
\[\begin{array}{rcl}
k \notin \mathrm{Dpt}(J)&\Longleftrightarrow& \omega^k \in V^2_{\mathbb{C}},\\
k \in \mathrm{Dpt}(J)&\Longleftrightarrow& \omega^k \in V^3_{\mathbb{C}}\setminus V^2_{\mathbb{C}}.
\end{array}\]
Furthermore, for $a_k,b_k\in \mathbb{C}$, it holds that
\begin{equation}\label{crnV2n-2}
\sum_{k=1}^{n-2}(a_k \omega^k + b_k \overline{\omega}^k ) \, \in \,V^2_{\mathbb{C}} \ \Longleftrightarrow \ a_k = b_k, \ \text{for}\ k \in \mathrm{Dpt}(J)\ \text{and}\ 3\leq k \leq n-2.
\end{equation}
\end{theorem}
Here $\{V^i\}_{i\in\mathbb{N}}$, defined in Section \ref{ncs_mcn}, is the corresponding annihilator of $\{\mathfrak{g}^i\}_{i\in\mathbb{N}}$. The latter is the descending central series of the nilpotent Lie algebra $\mathfrak{g}$. The complexified space $V^i \otimes_{\mathbb{R}} \mathbb{C}$ is denoted by $V^i_{\mathbb{C}}$. The symbol $ V^3_{\mathbb{C}}\setminus V^2_{\mathbb{C}}$ means the complement of $V^2_{\mathbb{C}}$ in $V^3_{\mathbb{C}}$ and the notations will be applied throughout the paper.

As a consequence of the structure theorem above, we have
\begin{corollary}\label{firstB-MaxN}
Let $(\mathfrak{g},J)$ be MaxN with $\nu(\mathfrak{g})=3$ and $\dim_{\mathbb{C}}\mathfrak{g}=n \geq 5$. Then it follows that
\[  \lfloor \frac{3n-4}{2}\rfloor \leq \dim_{\mathbb{R}} \mathfrak{g}^1 \leq \lfloor \frac{3n-1}{2}\rfloor ,
\quad \ \lfloor \frac{n-2}{2}\rfloor \leq \dim_{\mathbb{R}} \mathfrak{g}^2 \leq  \lfloor \frac{n+3}{2}\rfloor,\]
where $\lfloor a\rfloor$ is the maximal integer which does not exceed the number $a$.
Now let $M=(G/\Gamma,J)$ be a nilmanifold, where $G$ is a simply connected nilpotent Lie group endowed with a left invariant MaxN complex structure $J$ and $\Gamma$ is a cocompact lattice. If the Lie algebra $\mathfrak{g}$ of $G$ satisfies $\nu(\mathfrak{g})=3$ and $\dim_{\mathbb{C}}\mathfrak{g}=n \geq 5$, then the first Betti number $b_1(M)$ satisfies
\[ \lfloor \frac{n+2}{2}\rfloor  \leq b_1(M) \leq  \lfloor \frac{n+5}{2}\rfloor .\]
\end{corollary}

\vspace{0.3cm}

\section{Nilpotent and maximal nilpotent complex structures}\label{ncs_mcn}
Let $\mathfrak{g}$ be a nilpotent Lie algebra (NLA for short) of finite positive dimension, that is,
from the very definition, both the descending $\{\mathfrak{g}^i\}_{i\in\mathbb{N}}$ and ascending $\{\mathfrak{g}_i\}_{i\in\mathbb{N}}$
central series give a nice filtration for $\mathfrak{g}$ of the same finite length. The \emph{descending central series} $\{\mathfrak{g}^i\}_{i\in\mathbb{N}}$ is given by
\[\mathfrak{g}^0=\mathfrak{g},\quad \mathfrak{g}^k=[\mathfrak{g}^{k-1},\mathfrak{g}]\quad k\geq 1,\]
while the \emph{ascending central series} $\{\mathfrak{g}_i\}_{i\in\mathbb{N}}$ is constructed by
\[ \mathfrak{g}_0=0,\quad \mathfrak{g}_k=\{X \in \mathfrak{g} \big| [X,\mathfrak{g}]\subseteq \mathfrak{g}_{k-1}\}\quad k\geq 1, \]
where $\{\mathfrak{g}^i\}_{i\in\mathbb{N}}$ and $\{\mathfrak{g}_i\}_{i\in\mathbb{N}}$ are naturally the ideals of $\mathfrak{g}$.
Then the Lie algebra $\mathfrak{g}$ is said to be \emph{$s$-step nilpotent} if $\mathfrak{g}^s=0$ and $\mathfrak{g}^{s-1}\neq0$ or equivalently
$\mathfrak{g}_s=\mathfrak{g}$ and $\mathfrak{g}_{s-1}\subsetneq\mathfrak{g}$, and the symbol $\nu(\mathfrak{g})$ will be used here and afterwards as a substitute for the step $s$ (cf. \cite[Section 2]{CFGU} and \cite[Section 1]{Rol09Geo} for more details). The first pieces $\mathfrak{g}^1$ and $\mathfrak{g}_1$ usually are called the \emph{commutator} and the \emph{center} of the Lie algebra $\mathfrak{g}$ respectively.

The nilpotency of Lie algebras can also be interpreted in terms of the differential forms of the dual cases, which has already been done in \cite[Section 1]{Sal}, \cite{CFGU,Rol09Geo,Uga} and in more general situations \cite[Section 2]{Nak}, namely solvable Lie algebras which admit the Chevalley decompositions. By considering
the annihilators of $\{\mathfrak{g}^i\}_{i\in\mathbb{N}}$ accordingly, we have an \emph{ascending series of subspaces} $\{V^i\}_{i\in\mathbb{N}}$ of $\mathfrak{g}^*$, given by
\[ V^0=0,\quad V^{k}=\{\alpha\in\mathfrak{g}^* \big| d\alpha \in \bigwedge^2V^{k-1}\}\quad k\geq 1,\]
where it is easy but of paramount importance to see the first piece $V^1$ is nothing but the space of $d$-closed differential forms and
the exterior operator $d$ is also known as the Chevalley-Eilenberg differential on $\wedge^{*}\mathfrak{g}^{*}$. Then the $\nu(\mathfrak{g})$-step NLA $\mathfrak{g}$ can be alternatively defined by $V^{\nu(\mathfrak{g})}=\mathfrak{g}^*$ and
$V^{\nu(\mathfrak{g})-1} \subsetneq \mathfrak{g}^*$.

Hence, the Lie algebra $\mathfrak{g}$ is nilpotent, if the step $\nu(\mathfrak{g})$ is suppressed, is equivalent to the existence of a basis $\{\vartheta_i\}_{i=1}^m$ of $\mathfrak{g}$ such that
\begin{equation}\label{streq_vec}
[\vartheta_i,\vartheta_j]=\sum_{i<j<k}c_{ijk}\vartheta_k,
\end{equation}
and the Jacobi equality holds, or the existence of a basis $\{\phi^i\}_{i=1}^m$ of $\mathfrak{g}^*$ such that
\begin{equation}\label{streq_form}
d\phi^k=\sum_{i<j<k}\tilde{c}_{ijk}\phi^i \wedge \phi^j,
\end{equation}
and it satisfies $d^2=0$, which stems from tacitly considering the bases of $\{\mathfrak{g}^i\}_{i=0}^{\nu(\mathfrak{g})}$ and $\{V^i\}_{i=0}^{\nu(\mathfrak{g})}$. Clearly $c_{ijk}=-\tilde{c}_{ijk}$ if $\{\vartheta_i\}_{i=1}^m$ and $\{\phi^i\}_{i=1}^m$ happen to be the
dual bases, due to the well known equality \[d\alpha(X,Y)=-\alpha([X,Y]),\quad \alpha\in \mathfrak{g}^*,\ X,Y\in\mathfrak{g}.\]
Equation \eqref{streq_vec} or \eqref{streq_form} is usually called the structure equation
of the NLA $\mathfrak{g}$.

The \emph{complex structure} on the Lie algebra $\mathfrak{g}$ is
given by an endomorphism $J:\mathfrak{g}\rightarrow \mathfrak{g}$ of
the Lie algebra $\mathfrak{g}$ such that $J^2=-\1$, satisfying the
Nijenhuis condition
\begin{equation}\label{NJH}
[JX,JY]=J[JX,Y]+J[X,JY]+[X,Y],\quad X,Y\in \mathfrak{g}.
\end{equation}
Let $\mathfrak{g}_\mathbb{C}$ be the complexification of the Lie algebra $\mathfrak{g}$, namely
$\mathfrak{g} \otimes_{\mathbb{R}} \mathbb{C}$, and $\mathfrak{g}_\mathbb{C}^*$ be its dual, where $J$ extends to both spaces.
We denote by $\mathfrak{g}^{1,0}$ and $\mathfrak{g}^{0,1}$ the eigenspaces corresponding to the eigenvalues $\sqrt{-1}$ and $-\sqrt{-1}$ of $J$
as an endomorphism of $\mathfrak{g}_\mathbb{C}^*$ respectively. The dual space of $\mathfrak{g}^{1,0}$ is denoted by $\mathfrak{g}_{1,0}$, which
is the $\sqrt{-1}$-eigenspace of $J$ considered as an endomorphism of $\mathfrak{g}_\mathbb{C}$. The decomposition
$$\mathfrak{g}_\mathbb{C}^*=\mathfrak{g}^{1,0}\oplus\mathfrak{g}^{0,1},\quad \mathfrak{g}^{0,1}=\overline{\mathfrak{g}^{1,0}},$$
gives rise to a natural bigraduation on the complexified exterior
algebra
$$\bigwedge^*\mathfrak{g}_\mathbb{C}^*=\bigoplus_{p,q}\bigwedge^{p,q}
\mathfrak{g}^* =\bigoplus_{p,q}\bigwedge^p
\mathfrak{g}^{1,0}\wedge \bigwedge^q\mathfrak{g}^{0,1}.$$
The operator $d$ on $\mathfrak{g}^*$ naturally extends to the complexified exterior algebra $\mathfrak{g}_{\mathbb{C}}^*$, i.e.,
$d:\bigwedge^*\mathfrak{g}_\mathbb{C}^*\rightarrow\bigwedge^{*+1}\mathfrak{g}_\mathbb{C}^*.$
It is well known that the endomorphism $J$ is a complex structure if
and only if \begin{equation}\label{integrability}
d\mathfrak{g}^{1,0}\subseteq\bigwedge^{2,0}
\mathfrak{g}^*\oplus\bigwedge^{1,1} \mathfrak{g}^*.\end{equation}
Espeically for NLAs $\mathfrak{g}$, Salamon \cite[Theorem 1.3]{Sal} proves that
$J$ is a complex structure on $\mathfrak{g}$ if
and only if $\mathfrak{g}^{1,0}$ admits a basis $\{\omega^i\}_{i=1}^n$
such that $d\omega^1=0$ and
\[d\omega^i\in \mathcal{I}(\omega^1,\cdots,\omega^{i-1}),\quad 2\leq i \leq n,\]
where $\mathcal{I}(\omega^1,\cdots,\omega^{i-1})$ is the ideal in
$\wedge^*\mathfrak{g}_\mathbb{C}^*$ generated by
$\{\omega^1,\cdots,\omega^{i-1}\}$.

The main concern of this paper will be focused on the \emph{nilpotent complex structures} on the nilpotent Lie groups,
proposed by Cordero-Fern\'andez-Gray-Ugarte \cite{CFGU}.
\begin{definition}\label{nil_cplx}\cite[Definition 1 and 8]{CFGU}
Let the pair $(\mathfrak{g},J)$ be a NLA $\mathfrak{g}$ endowed with a complex structure $J$.
The \emph{ascending series} $\{\mathfrak{a}_k\}_{k\in\mathbb{N}}$ \emph{compatible with the complex structure} $J$ is defined by
\[\mathfrak{a}_0=0,\quad \mathfrak{a}_{k}=\{X\in\mathfrak{g} \big| [X,\mathfrak{g}] \subseteq \mathfrak{a}_{k-1},\ [JX,\mathfrak{g}]\subseteq \mathfrak{a}_{k-1}\}\quad k \geq 1.\]
The complex structure $J$ is said to be \emph{nilpotent} if some piece $\mathfrak{a}_t$ can reach the whole Lie algebra $\mathfrak{g}$.
Apparently, $\{\mathfrak{a}_k\}_{k\in\mathbb{N}}$ are $J$-invariant ideals of $\mathfrak{g}$ and $\mathfrak{a}_{k} \subseteq \mathfrak{g}_k$. From \cite[Lemma 2 and Proposition 10]{CFGU}, the smallest integer $t$ such that $\mathfrak{a}_t=\mathfrak{g}$ is unique, satisfying $\nu(\mathfrak{g}) \leq t\leq \dim_{\mathbb{C}}\mathfrak{g}$, denoted by $\nu(J)$.
\end{definition}
Cordero-Fern\'andez-Gray-Ugarte prove in \cite[Theorem 12 and 13]{CFGU} that, for a (possibly not nilpotent) Lie algebra $\mathfrak{g}$ endowed with a complex structure $J$, the Lie algebra $\mathfrak{g}$ is nilpotent with $J$ being a nilpotent complex structure on it if and only if $\mathfrak{g}^{1,0}$ admits a basis $\{\omega^i\}_{i=1}^n$, satisfying
\begin{equation}\label{streq_nil_cplx}
d\omega^k = \sum_{i<j<k}A_{ij}^k\omega^i \wedge \omega^j +\sum_{i,j<k}B_{ij}^k\omega^i\wedge\overline{\omega}^j,\quad 1 \leq k \leq n,
\end{equation}
since the equation \eqref{streq_nil_cplx} necessarily forces the underlying Lie algebra $\mathfrak{g}$ to be nilpotent by tacitly providing a basis of $\mathfrak{g}^*$, satisfying \eqref{streq_form}.

In order to reveal $\nu(J)$ via differential forms, we will introduce the following definition,
which is inspired by and essentially already hidden in \cite[the proof of Theorem 12]{CFGU}, by mimicking the construction of $\{V^k\}_{k\in\mathbb{N}}$ above.
\begin{definition}\label{another_nil_cplx}
Let the pair $(\mathfrak{g},J)$ be a NLA $\mathfrak{g}$ endowed with a complex structure $J$.
Define an \emph{ascending series of subspaces} $\{W^k\}_{k\in\mathbb{N}}$ of $\mathfrak{g}^{1,0}$ as follows,
\[W^0=0,\quad W^k=\{ \omega \in \mathfrak{g}^{1,0} \, \big| \ d \omega \in (W^{k-1} \wedge W^{k-1}) \oplus (W^{k-1} \wedge \overline{W}^{k-1})\}\quad k\geq 1.\]
The first piece $W^1$ is called the space of left invariant holomorphic differentials in \cite{FGV}, which always has positive dimension due to \cite[Theorem 1.3]{Sal}. If there is an integer $p$ such that $W^{p} = \mathfrak{g}^{1,0}$ and $W^{p-1} \subsetneq \mathfrak{g}^{1,0}$,
we denote by $\mu(J)$ the (unique) $p$. It is obvious that $\mu(J)>0$ and $W^i \subsetneq W^j$ for $0 \leq i<j \leq \mu(J)$,
if $\mu(J)$ exists.
\end{definition}

\begin{remark}
It is obvious that $W^k \oplus \overline{W}^k \subseteq V^k_{\mathbb{C}}$ for $k \in \mathbb{N} $ by definition, where $V^k_{\mathbb{C}} = V^k \otimes_{\mathbb{R}} \mathbb{C}$. However, $V^k$ is not necessarily $J$-invariant.
\end{remark}

A close observation, almost equivalent to \cite[Theorem 12]{CFGU} and the equation \eqref{streq_nil_cplx}, is the following
\begin{observation}\label{existence_mu}
Let the pair $(\mathfrak{g},J)$ be a NLA $\mathfrak{g}$ endowed with a complex structure $J$.
Then $J$ is a nilpotent complex structure if and only if $\mu(J)$ exists.
\end{observation}

\begin{proof}
If $J$ is a nilpotent complex structure, the equation \eqref{streq_nil_cplx} implies that $\omega^k \in W^k$ for $1 \leq k \leq n$, thus $\mu(J)$ exists. Conversely, the existence of $\mu(J)$ implies that $\mathfrak{g}^{1,0}$ admits a basis $\{\omega^k\}_{k=1}^n$, which satisfies \eqref{streq_nil_cplx}, by the very definition \ref{another_nil_cplx}.
\end{proof}

As in \cite[the proof of Theorem 12]{CFGU}, the nilpotency of the complex structure $J$ implies that the piece $\mathfrak{a}_{\nu(J)}$ in the ascending series $\{\mathfrak{a}_k\}_{k\in\mathbb{N}}$ reaches the whole Lie algebra $\mathfrak{g}$. It yields that the following sequence of quotient Lie algebras and homomorphisms
\begin{equation}\label{seq_qu}
\mathfrak{g} \> \mathfrak{g}/\mathfrak{a}_1 \> \cdots \> \mathfrak{g}/\mathfrak{a}_{k-1} \stackrel{\pi_k}{\>} \mathfrak{g}/\mathfrak{a}_{k}
\> \cdots \> \mathfrak{g}/\mathfrak{a}_{\nu(J)-1} \> 0,
\end{equation}
where $\pi_k$ is surjective and $J$ descends to the quotient Lie algebra $\mathfrak{g}/\mathfrak{a}_{k}$ for each $k$. By considering the dual of the sequence \eqref{seq_qu}, we have another sequence
\begin{equation}\label{seq_in}
0 \> (\mathfrak{g}/\mathfrak{a}_{\nu(J)-1})^* \> \cdots \> (\mathfrak{g}/\mathfrak{a}_{\nu(J)-k+1})^* \stackrel{{\rho}_k}{\>} (\mathfrak{g}/\mathfrak{a}_{\nu(J)-k})^* \> \cdots \> (\mathfrak{g}/\mathfrak{a}_{1})^* \> \mathfrak{g}^*,
\end{equation}
where $\rho_k$ is injective, and $\rho_{k}$ is the dual mapping of $\pi_{\nu(J)-k+1}$ for each $k$.
The purpose of Definition \ref{another_nil_cplx} above is to show that

\begin{theorem}\label{nu=mu}
Let the pair $(\mathfrak{g},J)$ be a NLA $\mathfrak{g}$ endowed with a nilpotent complex structure $J$.
Then the inclusions
\begin{equation}\label{inclusion}
(\mathfrak{g}/\mathfrak{a}_{\nu(J)-k})^*_{\mathbb{C}} \subseteq W^k \oplus \overline{W}^{k},\quad 0 \leq k \leq \nu(J),
\end{equation}
and the exclusions
\begin{equation}\label{exclusion}
(\mathfrak{g}/\mathfrak{a}_{\nu(J)-k-1})^*_{\mathbb{C}}
\nsubseteq W^k \oplus \overline{W}^{k},\quad 0 \leq k \leq \nu(J)-1,
\end{equation}
hold, where $(\mathfrak{g}/\mathfrak{a}_{\nu(J)-k})^*_{\mathbb{C}}=(\mathfrak{g}/\mathfrak{a}_{\nu(J)-k})^* \otimes_{\mathbb{R}}\mathbb{C}$ and $(\mathfrak{g}/\mathfrak{a}_{\nu(J)-k-1})^*_{\mathbb{C}}$ is similarly defined.
Hence, it follows that $\nu(J)=\mu(J)$.
\end{theorem}

The inclusions \eqref{inclusion} already appeared implicitly in \cite[the proof of Theorem 12]{CFGU}.

\begin{proof}
It is clear that $\nu(J)=\mu(J)$ if \eqref{inclusion} and \eqref{exclusion} are established, since
\[\mathfrak{g}^*_{\mathbb{C}} \subseteq W^{\nu(J)} \oplus \overline{W}^{\nu(J)}\quad \text{and}\quad
\mathfrak{g}^*_{\mathbb{C}} \nsubseteq W^{\nu(J)-1} \oplus \overline{W}^{\nu(J)-1}.\]
To prove \eqref{inclusion} and \eqref{exclusion}, we follow the notations and ideas in \cite[the proof in the theroem 12]{CFGU}.
For simplicity, we will write $\nu(J)$ as $\nu$ in the following. Denote by $n_i$ for $1\leq i \leq \nu\!-\!1$ and $n$ the complex dimensions $\dim_{\mathbb{C}}\mathfrak{a}_i$
and $\dim_{\mathbb{C}}\mathfrak{g}$ respectively.
Consider a basis $\{ X_i,\overline{X}_i \}_{i=1}^{n-n_{\nu\!-\!1}}$ of $(\mathfrak{g}/\mathfrak{a}_{\nu\!-\!1}
)_{\mathbb{C}}=\mathfrak{g}/\mathfrak{a}_{\nu\!-\!1} \otimes_{\mathbb{R}}\mathbb{C}$, and denote its dual basis in
$(\mathfrak{g}/\mathfrak{a}_{\nu\!-\!1})^*_{\mathbb{C}}$  by
$\{ \omega^i,\overline{\omega}^i \}_{i=1}^{n-n_{\nu\!-\!1}}$, where $X_i \in \mathfrak{g}_{1,0}$ and $\omega_i \in \mathfrak{g}^{1,0}$.
The Lie algebra $\mathfrak{g}/\mathfrak{a}_{\nu\!-\!1}$ is abelian, as $[\mathfrak{g},\mathfrak{g}] \subseteq \mathfrak{a}_{\nu\!-\!1}$
by the very definition of $\mathfrak{a}_{\nu\!-\!1}$. It follows that
\[ d \omega^i = 0,\quad 1\leq i \leq n-n_{\nu\!-\!1}, \]
which implies the $k=1$ case of \eqref{inclusion}.

For the  $k=1$ case of \eqref{exclusion}, let us extend the basis $\{ X_i,\overline{X}_i \}_{i=1}^{n-n_{\nu\!-\!1}}$ of $(\mathfrak{g}/\mathfrak{a}_{\nu\!-\!1})_{\mathbb{C}}$ to the one $\{ X_i,\overline{X}_i \}_{i=1}^{n-n_{\nu\!-\!2}}$ of $(\mathfrak{g}/\mathfrak{a}_{\nu\!-\!2})_{\mathbb{C}}$ such that $\{ X_i,\overline{X}_i \}_{i=n-n_{\nu\!-\!1}+1}^{n-n_{\nu\!-\!2}}$ is a basis of $(\mathfrak{a}_{\nu\!-\!1}/\mathfrak{a}_{\nu\!-\!2})_{\mathbb{C}}$, with the dual basis in $(\mathfrak{g}/\mathfrak{a}_{\nu\!-\!2})^*_{\mathbb{C}}$ denoted by $\{ \omega^i,\overline{\omega}^i \}_{i=1}^{n-n_{\nu\!-\!2}}$,
since
\[\mathfrak{g}/\mathfrak{a}_{\nu\!-\!1} \cong \mathfrak{g}/\mathfrak{a}_{\nu\!-\!2} \Big/ \mathfrak{a}_{\nu\!-\!1}/\mathfrak{a}_{\nu\!-\!2}.\]
It yields that $[\mathfrak{a}_{\nu\!-\!1}/\mathfrak{a}_{\nu\!-\!2},\,\mathfrak{g}/\mathfrak{a}_{\nu\!-\!2}]=0$ as $[\mathfrak{a}_{\nu\!-\!1},\mathfrak{g}] \subseteq \mathfrak{a}_{\nu\!-\!2}$. For each $1 \leq i \leq n-n_{\nu\!-\!1}$, we have $[X_i, \mathfrak{g}/\mathfrak{a}_{\nu\!-\!2}]\neq 0$ in $(\mathfrak{g}/\mathfrak{a}_{\nu\!-\!2})_{\mathbb{C}}$, as otherwise
some $X_i$ would belong to $(\mathfrak{a}_{\nu\!-\!1}/ \mathfrak{a}_{\nu\!-\!2})_{\mathbb{C}}$, which is not the case.
So we have
\begin{enumerate}
\item for $n-n_{\nu\!-\!1}+1 \leq k \leq n-n_{\nu\!-\!2}$, $1 \leq \ell \leq n-n_{\nu\!-\!2}$, the followings hold
\[[X_k,X_{\ell}]=0, \quad [X_k, \overline{X}_{\ell}]=0,\]
where both are considered to be equalities in $(\mathfrak{g}/\mathfrak{a}_{\nu\!-\!2})_{\mathbb{C}}$,
\item for any $1 \leq i \leq n-n_{\nu\!-\!1}$, there exists some $1 \leq j \leq n-n_{\nu\!-\!1}$ such that
\[ \text{either}\quad [X_i,X_j] \neq 0\quad \text{or}\quad [X_i,\overline{X}_j] \neq 0\]
as equalities in $(\mathfrak{a}_{\nu\!-\!1}/\mathfrak{a}_{\nu\!-\!2})_{\mathbb{C}}$.
\end{enumerate}
Reflecting these facts to the dual basis, we obtain that, for $n-n_{\nu\!-\!1}+1 \leq k \leq n-n_{\nu\!-\!2}$,
\[d \omega^k = \sum_{1\leq i < j \leq n-n_{\nu(J)-1}}A_{ij}^k\omega^i\wedge\omega^j+\sum_{1\leq i,j \leq n-n_{\nu(J)-1}}B_{ij}^k\omega^i\wedge\overline{\omega}^j,\]
and there exists at least one $k$ such that the coefficients $A_{ij}^k$ and $B_{ij}^k$ of the $d\omega^k$ don't all vanish, which implies that
$\omega^k \in W^2$ but $\omega^k \notin W^1$, namely $\omega^k \in W^2 \setminus W^1$. Therefore the case $k=2$ of \eqref{inclusion} and the case $k=1$ of \eqref{exclusion} are established,
and it can be assumed additionally that there exists some $q$, satisfying $n-n_{\nu\!-\!1} \leq q < n-n_{\nu\!-\!2}$, such that $\{\omega^i,\overline{\omega}^i\}_{i=1}^{q}$ is a basis of $(W^1 \oplus \overline{W}^1) \cap (\mathfrak{g}/\mathfrak{a}_{\nu\!-\!2})^*_\mathbb{C}$.

Now we proceed to prove the $k=2$ case of \eqref{exclusion}. Extend the basis $\{ X_i,\overline{X}_i \}_{i=1}^{n-n_{\nu\!-\!2}}$ of $(\mathfrak{g}/\mathfrak{a}_{\nu\!-\!2})_{\mathbb{C}}$ to the one $\{ X_i,\overline{X}_i \}_{i=1}^{n-n_{\nu\!-\!3}}$ of $(\mathfrak{g}/\mathfrak{a}_{\nu\!-\!3})_{\mathbb{C}}$ such that $\{ X_i,\overline{X}_i \}_{i=n-n_{\nu\!-\!2}+1}^{n-n_{\nu\!-\!3}}$ is a basis of $(\mathfrak{a}_{\nu\!-\!2}/\mathfrak{a}_{\nu\!-\!3})_{\mathbb{C}}$, with the dual basis in $(\mathfrak{g}/\mathfrak{a}_{\nu\!-\!3})^*_\mathbb{C}$ denoted by $\{ \omega^i,\overline{\omega}^i \}_{i=1}^{n-n_{\nu\!-\!3}}$,
since
\[\mathfrak{g}/\mathfrak{a}_{\nu\!-\!2} \cong \mathfrak{g}/\mathfrak{a}_{\nu\!-\!3} \Big/ \mathfrak{a}_{\nu\!-\!2}/\mathfrak{a}_{\nu\!-\!3}.\]
Once again we have $[\mathfrak{a}_{\nu\!-\!2}/\mathfrak{a}_{\nu\!-\!3},\, \mathfrak{g}/\mathfrak{a}_{\nu\!-\!3}]=0$ since $[\mathfrak{a}_{\nu\!-\!2},\mathfrak{g}] \subseteq \mathfrak{a}_{\nu\!-\!3}$. For each $n-n_{\nu\!-\!1}+1 \leq i \leq n-n_{\nu\!-\!2}$, we must have $[X_i, \mathfrak{g}/\mathfrak{a}_{\nu\!-\!3}]\neq 0$ in $(\mathfrak{g}/\mathfrak{a}_{\nu\!-\!3})_{\mathbb{C}}$, as otherwise $X_i$ would belong to $(\mathfrak{a}_{\nu\!-\!2}/ \mathfrak{a}_{\nu\!-\!3})_{\mathbb{C}}$, which is absurd.
Therefore we have
\begin{enumerate}
\item for $n-n_{\nu\!-\!2}+1 \leq k \leq n-n_{\nu\!-\!3}$, $1 \leq \ell \leq n-n_{\nu\!-\!3}$, the followings hold
\[[X_k,X_{\ell}]=0, \quad [X_k, \overline{X}_{\ell}]=0,\]
where both are considered as equalities in $(\mathfrak{g}/\mathfrak{a}_{\nu\!-\!3})_{\mathbb{C}}$,
\item for any $n-n_{\nu\!-\!1}+1 \leq i \leq n-n_{\nu\!-\!2}$, there exists some $1 \leq j \leq n-n_{\nu\!-\!2}$ such that
\[ \text{either}\quad [X_i,X_j] \neq 0\quad \text{or}\quad [X_i,\overline{X}_j] \neq 0,\]
considered as equalities $(\mathfrak{a}_{\nu\!-\!2}/\mathfrak{a}_{\nu\!-\!3})_{\mathbb{C}}$,
\item with a little more details involved here, some $\ell$ will be picked up such that $q+1 \leq \ell \leq n-n_{\nu\!-\!2}$, where $q$ is mentioned above. It follows that there exists some $1 \leq j \leq n-n_{\nu\!-\!2}$ such that
    \[\text{either}\quad [X_{\ell},X_j] \neq 0 \quad \text{or} \quad [X_{\ell},\overline{X}_j] \neq 0,\]
    considered as equalities $(\mathfrak{a}_{\nu\!-\!2}/\mathfrak{a}_{\nu\!-\!3})_{\mathbb{C}}$.
\end{enumerate}

Reflecting these facts to the dual basis, we obtain that, for $n-n_{\nu\!-\!2}+1 \leq k \leq n-n_{\nu\!-\!3}$,
\[d \omega^k = \sum_{1\leq i < j \leq n-n_{\nu\!-\!2}}A_{ij}^k\omega^i\wedge\omega^j+\sum_{1\leq i,j \leq n-n_{\nu\!-\!2}}B_{ij}^k\omega^i\wedge\overline{\omega}^j,\]
which implies the $k=3$ case of \eqref{inclusion}. The little more details indicate that there exists at least one $k$ for $n-n_{\nu\!-\!2}+1 \leq k \leq n-n_{\nu\!-\!3}$ such that the coefficients $\{A_{lj}^k\}_{j>l}$, $\{A_{il}^k\}_{i<l}$, $\{B_{lj}^k\}_{j=1}^{n-n_{\nu\!-\!2}}$ and $\{B_{il}^k\}_{ 1 \leq i \leq n-n_{\nu\!-\!2}, i \neq l}$ of the $d\omega^k$ don't all vanish, which implies that $\omega^k \in W^3 \setminus W^2$. This establishes the $k=2$ case of \eqref{exclusion}. The  other cases can be proved similarly.
\end{proof}

Note that some of the inclusions in \eqref{inclusion} could be strict, as illustrated by the following example.
\begin{example}\cite[Example 4]{CFGU}
Let the pair $(\mathfrak{g},J)$ be a NLA $\mathfrak{g}$ endowed with a nilpotent complex structure $J$,
determined by the structure equation
\[\begin{cases} d\omega^1 =0,\\
d\omega^2 = \omega^1 \wedge \overline{\omega}^1,\\
d\omega^3 = -\omega^1 \wedge \overline{\omega}^1,\\
d\omega^4 = \omega^1 \wedge (\omega^2+\overline{\omega}^2),\\
d\omega^5 = \frac{1}{2} \omega^1 \wedge (-\omega^2-\omega^3+2\omega^4+\overline{\omega}^2+\overline{\omega}^3),
\end{cases}\]
where $\{\omega^i\}_{i=1}^5$ is a basis of $\mathfrak{g}^{1,0}$.
The dual type of the structure equation immediately follows, with $\{X_i\}_{i=1}^5$ being the dual basis of $\{\omega^i\}_{i=1}^5$,
\[ \begin{cases}
[X_1, \overline{X}_1] = -X_2 + \overline{X}_2 + X_3 - \overline{X}_3,\\
[X_1,X_2]=-X_4 + \frac{1}{2} X_5, \quad [X_1,\overline{X}_2]=-X_4-\frac{1}{2}X_5,\\
[X_1,X_3]=\frac{1}{2}X_5, \quad [X_1,\overline{X}_3]=-\frac{1}{2}X_5,\\
[X_1,X_4]=-X_5.
\end{cases} \]
Then it is easy to show that
\[\begin{cases}
\mathfrak{a}_1^{\mathbb{C}} = \langle X_5, \overline{X}_5\rangle,\\
\mathfrak{a}_2^{\mathbb{C}} = \langle X_3,\overline{X}_3, X_4,\overline{X}_4, X_5, \overline{X}_5\rangle,\\
\mathfrak{a}_3^{\mathbb{C}} = \langle  X_2,\overline{X}_2, X_3,\overline{X}_3, X_4,\overline{X}_4, X_5, \overline{X}_5\rangle,\\
\mathfrak{a}_4^{\mathbb{C}} = \mathfrak{g}^*_{\mathbb{C}},\\
\end{cases}\]
where $\mathfrak{a}_i^{\mathbb{C}} = \mathfrak{a}_i \otimes_{\mathbb{R}} \mathbb{C}$ for any $i$.
It follows that
\[\begin{array}{rcl}
(\mathfrak{g}/\mathfrak{a}_4)^*_{\mathbb{C}} = 0 & = & W^0 \oplus \overline{W}^0 = 0,\\
(\mathfrak{g}/\mathfrak{a}_3)^*_{\mathbb{C}} = \langle \omega^i, \overline{\omega}^i\rangle_{i=1} & \subsetneq & W^1 \oplus \overline{W}^1 = \langle \omega^1, \omega^2+\omega^3, \overline{\omega}^1,\overline{\omega}^2+\overline{\omega}^3\rangle,\\
(\mathfrak{g}/\mathfrak{a}_2)^*_{\mathbb{C}} = \langle \omega^i, \overline{\omega}^i\rangle_{i=1}^2 & \subsetneq & W^2 \oplus \overline{W}^2 = \langle \omega^i, \overline{\omega}^i\rangle_{i=1}^3,\\
(\mathfrak{g}/\mathfrak{a}_1)^*_{\mathbb{C}} = \langle \omega^i, \overline{\omega}^i\rangle_{i=1}^4 & = & W^3 \oplus \overline{W}^3 = \langle \omega^i, \overline{\omega}^i\rangle_{i=1}^4,\\
\mathfrak{g}^*_{\mathbb{C}} = \langle \omega^i, \overline{\omega}^i\rangle_{i=1}^5 & = & W^4 \oplus \overline{W}^4 = \langle \omega^i, \overline{\omega}^i\rangle_{i=1}^5,\\
\end{array}\]
where $(\mathfrak{g}/\mathfrak{a}_i)^*_{\mathbb{C}}=(\mathfrak{g}/\mathfrak{a}_i)^* \otimes_{\mathbb{R}} \mathbb{C}$ for any $i$,
which implies that
\[\begin{split}
\omega^1 \in (\mathfrak{g}/\mathfrak{a}_3)^*_{\mathbb{C}},\quad\omega^1\notin W^0 \oplus \overline{W}^0,\\
\omega^2\in (\mathfrak{g}/\mathfrak{a}_2)^*_{\mathbb{C}},\quad\omega^2\notin W^1 \oplus \overline{W}^1,\\
\omega^4\in (\mathfrak{g}/\mathfrak{a}_1)^*_{\mathbb{C}},\quad\omega^4\notin W^2 \oplus \overline{W}^2,\\
\omega^5\in \mathfrak{g}^*_{\mathbb{C}},\quad\omega^5\notin W^3 \oplus \overline{W}^3.\\
\end{split}\]
\end{example}

It is natural to consider the annihilators of $\{W^i \oplus \overline{W}^i\}_{i\in\mathbb{N}}$ accordingly, which motivates the following definition. It is closely related to \cite[the first paragraph of Section 1.2.2]{Rol09Geo} and \cite[Lemma 3]{CF}.
\begin{definition}\label{annihilator}
Let the pair $(\mathfrak{g},J)$ be a NLA $\mathfrak{g}$ endowed with a complex structure $J$.
Define a \emph{descending series of subspaces} $\{\mathfrak{h}^k\}_{k\in\mathbb{N}}$ \emph{compatible with} $J$ of $\mathfrak{g}$,
\[\mathfrak{h}^0=\mathfrak{g},\quad \mathfrak{h}^k=[\mathfrak{h}^{k-1},\mathfrak{g}]+J[\mathfrak{h}^{k-1},\mathfrak{g}]\quad k\geq 1.\]
It is easy to verify that $\mathfrak{h}^k$ is a $J$-invariant ideal and $\mathfrak{g}^k + J\mathfrak{g}^k \subseteq \mathfrak{h}^k $ for $k \in \mathbb{N}$, since $[\mathfrak{h}^k,\mathfrak{g}] \subseteq [\mathfrak{h}^{k-1},\mathfrak{g}] \subseteq \mathfrak{h}^{k}$.
\end{definition}

\begin{lemma}\label{wcrt}
The following three statements hold.
\begin{enumerate}
\item Let $V$ be a real subspace of $\mathfrak{g}^*$. Then, for $\beta \in \wedge^2\mathfrak{g}^*$
\[\beta \in V \wedge V \Longleftrightarrow \forall \,\theta \in \mathfrak{g},\ \iota_{\theta}\beta \in V.\]
\item Let $U$ be a complex subspace of $\mathfrak{g}^*_{\mathbb{C}}$. It yields that, for $\beta \in \wedge^2\mathfrak{g}^*_{\mathbb{C}}$,
\[\beta \in U \wedge U \Longleftrightarrow \forall \,\theta \in \mathfrak{g}_{\mathbb{C}},\ \iota_{\theta}\beta \in U.\]
\item Let $W$ be a complex subspace of $\mathfrak{g}^{1,0}$. It follows that, for $\alpha \in \mathfrak{g}^{1,0}$,
\[d\alpha \in (W \wedge W) \oplus  (W \wedge \overline{W}) \Longleftrightarrow \forall \,\theta \in \mathfrak{g}_{\mathbb{C}},\ \iota_{\theta}d\alpha \in W \oplus \overline{W}.\]
\end{enumerate}
Here $\iota_{\theta}$ denotes the contraction operator with respect to $\theta$.
\end{lemma}

\begin{proof}
Let $\{\phi^i\}_{i=1}^n$ be a basis of $\mathfrak{g}^*$ with $\{\phi^i\}_{i=1}^m$ being the basis of $V$ for some $m \leq n$, where the dual basis is denoted by $\{\theta_i\}_{i=1}^n$.
If $\beta \in V \wedge V$, it yields that \[\beta = \sum_{i<j \leq m}c_{ij} \phi^i \wedge \phi^j\] for $c_{ij} \in \mathbb{R}$.
It is easy to see that, for $\theta \in \mathfrak{g}$,
\[\iota_{\theta}\beta = \sum_{i<j\leq m} c_{ij} (\iota_{\theta}\phi^i) \wedge \phi^j - c_{ij} \phi^i \wedge (\iota_{\theta}\phi^j) \ \in \text{V},\]
where it should be noted that $\iota_{\theta}\phi^i$ and $\iota_{\theta}\phi^j$ are constants. Conversely, suppose that $\forall \,\theta \in \mathfrak{g},\ \iota_{\theta}\beta\in V$. If \[\beta = \sum_{i<j \leq n}c_{ij} \phi^i \wedge \phi^j,\]
with the coefficients $\{c_{ij}\}_{i>m}$ and $\{c_{ij}\}_{\begin{subarray}{c} i \leq m\\ j>m \end{subarray}}$ are not all zeros. Then, when $c_{ij} \neq 0$ for some $i>m$, it yields that
\[\iota_{\theta_i} \beta = \sum_{j>i}c_{ij} \phi^j -\sum_{j<i}c_{ji}\phi^j, \]
where some term in $\sum_{j>i}c_{ij} \phi^j$ is not zero and thus $\iota_{\theta_i} \beta \notin V$. When $c_{ij} \neq 0$ for some $i \leq m$ and some $j>m$, it similarly follows that $\iota_{\theta_i} \beta \notin V$, since the term $c_{ij}\phi^j$ in $\sum_{j>i}c_{ij} \phi^j$ is not zero. Therefore, the coefficients $\{c_{ij}\}_{i>m}$ and $\{c_{ij}\}_{\begin{subarray}{c} i \leq m\\ j>m \end{subarray}}$ have to be all zeros, which implies $\beta \in V \wedge V$.

The second statement follows by the same method in the real case. As to the third, it should be noted that, for $\alpha \in \mathfrak{g}^{1,0}$,
\[d\alpha \in \left( W \wedge W \right) \oplus \left( W \wedge \overline{W} \right)
\ \Longleftrightarrow \ d\alpha \in \left( W \oplus \overline{W}\right) \wedge \left( W \oplus \overline{W}\right),\]
due to the integrability condition \eqref{integrability}. Then the following equivalence can be proved by the second statement
\[d\alpha \in \left( W \oplus \overline{W}\right) \wedge \left( W \oplus \overline{W}\right) \ \Longleftrightarrow \ \forall\, \theta \in \mathfrak{g}_{\mathbb{C}},\ \iota_{\theta}d\alpha \in W \oplus \overline{W}.\]
\end{proof}

\begin{proposition}\label{ann_equ}
For $k \in \mathbb{N}$, it holds that
\[\mathrm{Ann}\left(W^k \oplus \overline{W}^{k}\right) = \mathfrak{h}^k_{\mathbb{C}},\]
where the annihilator $\mathrm{Ann}(W^k \oplus \overline{W}^{k})$ of $W^k \oplus \overline{W}^{k}$ in $\mathfrak{g}_{\mathbb{C}}$ is defined by the set \[\{\theta\in\mathfrak{g}_{\mathbb{C}} \ \big| \ \forall \,\alpha \in W^k,\ \alpha(\theta)= \overline{\alpha}(\theta)=0\},\]
and $\mathfrak{h}^k_{\mathbb{C}}=\mathfrak{h}^k \otimes_{\mathbb{R}} \mathbb{C}$.
\end{proposition}

\begin{remark}\label{a=h}
Let the pair $(\mathfrak{g},J)$ be a NLA $\mathfrak{g}$ endowed with a nilpotent complex structure $J$.
It follows from Theorem \ref{nu=mu} and Proposition \ref{ann_equ} that $\mathfrak{h}^k \subseteq \mathfrak{a}_{\nu(J)-k}$ for $0 \leq k \leq \nu(J)$.
\end{remark}

\begin{proof}
The case of $k=0$ is trivial. As to the case of $k=1$, note that \[\mathfrak{h}^1 = \mathfrak{g}^1 + J\mathfrak{g}^1.\]
For $\alpha \in W^1$ and $X \in \mathfrak{g}^1$, where $X$ can be expressed as $\sum_{i=1}^m [Y_i,Z_i]$ for $Y_i,Z_i \in \mathfrak{g}$, it follows that \[\begin{aligned} \alpha(X) &= \sum_{i=1}^m -d\alpha(Y_i,Z_i)= 0,\\
\alpha(JX) &= J\alpha(X)=\sqrt{-1}\alpha(X)=\sum_{i=1}^m -\sqrt{-1}d\alpha(Y_i,Z_i)= 0,
\end{aligned}\]
and $\overline{\alpha}(X)=0$, $\overline{\alpha}(JX)=0$ are similarly obtained, which imply that
\[\mathfrak{h}^1_{\mathbb{C}} \subseteq \mathrm{Ann}(W^1 \oplus \overline{W}^{1}).\]
Then, let $\alpha \in \mathfrak{g}^*_{\mathbb{C}}$ which satisfies
\[ \alpha(\mathfrak{h}^1_{\mathbb{C}}) =0. \]
It yields that $\alpha(\mathfrak{g}^1 \otimes_{\mathbb{R}} \mathbb{C})=0$ and $\alpha(J\mathfrak{g}^1 \otimes_{\mathbb{R}} \mathbb{C})=0$.
From $\mathrm{Ann}(V^1_{\mathbb{C}}) = \mathfrak{g}^1_{\mathbb{C}}$, it follows that
\[d\alpha=0,\ d(J\alpha)=0,\]
which imply that \[d(\alpha^{1,0} +\alpha^{0,1})=0, \ d(\alpha^{1,0} -\alpha^{0,1})=0,\]
where $\alpha$ is decomposed as the sum of the $(1,0)$ and $(0,1)$-components $\alpha^{1,0}+\alpha^{0,1}$. Hence, it indicates that \[\alpha \in W^1 \oplus \overline{W}^{1},\] which follows that $\mathrm{Ann}(\mathfrak{h}^1_{\mathbb{C}}) \subseteq W^1 \oplus \overline{W}^{1}$ and thus \[\mathrm{Ann}(W^1 \oplus \overline{W}^{1}) \subseteq \mathfrak{h}^1_{\mathbb{C}}.\] Therefore, the case of $k=1$ is established.

By induction, assume that the statement $\mathrm{Ann}(W^i \oplus \overline{W}^{i}) = \mathfrak{h}^i_{\mathbb{C}}$ holds for $i \leq k$. Let us consider the case of $i=k+1$. For $\alpha \in W^{k+1}$, it follows that
\[\begin{aligned}
\alpha([\mathfrak{h}^{k},\mathfrak{g}]) &= -d\alpha(\mathfrak{h}^{k},\mathfrak{g}) =0,\\
\alpha(J[\mathfrak{h}^{k},\mathfrak{g}]) &= J\alpha([\mathfrak{h}^{k},\mathfrak{g}]) =\sqrt{-1}\alpha([\mathfrak{h}^{k},\mathfrak{g}]) \\
                                         &=-\sqrt{-1}d\alpha(\mathfrak{h}^{k},\mathfrak{g})=0,
\end{aligned} \]
and $\overline{\alpha}([\mathfrak{h}^{k},\mathfrak{g}])=0$, $\overline{\alpha}(J[\mathfrak{h}^{k},\mathfrak{g}])=0$ are similarly established, since $d\alpha \in (W^k \wedge W^k) \oplus (W^k \wedge \overline{W}^{k})$ and $\mathrm{Ann}(W^k \oplus \overline{W}^{k}) = \mathfrak{h}^k _{\mathbb{C}}$, which imply
\[\mathfrak{h}^{k+1}_{\mathbb{C}} \subseteq \mathrm{Ann}(W^{k+1} \oplus \overline{W}^{k+1}).\]
Then let $\alpha \in \mathfrak{g}^*_{\mathbb{C}}$ which satisfies
\[ \alpha(\mathfrak{h}^{k+1}_{\mathbb{C}}) =0.\]
It yields that $\alpha([\mathfrak{h}^{k},\mathfrak{g}]) =0$ and $\alpha(J[\mathfrak{h}^{k},\mathfrak{g}]) =0$, which imply
\[d\alpha(\mathfrak{h}^{k},\mathfrak{g})=0,\ \sqrt{-1}d(\alpha^{1,0}-\alpha^{0,1})(\mathfrak{h}^{k},\mathfrak{g})=0,\]
and thus
\[d\alpha^{1,0}(\mathfrak{h}^{k},\mathfrak{g})=0,\ d\alpha^{0,1}(\mathfrak{h}^{k},\mathfrak{g})=0.\]
It follows that, for $\forall \theta \in \mathfrak{g}_{\mathbb{C}}$,
\[\iota_{\theta} d\alpha^{1,0} (\mathfrak{h}^k)=0,\]
which is equivalent to, due to $\mathrm{Ann}(W^k \oplus \overline{W}^{k}) = \mathfrak{h}^k_{\mathbb{C}}$,
\[ \iota_{\theta} d\alpha^{1,0} \in W^k \oplus \overline{W}^{k}. \]
Therefore, from the lemma \ref{wcrt}, it yields that $d\alpha^{1,0} \in  (W^k \wedge W^k) \oplus (W^k \wedge \overline{W}^{k})$ and $d\overline{\alpha^{0,1}} \in  (W^k \wedge W^k) \oplus (W^k \wedge \overline{W}^{k})$ is similarly proved. Hence, it follows that \[\alpha=\alpha^{1,0}+\alpha^{0,1} \in W^{k+1} \oplus \overline{W}^{k+1},\] which indicates that $\mathrm{Ann}(\mathfrak{h}^{k+1}_{\mathbb{C}}) \subseteq W^{k+1} \oplus \overline{W}^{k+1}$ and thus \[\mathrm{Ann}(W^{k+1} \oplus \overline{W}^{k+1}) \subseteq \mathfrak{h}^{k+1}_{\mathbb{C}}.\]
The proof of the proposition is completed.
\end{proof}

\begin{proof}[Proof of Theorem \ref{equivalence}]
\eqref{i} $\Rightarrow$ \eqref{ii}: The statement \eqref{i} implies that $J$ is nilpotent and thus \eqref{ii} follows from Theorem \ref{nu=mu}.

\eqref{ii} $\Rightarrow$ \eqref{i}: The statement \eqref{ii} implies that $J$ is nilpotent by Observation \ref{existence_mu} and thus \eqref{i} follows from Theorem \ref{nu=mu}.

\eqref{ii} $\Longleftrightarrow$ \eqref{iii}: This is due to Proposition \ref{ann_equ}.
\end{proof}

\begin{definition}\label{newdef_nil_cplx}
The integer $t$ such that the condition \eqref{i}, \eqref{ii} or \eqref{iii} of Theorem \ref{equivalence} holds, will be called \emph{the step of the complex structure} $J$, denoted by a unified symbol $\nu(J)$ henceforth.
\end{definition}

\begin{proof}[Proof of Corollary \ref{equ_mcn}]
If the pair $(\mathfrak{g},J)$ is MaxN, it follows that
\[ \nu(J)=\dim_{\mathbb{C}}\mathfrak{g}=n, \]
which implies, by Theorem \ref{equivalence}, that the equalities hold:
\[\dim_{\mathbb{C}}W^i=i,\quad 1\leq i \leq n.\]
Thus, there is a basis $\{\omega^i\}_{i=1}^n$ of $\mathfrak{g}^{1,0}$ such that
\[ W^i=\mathrm{span}_{\mathbb{C}}\{\omega^1,\cdots,\omega^i\},\]
which is equivalent to \eqref{streq_nil_cplx_mcn} by the definition of $W^i$.

Conversely, the equation \eqref{streq_nil_cplx_mcn} forces the underlying Lie algebra $\mathfrak{g}$ to be nilpotent, since
the structure equation via $\{\mathfrak{Re}(\omega^i),\mathfrak{Im}(\omega^i)\}_{i=1}^n$ is of the type \eqref{streq_form}.
It is then clear that
\[\omega^i \in W^i, \quad 1 \leq i \leq n,\]
which leads to
\[\mathrm{span}_{\mathbb{C}}\{\omega^1,\cdots,\omega^i\} \subseteq W^i, \quad 1 \leq i \leq n.\]
Note that  when $i=n$ the above the inclusion is necessarily an equality, since $\{\omega^i\}_{i=1}^n$ is a basis of $\mathfrak{g}^{1,0}$,
which implies $\nu(J)\leq n$ and thus $J$ is nilpotent by Theorem \ref{equivalence}. Besides, it follows that
\[ \mathrm{span}_{\mathbb{C}}\{\omega^1\} = W^1, \]
due to the condition that the coefficients $\{A_{i,k-1}^k\}_{i<k-1}$, $\{B_{k-1,j}^k\}_{j=1}^{k-1}$ and $\{B_{i,k-1}^k\}_{i<k-1}$ of $d\omega^k$ don't all vanish for $1 \leq k \leq n$ in \eqref{streq_nil_cplx_mcn}. By induction, it follows easily that
\[\mathrm{span}_{\mathbb{C}}\{\omega^1,\cdots,\omega^i\} = W^i, \quad 2 \leq i \leq n,\]
which indicates $\nu(J)=\dim_{\mathbb{C}}\mathfrak{g}=n$.

It yields that the MaxN complex structure implies $\dim_{\mathbb{C}}\mathfrak{a}_i=i$ and $\dim_{\mathbb{C}}W^i=i$ for $0 \leq i \leq \nu(J)$.
So by the dimension reasons, the equality $(\mathfrak{g}/\mathfrak{a}_{\nu(J)-k})^*_{\mathbb{C}} = W^k \oplus \overline{W}^k$ holds for $0 \leq k \leq \nu(J)$.
\end{proof}

\begin{proof}[Proof of Corollary \ref{mcn_hierarchy}]
The maximal nilpotency of $(\mathfrak{g},J)$ implies that $\dim_{\mathbb{C}}\mathfrak{a}_k = k$ and $\dim_{\mathbb{C}}W^k =k$ for $0 \leq k \leq n$, by Theorem \ref{equivalence}. It is easy to verify that, for any given $0 \leq k \leq \nu(J)$, the ascending series $\{\mathfrak{a}_i\}_{i \in \mathbb{N}}$ of $(\mathfrak{a}_{k},J)$ is exactly $\{\mathfrak{a}_i\}_{i=0}^k$ and thus the number $\nu(J)$ for $(\mathfrak{a}_{k},J)$ is equal to $k$. Similarly by Corollary \ref{equ_mcn}, it follows that the equality $(\mathfrak{g}/\mathfrak{a}_{\nu(J)-k})^*_{\mathbb{C}} = W^k \oplus \overline{W}^k$ holds for $0 \leq k \leq \nu(J)$. Then for any given $0 \leq k \leq \nu(J)$, the ascending series $\{W^i\}_{i \in \mathbb{N}}$ of $(\mathfrak{g}/\mathfrak{a}_{\nu(J)-k},J)$ is exactly $\{W^i\}_{i=0}^k$, which implies that $\nu(J)$ of $(\mathfrak{g}/\mathfrak{a}_{\nu(J)-k},J)$ is equal to $k$. Therefore $(\mathfrak{a}_k,J)$ and $(\mathfrak{g}/\mathfrak{a}_{\nu(J)-k},J)$ are MaxN. The equalities $\mathfrak{a}_{\nu(J)-k} = \mathfrak{h}^k$ for $0 \leq k \leq \nu(J)$ result from Remark \ref{a=h} and $\dim_{\mathbb{C}}\mathfrak{a}_{\nu(J)-k} = \dim_{\mathbb{C}}\mathfrak{h}^k = n-k$ in the MaxN case .
\end{proof}

\begin{proof}[Proof of Corollary \ref{mcn_agdim}]
By \cite[Proposition 1]{FGV}, the equality $H^0(M,d\mathcal{O}_M)=W^1$ is established, where $H^0(M,d\mathcal{O}_M)$ denotes the space of $d$-closed holomorphic $1$-forms on $M$. The complex structure being MaxN implies that $\dim_{\mathbb{C}}W^1=1$. Then $a(M) \leq \dim_{\mathbb{C}} H^0(M,d\mathcal{O}_M) =1$ from \cite[Main theorem]{FGV}. By the definition of the Albanese variety $\mathrm{Alb}(M)$ (cf. \cite[Definition 2.1]{Rol09Geo}),
it follows that
\[\mathrm{Alb}(M)= \frac{H^0(M,d\mathcal{O}_M)^*}{\mathrm{im}(H_1(M,\mathbb{Z})\rightarrow H_1(M,\mathbb{C}) \rightarrow H^0(M,d\mathcal{O}_M)^*)},\]
where the projection $H_1(M,\mathbb{C}) \rightarrow H^0(M,d\mathcal{O}_M)^*$ is induced by the injection
\[ H^0(M,d\mathcal{O}_M) \oplus \overline{H^0(M,d\mathcal{O}_M)} \hookrightarrow H^1_{dR}(M,\mathbb{C}).\]
From \cite[Proposition 1]{FGV} and Proposition \ref{ann_equ}, we get
\[H^0(M,d\mathcal{O}_M) \oplus \overline{H^0(M,d\mathcal{O}_M)}= W^1 \oplus \overline{W}^1 = (\mathfrak{g}/\mathfrak{h}^1)^* \otimes_{\mathbb{R}} {\mathbb{C}}, \]
where $\mathfrak{h}^1= \mathfrak{g}^1+J\mathfrak{g}^1$ as in Definition \ref{annihilator}. Note that $\pi_1(M)=\Gamma$, and $\dim_{\mathbb{C}}W^1 =1$ since $J$ is MaxN. It yields that $\mathrm{Alb}(M)$ is a complex torus of complex dimension $1$. Hence, it follows from \cite[Main theorem]{FGV} that $a(M)=a(\mathrm{Alb}(M))=1$.

The Albanese map is given by the integration over $\Omega$ along paths on $M$,
where $\Omega$ is the basis of $H^0(M,d\mathcal{O}_M)=W^1$. Clearly, $\Omega$ is nowhere vanishing and thus the Albanese map $\Psi: M \rightarrow \mathrm{Alb}(M)$ is a submersion. Remmert's proper mapping theorem \cite[Page 32]{BHPV} implies that $\Psi$ is surjective, so $\Psi$ is a smooth fibration.
\end{proof}

\section{Proof of Theorem \ref{step2}}\label{pfstep2}
\begin{proof}[Proof of Theorem \ref{step2}]
The Lie algebra $\mathfrak{g}$ is $2$-step nilpotent, which is equivalent to $\mathfrak{g}^1 \subseteq \mathfrak{g}_1 \subsetneq \mathfrak{g}$.
The situation can be separated into two cases as follows
\begin{enumerate}
\item\label{s-1} $\mathfrak{g}_1$ is $J$-invariant,
\item\label{s-2} $\mathfrak{g}_1$ is not $J$-invariant.
\end{enumerate}

Case \eqref{s-1}: The equality $\mathfrak{a}_1= \mathfrak{g}_1$ holds and thus $\mathfrak{a}_2= \mathfrak{g}$ follows by the definition of $\mathfrak{a}_2$. Then $\nu(J)=2$.

Case \eqref{s-2}: We will show that $\mathfrak{a}_3= \mathfrak{g}$ in this case.
Define
\[\mathcal{V}^{J}_1 \stackrel{\vartriangle}{=} [\mathfrak{g}_1+J\mathfrak{g}_1,\,\mathfrak{g}]
=\mathrm{span}_{\mathbb{R}}\{ [y,x] \,\big| \,y\in \mathfrak{g}_1+J\mathfrak{g}_1, \,x\in\mathfrak{g}\},\]
introduced by Rollenske \cite{Rol09Geo}, which is a non-trivial $J$-invariant subspace of $\mathfrak{g}^1$ as shown in \cite[Lemma 3.2]{Rol09Geo}, and thus it is necessarily contained in $\mathfrak{g}_1$. Whereas $\mathfrak{a}_1$ is the largest subspace of $\mathfrak{g}_1$ which is invariant under $J$ by \cite[Corollary 5(i)]{CFGU}, it yields that \[\mathcal{V}^J_1 \subseteq \mathfrak{a}_1,\] which implies $[J\mathfrak{g}_1,\mathfrak{g}] \subseteq \mathfrak{a}_1$. By the very definition of $\mathfrak{a}_2$, it follows that $\mathfrak{g}_1 \subsetneq \mathfrak{a}_2$ and thus $\mathfrak{g}^1 \subsetneq \mathfrak{a}_2$. Therefore, $\mathfrak{a}_3=\mathfrak{g}$ by \cite[Lemma 6(ii)]{CFGU} or the very definition of $\mathfrak{a}_3$. So for 2-step ${\mathfrak g}$ one always  has
\[2\leq \nu(J) \leq 3.\]
This completes the proof of Theorem  \ref{step2}. \end{proof}
\begin{remark}
An example, satisfying $\nu(\mathfrak{g})=2$ and $\nu(J)=\dim_{\mathbb{C}}\mathfrak{g}=3$, is illustrated in \cite[Example 3]{CFGU}.
\end{remark}

\vspace{0.3cm}

\section{Example \ref{ex_max}}\label{ex}
The results developed in Section \ref{ncs_mcn} will be applied here to show the basic properties of Example \ref{ex_max}.
\begin{proposition}
Both types of complex structures $(\mathfrak{k}_{n}^{\mathrm{I}},J_{n}^{\mathrm{I}})$ and $(\mathfrak{k}_{n}^{\mathrm{II}},J_{n}^{\mathrm{II}})$ are MaxN for $n\geq1$. For the Lie Algebras, $\nu(\mathfrak{k}_n^{\mathrm{I}})=\nu(\mathfrak{k}_n^{\mathrm{II}})=3$ for $n\geq4$, and $\nu(\mathfrak{k}_3^{\mathrm{I}})=3,\nu(\mathfrak{k}_3^{\mathrm{II}})=2$ for $n=3$.
\end{proposition}
\begin{proof}
It is easy to see that the two pairs $(\mathfrak{k}_{n}^{\mathrm{I}},J_{n}^{\mathrm{I}})$ and $(\mathfrak{k}_{n}^{\mathrm{II}},J_{n}^{\mathrm{II}})$ are MaxN for $n \geq 1$ by Corollary \ref{equ_mcn}, from their structure equations. The key in figuring out the steps of $\mathfrak{k}_{n}^{\mathrm{I}}$ and $\mathfrak{k}_{n}^{\mathrm{II}}$ lies in $V^1$. We will consider the following two cases separately:
\begin{enumerate}
\item\label{even} $n = 2k$,  where $k \geq 2$;
\item\label{odd} $n = 2k-1$, where $k \geq 2$.
\end{enumerate}

Case \eqref{even}: From the structure equations of $(\mathfrak{k}_{2k}^{\mathrm{I}},J_{2k}^{\mathrm{I}})$ and $(\mathfrak{k}_{2k}^{\mathrm{II}},J_{2k}^{\mathrm{II}})$, it yields that
\[\begin{split}
V^{1}_{\!\mathbb{C}\,(\mathrm{I},2k)} &= \langle\omega^1\!, \, \overline{\omega}^1\!,\,\omega^2+\overline{\omega}^2\!,\, \omega^3 + \overline{\omega}^{3}\!,\, \cdots,\,\omega^{2k-1}+\overline{\omega}^{2k-1}\rangle,\\
V^1_{\!\mathbb{C}\,({\mathrm{II}},2k)}  &= \langle\omega^1\!,\, \overline{\omega}^1\!,\,\omega^2+\overline{\omega}^2\!,\,\omega^4 + \overline{\omega}^{4}\!,\,\cdots,\omega^{2k}+\overline{\omega}^{2k}\rangle,\\
\end{split}\]
where the complexifixed $V^i$-spaces of $(\mathfrak{k}_{2k}^{\mathrm{I}},J_{2k}^{\mathrm{I}})$ and $(\mathfrak{k}_{2k}^{\mathrm{II}},J_{2k}^{\mathrm{II}})$ are denoted by $V^{i}_{\!\mathbb{C}\,(\mathrm{I},2k)}$ and $V^i_{\!\mathbb{C}\,({\mathrm{II}},2k)}$ respectively, based on the definition of $\{V^i\}_{i \in \mathbb{N}}$ in Section \ref{ncs_mcn}.
Similar notations are applied below. It then follows that, for any $2\leq i \leq k$,
\[\begin{split}
\omega^{2i}, \,\omega^{2} \in V^2_{\!\mathbb{C}\,(\mathrm{I},2k)}, \qquad
\omega^{2i-1} \in V^3_{\!\mathbb{C}\,(\mathrm{I},2k)} \setminus  V^2_{\!\mathbb{C}\,(\mathrm{I},2k)},\\
\omega^{2i-1}\!, \,\omega^{2} \in  V^2_{\!\mathbb{C}\,(\mathrm{II},2k)}, \qquad
\omega^{2i} \in V^3_{\!\mathbb{C}\,(\mathrm{II},2k)} \setminus V^2_{\!\mathbb{C}\,(\mathrm{II},2k)}.
\end{split}\]
It implies that $\nu(\mathfrak{k}_{2k}^{\mathrm{I}})=\nu(\mathfrak{k}_{2k}^{\mathrm{II}})=3$ for all $k \geq2$.

Case \eqref{odd}: Similarly, from the structure equations of $(\mathfrak{k}_{2k-1}^{\mathrm{I}},J_{2k-1}^{\mathrm{I}})$ and $(\mathfrak{k}_{2k-1}^{\mathrm{II}},J_{2k-1}^{\mathrm{II}})$, it yields that
\[\begin{split}
V^1_{\!\mathbb{C}\,(\mathrm{I},2k-1)} &= \langle\omega^1\!, \, \overline{\omega}^1\!, \,\omega^2+\overline{\omega}^2\!, \, \omega^3 + \overline{\omega}^{3}\!, \,\cdots\!, \, \omega^{2k-3}+\overline{\omega}^{2k-3}\!, \, \omega^{2k-1}+\overline{\omega}^{2k-1}\rangle,\\
V^1_{\!\mathbb{C}\,(\mathrm{II},2k-1)} &= \langle\omega^1\!, \, \overline{\omega}^1\!, \,\omega^2+\overline{\omega}^2\!, \,\omega^4 + \overline{\omega}^{4}\!, \,\cdots\!, \, \omega^{2k-2}+\overline{\omega}^{2k-2}\rangle.\\
\end{split}\]
Then, it follows that, for $2 \leq i \leq k-1$,
\[\begin{split}
\omega^{2i}\!, \,\omega^2 \in V^2_{\mathbb{C}\,(\mathrm{I},2k-1)}, \qquad
\omega^{2k-1}\!, \,\omega^{2i-1} \in V^3_{\mathbb{C}\,(\mathrm{I},2k-1)} \setminus V^2_{\mathbb{C}\,(\mathrm{I},2k-1)},\\
\omega^{2k-1}\!, \,\omega^{2i-1}\!, \,\omega^{2} \in V^2_{\mathbb{C}\,(\mathrm{II},2k-1)}, \qquad
\omega^{2i} \in V^3_{\mathbb{C}\,(\mathrm{II},2k-1)} \setminus V^2_{\mathbb{C}\,(\mathrm{II},2k-1)}.
\end{split}\]
It implies that $\nu(\mathfrak{k}_{2k-1}^{\mathrm{I}})=\nu(\mathfrak{k}_{2k-1}^{\mathrm{II}})=3$ for $k \geq3$. When $k=2$, it turns out that
\[
\omega^{3} \in V^3_{\mathbb{C}\,(\mathrm{I},3)} \setminus V^2_{\mathbb{C}\,(\mathrm{I},3)}, \qquad
V^2_{\mathbb{C}\,(\mathrm{II},3)} = V^3_{\mathbb{C}\,(\mathrm{II},3)},
\]
which indicates that $\nu(\mathfrak{k}_3^{\mathrm{I}})=3$ and $\nu(\mathfrak{k}_3^{\mathrm{II}})=2$.
\end{proof}
\begin{remark}
It follows easily from the proof above that $\mathfrak{k}_{2k-1}^{\mathrm{I}}$ and $\mathfrak{k}_{2k-1}^{\mathrm{II}}$ are not isomorphic Lie algebras for $k \geq 2$, due to the different dimensions of $V^1_{\mathbb{C}\,(\mathrm{I},2k-1)}$ and $V^1_{\mathbb{C}\,(\mathrm{II},2k-1)}$.
\end{remark}

\vspace{0.3cm}

\section{A structure theorem}\label{structure}
\subsection{The general structure of a MaxN pair $(\mathfrak{g},J)$}
\begin{proof}[Proof of Theorem \ref{equ_mcn_refined}]
The MaxN complex structure implies the existence of a basis $\{\omega^k\}_{k=1}^n$ of $\mathfrak{g}^{1,0}$, which satisfies \eqref{streq_nil_cplx_mcn}. The induction will be applied to modify the basis $\{\omega^k\}_{k=1}^n$ such that it enjoys the desired property,
where $\{d\omega^k\}_{k=1}^n$ still maintain the type of \eqref{streq_nil_cplx_mcn}.

It is clear that $d\omega^1=0$ and thus the statement \eqref{dpt} holds. Assume that $\omega^i$ satisfies \eqref{dpt} or
\eqref{indpt} for each $1 \leq i \leq k$, and \eqref{streq_nil_cplx_mcn} is established for $\{\omega^i\}_{i=1}^n$. Let us consider
the case $i=k+1$ and suppose that $\omega^{k+1}$ does not satisfy the statement \eqref{indpt}, namely, there exists $a,b\in\mathbb{C}$, which are not both zeros, such that
\[a \,d\omega^{k+1} + b\,d\overline{\omega}^{k+1} \in \mathrm{span}_{\mathbb{C}}\{d\omega^{k},d\overline{\omega}^{k},\cdots,d\omega^2,d\overline{\omega}^2,d\omega^1,d\overline{\omega}^1\}.\]
It implies that $|a|=|b|$. Otherwise, the inclusion
\[\bar{b} \,d\omega^{k+1} + \bar{a} \,d\overline{\omega}^{k+1} \in \mathrm{span}_{\mathbb{C}}\{d\omega^{k},d\overline{\omega}^{k},\cdots,d\omega^2,d\overline{\omega}^2,d\omega^1,d\overline{\omega}^1\},\]
would yield that
\[ d\omega^{k+1} \in \mathrm{span}_{\mathbb{C}}\{d\omega^{k},d\overline{\omega}^{k},\cdots,d\omega^2,d\overline{\omega}^2,d\omega^1,d\overline{\omega}^1\},\]
which is a contradiction with the structure equation \eqref{streq_nil_cplx_mcn} and Remark \ref{d-indept}. The condition $|a|=|b|\neq 0$ enable us to use $c\omega^{k+1}$ as a substitute for $\omega^{k+1}$, where $|c|=1$, so that  $\{\omega^i\}_{i\leq k}$ are left unchanged, while
\[ d\omega^{k+1} + d\overline{\omega}^{k+1} \in \mathrm{span}_{\mathbb{C}}\{d\omega^{k},d\overline{\omega}^{k},\cdots,d\omega^2,d\overline{\omega}^2,d\omega^1,d\overline{\omega}^1\}.\]
It is clear that $d\omega^{k+1}$ still satisfies the type given in \eqref{streq_nil_cplx_mcn}. It follows that
\[ d\omega^{k+1} + d\overline{\omega}^{k+1} = \sum_{i=1}^k a_i d\omega^i + b_i d\overline{\omega}^i, \]
for some $a_i,b_i \in \mathbb{C}$, and it is easy to see that we may assume $\overline{b}_i = a_i$ for each $1 \leq i \leq k$.
It yields that
\[d\left(\omega^{k+1}-\sum_{i=1}^k a_i\omega^i\right) + d\overline{\left(\omega^{k+1}-\sum_{i=1}^k a_i\omega^i\right)} =0,\]
which enables us to replace $\omega^{k+1}$ by
\[\omega^{k+1}-\sum_{i=1}^k a_i\omega^i\]
with $\{\omega^i\}_{i\leq k}$ left unchanged. The operation can be verified such that $d\omega^{k+1}$ still maintains the type as in the structure equation \eqref{streq_nil_cplx_mcn} and $\omega^{k+1}$ now satisfies \eqref{dpt}. This completes the proof of Theorem \ref{equ_mcn_refined}.
\end{proof}

\begin{remark}\label{convention}
Let $(\mathfrak{g},J)$ be MaxN.
\begin{enumerate}
\item\label{omega2} It is clear that $d\omega^2=B_{11}^2\omega^1\wedge \overline{\omega}^1$ where $B_{11}^2\neq0$ due to the structure equation \eqref{streq_nil_cplx_mcn} and thus $\omega^2$ doesn't satisfies the statement \eqref{indpt} in Theorem \ref{equ_mcn_refined}. Then $\omega^2$ can always be assumed to satisfy $d\omega^2=\omega^1\wedge\overline{\omega}^1$ and the statement \eqref{dpt} is established for $\omega^2$, after $\omega^2$ is substituted for $\frac{\omega^2}{B_{11}^2}$.
\item\label{omegak} For $3\leq k \leq n$, it can always be assumed that $d\omega^k$ has no summand of some multiple of $\omega^1\wedge\overline{\omega}^1$, since it is possible to replace $\omega^k$ by $\omega^k-c\omega^2$ for some $c\in\mathbb{C}$, such that $\{\omega^k\}_{k=1}^n$ still enjoys the property of Theorem \ref{equ_mcn_refined}.
\end{enumerate}
\end{remark}

It is clear that $1,2\in \mathrm{Dpt}(J)$, the index set introduced in Definition \ref{DPT}. Example \ref{ex_max} shows that both the case $3 \in \mathrm{Dpt}(J)$ and $3 \notin \mathrm{Dpt}(J)$ can occur, even if we impose the condition $\nu(\mathfrak{g})=3$  on the MaxN pair $(\mathfrak{g},J)$.

\begin{proof}[Proof of Lemma \ref{invariance_dpt}]
Let $\omega^k=\sum_{j=1}^n a^k_j\tau^j$ where $\det(a^k_j)\neq0$. The equation $d\omega^1=0$ implies that
\[\sum_{j=2}^n a^1_j d\tau^j=0, \]
and thus $a^1_j=0$ for $j \geq 2$ due to the linear independence of $\{d\tau^2,\cdots,d\tau^n\}$ from Remark \ref{d-indept}, since the basis $\{\tau^j\}_{j=1}^n$ satisfies \eqref{streq_nil_cplx_mcn}. By induction, assume that the equality $a^i_j=0$ for $j\geq i+1$ is established when $1 \leq i \leq k-1$. Let us consider the case $i=k$. The equality follows \[d\omega^k =\sum_{j=1}^n a^k_j d\tau^j=\sum_{j\leq k}^n a^k_j d\tau^j +\sum_{j\geq k+1}^n a^k_j d\tau^j.\]
The left hand side of the above is just a linear combination of $\tau^p\wedge\tau^q$ and $\tau^p\wedge\bar{\tau}^q$, where $p,q\leq k-1$, while the coefficients $\{a^k_j\}_{j\geq k+1}$ on the right hand side all need to vanish, due to the structure equation \eqref{streq_nil_cplx_mcn} and the induction. Therefore, it yields that $a^k_j=0$ for $1 \leq k \leq n$ and $j \geq k+1$. Then $a^k_k\neq0$ for $1 \leq k \leq n$ follows easily from $\det(a^k_j)\neq0$.

To prove the second statement in the lemma, we will show that, for each $1 \leq k \leq n$,
\[ \omega^k\ \text{satisfies}\ \eqref{indpt} \ \Longleftrightarrow \ \tau^k\ \text{satisfies}\ \eqref{indpt}.\]
Clearly, only one direction needs to be shown. Suppose that $\tau^k$ satisfies \eqref{indpt}. The inclusion
\[ a \,d\omega^k + b\,d\overline{\omega}^k \, \in \, \mathrm{span}_{\mathbb{C}}\{d\omega^{k-1}\!,\,d\overline{\omega}^{k-1}\!,\,\cdots ,\,d\omega^2\!,\,d\overline{\omega}^2\!,\,d\omega^1\!,\,d\overline{\omega}^1\},\]
implies that
\[ a a^k_k d\tau^k + b \bar{a}^k_k d\overline{\tau}^k \,\in \, \mathrm{span}_{\mathbb{C}}\{d\tau^{k-1}\!,\,d\overline{\tau}^{k-1}\!,\,\cdots\!,\,d\tau^2,d\overline{\tau}^2\!,\,d\tau^1\!,\,d\overline{\tau}^1\},\]
since $\omega^k=\sum_{j=1}^n a^k_j\tau^j$, where $a^k_j=0$ for $j\geq k+1$ and $a^k_k\neq0$. Therefore, $a=b=0$.
\end{proof}

The admissible coframe $\{\omega^k\}_{k=1}^n$ and the index set $\mathrm{Dpt}(J)$ in Definition \ref{DPT} will be applied to study the structure of the MaxN pair $(\mathfrak{g},J)$ henceforth.

\begin{corollary}\label{basisV1}
Let $(\mathfrak{g},J)$ be MaxN and $\{\omega^k\}_{k=1}^n$ is an admissible coframe. Then the set
\[\{\omega^1,\overline{\omega}^1,\omega^k+\overline{\omega}^k \,\big| \,k \in \mathrm{Dpt}(J), k\geq 2\}\]
is a basis of $V^1_{\mathbb{C}}$, where $V^1_{\mathbb{C}}=V^1 \otimes \mathbb{C}$.
\end{corollary}

\begin{proof}
Let $\alpha \in V^1_{\mathbb{C}}$, that is, $\alpha=\sum_{k=1}^n a_k\omega^k + b_k \overline{\omega}^k$ for $a_k,b_k \in \mathbb{C}$ and $d \alpha=0$. It follows that
\[ \sum_{k=1}^n a_kd\omega^k + b_k d\overline{\omega}^k =0.\]
If $\omega^n$ satisfies the statement \eqref{indpt} in Theorem \ref{equ_mcn_refined}, it yields that $a_n=b_n=0$. Otherwise, $\omega^n$ must satisfy the statement \eqref{dpt}, so we have
\[ (a_n-b_n)d\omega^n \in \mathrm{span}_{\mathbb{C}}\{d\omega^{n-1}\!,\,d\overline{\omega}^{n-1}\!,\,\cdots\!,\,d\omega^2,d\overline{\omega}^2\!,\,
d\omega^1\!,\,d\overline{\omega}^1\}, \]
which implies that $a_n=b_n$ and $\alpha= a_n (\omega^n+\overline{\omega}^n)+\sum_{k=1}^{n-1} (a_k \omega^k + b_k \overline{\omega}^k)$, since
\[d\omega^n \notin \mathrm{span}_{\mathbb{C}}\{d\omega^{n-1},d\overline{\omega}^{n-1},\cdots,d\omega^2,d\overline{\omega}^2,d\omega^1,d\overline{\omega}^1\}.\]
This indicates that  $n\in \mathrm{Dpt}(J)$ and $\alpha$ has a summand of $a_n(\omega^n+\overline{\omega}^n)$, if $a_n,b_n$ are not both zeroes. This completes the proof of the corollary.
\end{proof}
\begin{remark}\label{dimV1}
Let $(\mathfrak{g},J)$ be MaxN. Then $\dim_{\mathbb{C}} V^1_{\mathbb{C}} = |\mathrm{Dpt}(J)|+1$,
where $|\mathrm{Dpt}(J)|$ denotes the number of elements in $\mathrm{Dpt}(J)$.
\end{remark}

\begin{theorem}\label{elementsV2}
Let $(\mathfrak{g},J)$ be MaxN and $\{\omega^k\}_{k=1}^n$ is an admissible coframe. Then, for $3\leq k \leq n$,
\[\omega^k \in V^2_{\mathbb{C}}\quad \Longleftrightarrow\quad
d\omega^k = \sum_{\begin{subarray}{c} 2\leq j\leq k-1 \\ j \in \mathrm{Dpt}(J) \end{subarray}} \!a_j \omega^1 \wedge (\omega^j + \overline{\omega}^j),\ \text{where}\ \,a_{k-1}\neq 0. \]
In particular, $\omega^k \in V^2_{\mathbb{C}}$ implies that
\[k-1 \in \mathrm{Dpt}(J),\quad k \notin \mathrm{Dpt}(J)\quad \text{and}\quad \omega^{k+1}\notin  V^2_{\mathbb{C}}.\]
\end{theorem}

\begin{proof}
From the definition of $V^2$ and Corollary \ref{basisV1}, the condition $\omega^k \in V^2 _{\mathbb{C}}$ implies that
\[d\omega^k \in \bigwedge^2 \mathrm{span}_{\mathbb{C}}\{\omega^1,\overline{\omega}^1,\omega^j+\overline{\omega}^j\},\]
where $j \in \mathrm{Dpt}(J)$ and $j\geq 2$. From the assumption \eqref{omegak} in Remark \ref{convention} and the fact that the type of $d\omega^k$ is $(2,0)$ and $(1,1)$, it follows that
\[ d\omega^k = \sum_{\begin{subarray}{c} 2\leq j\leq k-1 \\ j \in \mathrm{Dpt}(J) \end{subarray}} a_j \omega^1 \wedge (\omega^j + \overline{\omega}^j),\]
where $a_{k-1} \neq 0$ due to the fact that $d\omega^k$ needs to maintain the type of \eqref{streq_nil_cplx_mcn}. Therefore, it forces that $k-1 \in \mathrm{Dpt}(J)$ and $k \notin \mathrm{Dpt}(J)$, since the type of $d\omega^k$ is not $(1,1)$ and thus $\omega^k$ can not satisfy the statement \eqref{dpt} in Theorem \ref{equ_mcn_refined}. If $\omega^{k+1} \in V^2_{\mathbb{C}}$, then we would have $k \in \mathrm{Dpt}(J)$, a contradiction. So we must have $\omega^{k+1} \not\in V^2_{\mathbb{C}}$.

Conversely, assume that \[d\omega^k = \sum_{\begin{subarray}{c} 2\leq j\leq k-1 \\ j \in \mathrm{Dpt}(J) \end{subarray}} a_j \omega^1 \wedge (\omega^j + \overline{\omega}^j),\] where $a_{k-1}\neq 0$. Then it is obvious that $\omega^2 \in  V^2_{\mathbb{C}}$ by the definition of $V^2$. So the proof is completed.
\end{proof}

\begin{lemma}\label{V2todpt}
Let $(\mathfrak{g},J)$ be MaxN and $\{\omega^k\}_{k=1}^n$ is an admissible coframe. If for some $k\geq 2$, there are constants  $a_i$, $b_i$ ($1\leq i\leq k-1$) and $a_k\in \mathbb{C}$, such that  $a_k \neq 0$ and
\[a_k \omega^k + \sum_{i=1}^{k-1}(a_i \omega^i +b_i \overline{\omega}^i) \in V^2_{\mathbb{C}},\]
then it holds that $k-1 \in \mathrm{Dpt}(J)$.
\end{lemma}

\begin{proof}
Without loss of generality, we may assume that $k \geq 4$, since the other cases are trivial.
It is clear that, by the definition of $V^2$ and Corollary \ref{basisV1},
\[\begin{aligned}
& a_k d\omega^k + \sum_{i=1}^{k-1}(a_i d\omega^i +b_i d\overline{\omega}^i)\\
=\quad&\sum_{\begin{subarray}{c} 2 \leq p < q\\ p,q \in \mathrm{Dpt}(J) \end{subarray}}c_{pq} (\omega^p + \overline{\omega}^p) \wedge (\omega^q + \overline{\omega}^q)
 + \sum_{\begin{subarray}{c}  q \geq 2\\ q \in \mathrm{Dpt}(J) \end{subarray}}c_{1q}\omega^1 \wedge (\omega^q + \overline{\omega}^q) \ + \\
& +  \sum_{\begin{subarray}{c} q \geq 2 \\ q \in \mathrm{Dpt}(J) \end{subarray}}c_{q1} (\omega^q + \overline{\omega}^q) \wedge \overline{\omega}^1  \ \, + \ \,c_{11} \omega^1 \wedge \overline{\omega}^1,
\end{aligned}\]
for $c_{pq} \in \mathbb{C}$. Since $a_{k} \neq 0$, some terms containing $\omega^{k-1}$ or $\overline{\omega}^{k-1}$ will appear in the left hand side expression. On the other hand, all the indices of $\omega$ and $\overline{\omega}$ on the right hand side of the above equality belong to $\mathrm{Dpt}(J)$. This shows that $k-1 \in \mathrm{Dpt}(J)$.
\end{proof}

Another technical lemma is the following.

\begin{lemma}\label{d^2=0}
Let $\{\omega^k\}_{k=1}$ be a basis of $\mathfrak{g}^{1,0}$ satisfying the structure equation \eqref{streq_nil_cplx_mcn}. Then it holds that
\[\begin{cases}
B^k_{k-1,k-1}=0,\quad 3\leq k \leq n,\\
B^k_{i,k-1}=0,\quad 2 \leq i \leq k-2.
\end{cases}\]
Also,
\[A^k_{i,k-1}=0\quad \text{for}\quad 2 \leq i \leq k-2 \quad\Longleftrightarrow\quad B^k_{k-1,i}=0\quad \text{for}\quad 2 \leq i \leq k-2.\]
Furthermore, the condition $d\omega^k+d\overline{\omega}^k=0$ will force that the following equalities hold
\[\begin{cases}
A^k_{ij}=0,\quad i<j\leq k-1,\\
B^k_{k-1,i}=0,\quad 2 \leq i \leq k-2,\\
B^k_{1,k-1}=\overline{B^{k}_{k-1,1}} \neq 0,\\
B^k_{ij}=\overline{B^k_{ji}},\quad i,j \leq k-2.
\end{cases}\]
\end{lemma}

\begin{proof}
It is easy to see that $d\omega^k$ can be written as
\begin{equation}\label{sum_of_domega}
\begin{aligned}
d\omega^k &= B^k_{k-1,k-1} \omega^{k-1} \wedge \overline{\omega}^{k-1} \\
          &+ \sum_{2 \leq i \leq k-2} A_{i,k-1}^k \omega^i \wedge \omega^{k-1} + B^k_{k-1,i} \omega^{k-1} \wedge \overline{\omega}^i
           + B^k_{i,k-1} \omega^i \wedge \overline{\omega}^{k-1}\\
          &+ A^k_{1,k-1}\omega^1 \wedge \omega^{k-1} + B^k_{k-1,1} \omega^{k-1} \wedge \overline{\omega}^1 + B^k_{1,k-1} \omega^1 \wedge \overline{\omega}^{k-1} \\
          &+ \sum_{i<j\leq k-2}A^k_{ij}\omega^i \wedge \omega^j + \sum_{i,j\leq k-2}B^k_{ij}\omega^i \wedge \overline{\omega}^j,
\end{aligned}
\end{equation}
where the coefficients $\{A_{i,k-1}^k\}_{i<k-1}$, $\{B_{k-1,j}^k\}_{j=1}^{k-1}$ and $\{B_{i,k-1}^k\}_{i<k-1}$ don't all vanish.
The equality $d^2=0$ forces that $B^k_{k-1,k-1}=0$ when $3\leq k \leq n$, since $d(\omega^{k-1} \wedge \overline{\omega}^{k-1})$  contains at least two summands, one of which is some nonzero multiple of $\omega^{k-2} \wedge \overline{\omega}^{k-2} \wedge \overline{\omega}^{k-1}$, or $\omega^i \wedge \omega^{k-2} \wedge \overline{\omega}^{k-1}$ or $\omega^i \wedge \overline{\omega}^{k-2}\wedge \overline{\omega}^{k-1}$ for some $ i \leq k-3$, and the other one is some nonzero multiple of $\omega^{k-1} \wedge \omega^{k-2} \wedge \overline{\omega}^{k-2}$, or $\omega^{k-1}\wedge \overline{\omega}^i \wedge \overline{\omega}^{k-2}$ or $\omega^{k-1} \wedge \overline{\omega}^i \wedge \omega^{k-2}$ for some $ i \leq k-3$, while the $d$-operation of other terms in \eqref{sum_of_domega}, such as
\begin{gather*}
\{d(\omega^i \wedge \omega^{k-1})\}_{2 \leq i \leq k-2},\ \{d(\omega^{k-1} \wedge \overline{\omega}^i)\}_{2 \leq i \leq k-2},\
\{d(\omega^i \wedge \overline{\omega}^{k-1})\}_{2 \leq i \leq k-2},\\
d(\omega^1 \wedge \omega^{k-1}),\ d(\omega^{k-1} \wedge \overline{\omega}^1),\ d(\omega^1 \wedge \overline{\omega}^{k-1}),\\
\{d(\omega^i \wedge \omega^j)\}_{i<j\leq k-2},\ \{d(\omega^i \wedge \overline{\omega}^j)\}_{i,j\leq k-2},
\end{gather*}
have neither of the above two types of summands above by comparison. Here the condition that the coefficients $\{A_{i,k-2}^{k-1}\}_{i<k-2}$, $\{B_{k-2,j}^{k-1}\}_{j=1}^{k-2}$ and $\{B_{i,k-2}^{k-1}\}_{i<k-2}$ of $d\omega^{k-1}$ don't all vanish is used.

Then we note that $d(\omega^{k-2} \wedge \overline{\omega}^{k-1})$  contains at least one summand, which is some nonzero multiple of $\omega^{k-3} \wedge \overline{\omega}^{k-3} \wedge \overline{\omega}^{k-1}$, or $\omega^i \wedge \omega^{k-3} \wedge \overline{\omega}^{k-1}$ or $\omega^i \wedge \overline{\omega}^{k-3} \wedge \overline{\omega}^{k-1}$ for some $i \leq k-4$, while the $d$-operation of other terms in \eqref{sum_of_domega}, such as \begin{gather*}
\{d(\omega^i \wedge \omega^{k-1})\}_{2 \leq i \leq k-2},\ \{d(\omega^{k-1} \wedge \overline{\omega}^i)\}_{2 \leq i \leq k-2},\
\{d(\omega^i \wedge \overline{\omega}^{k-1})\}_{2 \leq i \leq k-3},\\
d(\omega^1 \wedge \omega^{k-1}),\ d(\omega^{k-1} \wedge \overline{\omega}^1),\ d(\omega^1 \wedge \overline{\omega}^{k-1}),\\
\{d(\omega^i \wedge \omega^j)\}_{i<j\leq k-2},\ \{d(\omega^i \wedge \overline{\omega}^j)\}_{i,j\leq k-2},
\end{gather*}
can't have such a summand by the type comparison, where the structure equation \eqref{streq_nil_cplx_mcn} is used. It implies that $B^{k}_{k-2,k-1}=0$. Similarly, it can also be shown that \[B^{k}_{k-3,k-1}=B^{k}_{k-4,k-1}=\cdots=B^{k}_{2,k-1}=0.\]
Therefore, $B^k_{i,k-1}=0$ for $2 \leq i \leq k-2$.

If $A^k_{i,k-1}=0$ for $2 \leq i \leq k-2$, then $d\omega^k$ is expressed as
\begin{equation}\label{sum_of_domega_A=0}
\begin{aligned}
d\omega^k &= \sum_{2 \leq i \leq k-2} B^k_{k-1,i} \omega^{k-1} \wedge \overline{\omega}^i \\
&+ A^k_{1,k-1}\omega^1 \wedge \omega^{k-1} + B^k_{k-1,1} \omega^{k-1} \wedge \overline{\omega}^1 + B^k_{1,k-1} \omega^1 \wedge \overline{\omega}^{k-1} \\
&+\sum_{i<j\leq k-2}A^k_{ij}\omega^i \wedge \omega^j+ \sum_{i,j\leq k-2}B^k_{ij}\omega^i \wedge \overline{\omega}^j.\\
\end{aligned}
\end{equation}
The equality $d^2=0$ implies $B^k_{k-1,k-2}=0$, since $d(\omega^{k-1} \wedge \overline{\omega}^{k-2})$ contains a summand which is some nonzero multiple of $\omega^{k-1} \wedge \omega^{k-3} \wedge \overline{\omega}^{k-3}$, or $\omega^{k-1} \wedge \overline{\omega}^i \wedge \overline{\omega}^{k-3}$ or $\omega^{k-1} \wedge \overline{\omega}^i \wedge \omega^{k-3}$ for some $i \leq k-4$, while the $d$-operation of other terms in \eqref{sum_of_domega_A=0}, such as
\begin{gather*}
\{d(\omega^{k-1} \wedge \overline{\omega}^i)\}_{2 \leq i \leq k-3},\\
d(\omega^{k-1} \wedge \omega^1),\ d(\omega^{k-1} \wedge \overline{\omega}^1),\ d(\omega^1 \wedge \overline{\omega}^{k-1}),\\
\{d(\omega^i \wedge \omega^j)\}_{i<j\leq k-2},\ \{d(\omega^i \wedge \overline{\omega}^j)\}_{i,j\leq k-2},
\end{gather*}
can't have such a summand due to the structure equation \eqref{streq_nil_cplx_mcn}. Similarly, it can be shown that
 \[B^{k}_{k-1,k-3}=B^{k}_{k-1,k-4}=\cdots=B^{k}_{k-1,2}=0.\]
Conversely if $B^k_{k-1,i}=0$ for $2 \leq i \leq k-2$, it can also be shown that $A^k_{i,k-i}=0$ for $2 \leq i \leq k-2$ as above.

It is clear that the condition $d\omega^k + d\overline{\omega}^k =0$ implies that the type of $d\omega^k$ is necessarily $(1,1)$ and thus $A_{ij}^k=0$ for $i<j \leq k-1$. Then $d\omega^k$ is expressed as
\begin{equation}\label{sum_of_domega_dept}
d\omega^k = \sum_{2 \leq i \leq k-2} B^k_{k-1,i} \omega^{k-1} \wedge \overline{\omega}^i + B^k_{k-1,1} \omega^{k-1} \wedge \overline{\omega}^1 + B^k_{1,k-1} \omega^1 \wedge \overline{\omega}^{k-1} + \sum_{i,j\leq k-2}B^k_{ij}\omega^i \wedge \overline{\omega}^j.\\
\end{equation}
Hence, the vanishing of $ B^k_{k-1,i} $ for $2 \leq i \leq k-2$ results from the equalities $A^k_{i,k-1}=0$ for $2 \leq i \leq k-2$ as shown above, and the equalities $B^k_{1,k-1}=\overline{B^{k}_{k-1,1}}$ and $B^k_{ij}=\overline{B^k_{ji}}$ for $i,j \leq k-2$ follow easily from $d\omega^k+d\overline{\omega}^k=0$. Here $B^k_{1,k-1}=\overline{B^{k}_{k-1,1}} \neq 0$ due to the fact that the coefficients of $d\omega^k$, concerning with $\omega^{k-1}$ or $\overline{\omega}^{k-1}$, don't all vanish, by the structure equation \eqref{streq_nil_cplx_mcn}.
\end{proof}

The following trivial lemma is needed for further investigation.

\begin{lemma}\label{basis_change}
Let $(\mathfrak{g},J)$ be MaxN and $\{\omega^i\}_{i=1}^n$ is an admissible coframe. Assume that $k \notin \mathrm{Dpt}(J)$ for some $k \geq 3$. Then, if a new coframe $\{\tilde{\omega}^i\}_{i=1}^n$ of $\mathfrak{g}^{1,0}$ is constructed by
\[\begin{cases}
\tilde{\omega}^i=\omega^i, \quad 1 \leq i \leq n, i \neq k, \\
\tilde{\omega}^k =c_k \omega^k + \sum_{i=1}^{k-1} c_i \omega^i,
\end{cases}\]
where $c_i \in \mathbb{C}$ and $c_k \neq 0$, it is still an admissible coframe.
\end{lemma}

\begin{proof}
It is easy to check that $\{d\tilde{\omega}^i\}_{i=1}^n$ still maintains the type of \eqref{streq_nil_cplx_mcn}
and $\{\tilde{\omega}^i\}_{ 1 \leq i \leq n,\,i \neq k}$ satisfies \eqref{dpt} or \eqref{indpt} in Theorem \ref{equ_mcn_refined}. So it suffices to show that for $a,b\in\mathbb{C}$ the following holds
\begin{equation}\label{tomega}
a \,d\tilde{\omega}^k + b\,d\overline{\tilde{\omega}}^k \in \mathrm{span}_{\mathbb{C}}\{d\tilde{\omega}^{k-1}\!,\,d\overline{\tilde{\omega}}^{k-1}\!,\,\cdots\!,\, d\tilde{\omega}^2\!,\,d\overline{\tilde{\omega}}^2\!,\,d\tilde{\omega}^1\!,\,d\overline{\tilde{\omega}}^1\} \ \Longleftrightarrow \ a=b=0.
\end{equation}
Since $k \notin \mathrm{Dpt}(J)$,  $\omega^k$ satisfies \eqref{indpt} in Theorem \ref{equ_mcn_refined}, and
\[a \,d\tilde{\omega}^k + b\,d\overline{\tilde{\omega}}^k = a c_k d\omega^k + b \overline{c_k} d \overline{\omega}^k
+ \sum_{i=1}^{k-1} (ac_i d\omega^i + b \overline{c_i} d\overline{\omega}^i),\]
where $c_k \neq 0$, the conclusion \eqref{tomega} above follows obviously so $\tilde{\omega}^k$ still satisfies \eqref{indpt} of Theorem \ref{equ_mcn_refined}.
\end{proof}

\subsection{The structure theorem of  MaxN $(\mathfrak{g},J)$ with $\nu(\mathfrak{g})=3$}

Now the condition $\nu(\mathfrak{g})=3$ comes to play in the study of the structure for MaxN complex structures.

\begin{proposition}\label{inpdtnext}
Let $(\mathfrak{g},J)$ be MaxN where $\nu(\mathfrak{g})=3$ and $\{\omega^i\}_{i=1}^n$ is an admissible coframe. Assume that $k \notin \mathrm{Dpt}(J)$ for some $3\leq k \leq n-2$. Then we have
 \begin{enumerate}
\item\label{fst511} $\omega^{k+1} \in V^3_{\mathbb{C}} \setminus V^2_{\mathbb{C}}$,
\item\label{scd511} the expression $d\omega^{k+2}$ is
\[\begin{aligned}
d\omega^{k+2} &= A^{k+2}_{1,k+1}\omega^1 \wedge \omega^{k+1} + B^{k+2}_{1,k+1} \omega^1 \wedge \overline{\omega}^{k+1} \\
              & + \sum_{i<j\leq k} A_{ij}^{k+2} \omega^i \wedge \omega^j + \sum_{i,j\leq k} B^{k+2}_{ij} \omega^i \wedge \overline{\omega}^j, \\
\end{aligned}\]
where $|A^{k+2}_{1,k+1}|=|B^{k+2}_{1,k+1}|\neq 0$, and $\omega^{k+1}$ satisfies the following:
\[ A_{1,k+1}^{k+2} \omega^{k+1} + B_{1,k+1}^{k+2} \overline{\omega}^{k+1} + \sum_{1<j\leq k}A_{1j}^{k+2} \omega^j + \sum_{j \leq k}B^{k+2}_{1j} \overline{\omega}^j \in V^2_{\mathbb{C}}, \]
\item\label{thd511} a new admissible coframe $\{\tilde{\omega}^i\}_{i=1}^n$ can be constructed such that
    \[ \tilde{\omega}^i =\omega^i \quad \text{for}\quad i \neq k \quad \text{and} \quad \tilde{\omega}^k \in  V^2_{\mathbb{C}}.\]
\end{enumerate}
\end{proposition}

\begin{proof}
The condition $k \notin \mathrm{Dpt}(J)$ for some $k$ satisfying $3\leq k \leq n-2$ implies that $\omega^{k+1} \in V^3_{\mathbb{C}}\setminus V^2_{\mathbb{C}}$, by Theorem \ref{elementsV2} and the assumption $\nu(\mathfrak{g})=3$. It follows by Lemma \ref{d^2=0} that
\[\begin{aligned}
d\omega^{k+2} &=\sum_{2\leq i \leq k} A^{k+2}_{i,k+1}\omega^i \wedge \omega^{k+1} + B^{k+2}_{k+1,i}\omega^{k+1}\wedge\overline{\omega}^i \\
              &+ A^{k+2}_{1,k+1}\omega^1 \wedge \omega^{k+1} + B^{k+2}_{k+1,1}\omega^{k+1} \wedge \overline{\omega}^1 + B^{k+2}_{1,k+1} \omega^1 \wedge \overline{\omega}^{k+1} \\
              & + \sum_{i<j\leq k} A_{ij}^{k+2} \omega^i \wedge \omega^j + \sum_{i,j\leq k} B^{k+2}_{ij} \omega^i \wedge \overline{\omega}^j. \\
\end{aligned}\]
The condition $\nu(\mathfrak{g})=3$ forces that $\omega^{k+2} \in V^3_{\mathbb{C}}$, which implies, by the very definition of $V^3$, \[d\omega^{k+2} \in  \bigwedge^2 V^2_{\mathbb{C}}.\]
By Lemma \ref{wcrt}, we know that for $\forall \,\theta \in \mathfrak{g}_{\mathbb{C}}$,
\[\iota_{\theta}d\omega^{k+2} \in V^2_{\mathbb{C}}.\]
Denote the dual frame of $\{\omega^i\}_{i=1}^n$  by $\{\theta_i\}_{i=1}^n$. Then for $1 \leq i \leq k$, we have
\[\iota_{\theta_i} d\omega^{k+2} \in V^2_{\mathbb{C}},\quad \iota_{\overline{\theta}_i} d\omega^{k+2} \in V^2_{\mathbb{C}}.\]
It follows that, for $2 \leq i \leq k$,
\[\begin{aligned}
\iota_{\theta_i} d\omega^{k+2} &= A_{i,k+1}^{k+2} \omega^{k+1} + \sum_{i<j\leq k}A_{ij}^{k+2} \omega^j
- \sum_{j <i} A^{k+2}_{ji} \omega^j+ \sum_{j \leq k}B^{k+2}_{ij} \overline{\omega}^j \in V^2_{\mathbb{C}}, \\
\iota_{\overline{\theta}_i} d\omega^{k+2} &= -B_{k+1,i}^{k+2} \omega^{k+1}- \sum_{j \leq k} B^{k+2}_{ji} \omega^j \in V^2_{\mathbb{C}},\\
\iota_{\overline{\theta}_1} d\omega^{k+2} &= -B_{k+1,1}^{k+2} \omega^{k+1}- \sum_{j \leq k} B^{k+2}_{j1} \omega^j \in V^2_{\mathbb{C}}.
\end{aligned}\]
From $k \notin \mathrm{Dpt}(J)$ and Lemma \ref{V2todpt}, it yields that, for $2 \leq i \leq k$,
\[ A_{i,k+1}^{k+2}=0,\quad B_{k+1,i}^{k+2}=0,\quad B_{k+1,1}^{k+2}=0,\]
so it is impossible for both $A_{1,k+1}^{k+2}$ and $B_{1,k+1}^{k+2}$ to be zero, since $d\omega^{k+2}$ satisfies the structure equation \eqref{streq_nil_cplx_mcn} and the coefficients of $d\omega^{k+2}$, concerning with $\omega^{k+1}$ or $\overline{\omega}^{k+1}$, don't all vanish.
Similarly, one has
\[\begin{aligned}
\iota_{\theta_1} d\omega^{k+2} &= A_{1,k+1}^{k+2} \omega^{k+1} + B_{1,k+1}^{k+2} \overline{\omega}^{k+1} + \sum_{1<j\leq k}A_{1j}^{k+2} \omega^j + \sum_{j \leq k}B^{k+2}_{1j} \overline{\omega}^j \in V^2_{\mathbb{C}}. \\
\end{aligned}\]
Note that $|A_{1,k+1}^{k+2}|=|B_{1,k+1}^{k+2}| \neq 0$. Otherwise, it follows that
\[ \begin{aligned}
A_{1,k+1}^{k+2} \omega^{k+1} + B_{1,k+1}^{k+2} \overline{\omega}^{k+1} + \sum_{1<j\leq k}A_{1j}^{k+2} \omega^j + \sum_{j \leq k}B^{k+2}_{1j} \overline{\omega}^j & \in V^2_{\mathbb{C}}, \\
\overline{B_{1,k+1}^{k+2}} \omega^{k+1} + \overline{A_{1,k+1}^{k+2}} \overline{\omega}^{k+1}
+ \sum_{j \leq k}\overline{B^{k+2}_{1j}} \omega^j + \sum_{1<j\leq k} \overline{A_{1j}^{k+2}} \overline{\omega}^j & \in V^2_{\mathbb{C}}.
\end{aligned}\]
Since $\det\begin{pmatrix} A_{1,k+1}^{k+2} & B_{1,k+1}^{k+2} \\[5pt] \overline{B_{1,k+1}^{k+2}} & \overline{A_{1,k+1}^{k+2}} \end{pmatrix}=|A_{1,k+1}^{k+2}|^2-|B_{1,k+1}^{k+2}|^2 \neq 0$, there exists complex numbers $\{c_i,d_i\}_{i=1}^k$, such that
\[\omega^{k+1} + \sum_{i=1}^k c_i \omega^i + d_i \overline{\omega}^i \in V^2_{\mathbb{C}},\]
which is a contradiction to Lemma \ref{V2todpt} due to the fact that $k \notin \mathrm{Dpt}(J)$. Therefore, the statements \eqref{fst511} and \eqref{scd511} are proved.

As to the statement \eqref{thd511}, by the statement \eqref{scd511}, the definition of $V^2$, and Corollary \ref{basisV1}, we have
\[\begin{aligned}
& d \Big(A_{1,k+1}^{k+2} \omega^{k+1} + B_{1,k+1}^{k+2} \overline{\omega}^{k+1} + \sum_{1<j\leq k}A_{1j}^{k+2} \omega^j + \sum_{j \leq k}B^{k+2}_{1j} \overline{\omega}^j \Big)\\
=\quad & \sum_{\begin{subarray}{c} 2 \leq p < q\\ p,q \in \mathrm{Dpt}(J) \end{subarray}}c_{pq} (\omega^p + \overline{\omega}^p) \wedge (\omega^q + \overline{\omega}^q) \ + \sum_{\begin{subarray}{c}  q \geq 2\\ q \in \mathrm{Dpt}(J) \end{subarray}}c_{1q}\omega^1 \wedge (\omega^q + \overline{\omega}^q)\ \, +   \\
& +  \sum_{\begin{subarray}{c} q \geq 2 \\ q \in \mathrm{Dpt}(J) \end{subarray}}c_{q1} (\omega^q + \overline{\omega}^q) \wedge \overline{\omega}^1 \ \, + \ \, c_{11} \omega^1 \wedge \overline{\omega}^1,
\end{aligned}\]
for $c_{pq} \in \mathbb{C}$. Note that the indices $p,q$ on the right hand side of the above equality are at most $k-1$ due to $k \notin \mathrm{Dpt}(J)$, and thus $d (A_{1,k+1}^{k+2} \omega^{k+1} + B_{1,k+1}^{k+2} \overline{\omega}^{k+1})$ contains no summands concerning with $\omega^k$ or $\overline{\omega}^k$. By Lemma \ref{d^2=0}, the term $d \omega^{k+1}$ can be expressed as
\[\begin{aligned}
d\omega^{k+1} &=\sum_{2\leq i \leq k-1} A^{k+1}_{i,k}\omega^i \wedge \omega^{k} + B^{k+1}_{k,i}\omega^{k}\wedge\overline{\omega}^i \\
              &+ A^{k+1}_{1,k}\omega^1 \wedge \omega^{k} + B^{k+1}_{k,1}\omega^{k} \wedge \overline{\omega}^1 + B^{k+1}_{1,k} \omega^1 \wedge \overline{\omega}^{k} \\
              & + \sum_{i<j\leq k-1} A_{ij}^{k+1} \omega^i \wedge \omega^j + \sum_{i,j\leq k-1} B^{k+1}_{ij} \omega^i \wedge \overline{\omega}^j. \\
\end{aligned}\]
The disappearance of $\omega^k$ and $\overline{\omega}^k$ in $d (A_{1,k+1}^{k+2} \omega^{k+1} + B_{1,k+1}^{k+2} \overline{\omega}^{k+1})$ implies that, for $2 \leq i \leq k-1$,  \[ A^{k+1}_{i,k}=0,\quad A^{k+1}_{1,k}=0.\]
By Lemma \ref{d^2=0} again, it yields that, for $2 \leq i \leq k-1$,  \[ B^{k+1}_{k,i}=0,\]
which follows that
\[\begin{aligned}
d\omega^{k+1} &= B^{k+1}_{k,1}\omega^{k} \wedge \overline{\omega}^1 + B^{k+1}_{1,k} \omega^1 \wedge \overline{\omega}^{k} \\
              & + \sum_{i<j\leq k-1} A_{ij}^{k+1} \omega^i \wedge \omega^j + \sum_{i,j\leq k-1} B^{k+1}_{ij} \omega^i \wedge \overline{\omega}^j, \\
\end{aligned}\]
where the equality $\det\begin{pmatrix} A^{k+2}_{1,k+1} & B^{k+2}_{1,k+1} \\[5pt] \overline{B^{k+1}_{1,k}} & B^{k+1}_{k,1}\end{pmatrix}=0$ holds and thus $B^{k+1}_{k,1}, B^{k+1}_{1,k}$ are both non-zero. Now apply the condition $\nu(\mathfrak{g})=3$  to $\omega^{k+1}$, which implies that
\[d\omega^{k+1} \in \bigwedge^2 V^2_{\mathbb{C}},\]
so by Lemma \ref{wcrt} we get,
\[\iota_{\overline{\theta}_1}d\omega^{k+1} =  -B^{k+1}_{1,k}\omega^{k} - \sum_{i\leq k-1} B^{k+1}_{i1} \omega^i \in V^2_{\mathbb{C}}.\]
Therefore, a new admissible coframe $\{\tilde{\omega}^i\}_{i=1}^n$ can be constructed as
\[\begin{cases}\tilde{\omega}^i = \omega^i,\quad 1 \leq i \leq n, \, i \neq k, \\
\tilde{\omega}^k = \omega^{k} + \frac{1}{B^{k+1}_{1,k}} \sum_{i\leq k-1} B^{k+1}_{i1} \omega^i,
\end{cases}\]
by Lemma \ref{basis_change}, and satisfies
$\tilde{\omega}^k \in V^2_{\mathbb{C}}$.
\end{proof}

\begin{proposition}\label{indpttodpt}
Let $(\mathfrak{g},J)$ be MaxN with $\nu(\mathfrak{g})=3$ and $\{\omega^i\}_{i=1}^n$ is an admissible coframe. If  $k \notin \mathrm{Dpt}(J)$ and $3 \leq k \leq n-3$, then  $k+1 \in \mathrm{Dpt}(J)$.
\end{proposition}

\begin{proof}
Assume the contrary, namely, $k+1 \notin \mathrm{Dpt}(J)$. Since $4 \leq k+1 \leq n-2$,  by Proposition \ref{inpdtnext}, we know that $\omega^{k+2}$ satisfies the condition
\[ A_{1,k+2}^{k+3} \omega^{k+2} + B_{1,k+2}^{k+3} \overline{\omega}^{k+2} + \sum_{1<j\leq k+1}A_{1j}^{k+3} \omega^j + \sum_{j \leq k+1}B^{k+3}_{1j} \overline{\omega}^j \in V^2_{\mathbb{C}}, \]
where the coefficients $A_{1,k+2}^{k+3},B_{1,k+2}^{k+3},\{A_{1j}^{k+3}\}_{1<j\leq k+1},\{B^{k+3}_{1j}\}_{j \leq k+1}$ come from $d\omega^{k+3}$ and $|A_{1,k+2}^{k+3}|=|B_{1,k+2}^{k+3}| \neq 0$. By the definition of $V^2$ and Corollary \ref{basisV1}, it similarly follows that
\[\begin{aligned}
& d \Big(A_{1,k+2}^{k+3} \omega^{k+2} + B_{1,k+2}^{k+3} \overline{\omega}^{k+2} + \sum_{1<j\leq k+1}A_{1j}^{k+3} \omega^j + \sum_{j \leq k+1}B^{k+3}_{1j} \overline{\omega}^j \Big)\\
=\quad&\sum_{\begin{subarray}{c} 2 \leq p < q\\ p,q \in \mathrm{Dpt}(J) \end{subarray}}c_{pq} (\omega^p + \overline{\omega}^p) \wedge (\omega^q + \overline{\omega}^q) \\
& + \sum_{\begin{subarray}{c}  q \geq 2\\ q \in \mathrm{Dpt}(J) \end{subarray}}c_{1q}\omega^1 \wedge (\omega^q + \overline{\omega}^q)
+  \sum_{\begin{subarray}{c} q \geq 2 \\ q \in \mathrm{Dpt}(J) \end{subarray}}c_{q1} (\omega^q + \overline{\omega}^q) \wedge \overline{\omega}^1 \\
& + c_{11} \omega^1 \wedge \overline{\omega}^1,
\end{aligned}\]
for $c_{pq} \in \mathbb{C}$. Note that $k \notin \mathrm{Dpt}(J)$ and $k+1 \notin \mathrm{Dpt}(J)$ now, and thus $p,q$ on the right hand side of the equality are at most $k-1$. However, from $k \notin \mathrm{Dpt}(J)$ and the statement \eqref{scd511} of Proposition \ref{inpdtnext}, it is clear that \[d \Big(A_{1,k+2}^{k+3} \omega^{k+2} + B_{1,k+2}^{k+3} \overline{\omega}^{k+2} + \sum_{1<j\leq k+1}A_{1j}^{k+3} \omega^j + \sum_{j \leq k+1}B^{k+3}_{1j} \overline{\omega}^j \Big)\]
indeed contains summands concerning with $\omega^{k+1}$ and $\overline{\omega}^{k+1}$, which is a contradiction. Therefore, we must have $k+1 \in \mathrm{Dpt}(J)$.
\end{proof}

Proposition \ref{indpttodpt} says that, if $k \notin \mathrm{Dpt}(J)$ for some $k$ in the range $3 \leq k \leq n-3$, then the  next index $k+1$ has to belong to $\mathrm{Dpt}(J)$. Similarly, when $k \in \mathrm{Dpt}(J)$ for some $3 \leq k \leq n-1$, then the next index $k+1$ can't live in $\mathrm{Dpt}(J)$ as shown in what follows.

\begin{proposition}\label{dpttoindpt}
Let $(\mathfrak{g},J)$ be MaxN with $\nu(\mathfrak{g})=3$ and $\{\omega^i\}_{i=1}^n$ is an admissible coframe. If $k \in \mathrm{Dpt}(J)$ and $3 \leq k \leq n-1$, then $k+1 \notin \mathrm{Dpt}(J)$.
\end{proposition}

\begin{proof}
By Lemma \ref{d^2=0}, the condition $k \in \mathrm{Dpt}(J)$ implies that
\begin{equation}\label{ddpt}
d\omega^k =  B^k_{k-1,1} \omega^{k-1} \wedge \overline{\omega}^1 + B^k_{1,k-1} \omega^1 \wedge \overline{\omega}^{k-1} + \sum_{i,j\leq k-2}B^k_{ij}\omega^i \wedge \overline{\omega}^j,
\end{equation}
where $B^k_{1,k-1}=\overline{B^{k}_{k-1,1}} \neq 0$ and $B^k_{ij}=\overline{B^k_{ji}}$ for $i,j \leq k-2$ hold. By Theorem \ref{elementsV2} and the assumption $\nu(\mathfrak{g})=3$, we get $\omega^k \in V^3_{\mathbb{C}} \setminus V^2_{\mathbb{C}}$.

Suppose that $k+1 \in \mathrm{Dpt}(J)$. It follows similarly as in Lemma \ref{d^2=0} that,
\begin{equation}
d\omega^{k+1} =  B^{k+1}_{k,1} \omega^{k} \wedge \overline{\omega}^1 + B^{k+1}_{1,k} \omega^1 \wedge \overline{\omega}^{k} + \sum_{i,j\leq k-1}B^{k+1}_{ij}\omega^i \wedge \overline{\omega}^j,
\end{equation}
where $B^{k+1}_{1,k}=\overline{B^{k+1}_{k,1}} \neq 0$ and $B^{k+1}_{ij}=\overline{B^{k+1}_{ji}}$ for $i,j \leq k-1$ hold. The theorem \ref{elementsV2} and the condition $\nu(\mathfrak{g})=3$ imply that
$\omega^{k+1} \in V^3_{\mathbb{C}} \setminus V^2_{\mathbb{C}}$. It yields that, from Lemma \ref{wcrt} and the definition of $V^3$, for $\theta \in \mathfrak{g}_{\mathbb{C}}$,
\[\iota_{\theta}d\omega^{k+1} \in V^2_{\mathbb{C}}.\]
With $\{\theta_i\}_{i=1}^n$ denoted by the dual basis of $\{\omega^i\}_{i=1}^n$, it follows that
\[\iota_{\overline{\theta}_1}d\omega^{k+1}=-B^{k+1}_{k,1} \omega^{k}-\sum_{i\leq k-1}B^{k+1}_{i1}\omega^{i} \in V^2_{\mathbb{C}}. \]
Then, since $B^{k+1}_{k,1} \omega^{k}+\sum_{i\leq k-1}B^{k+1}_{i1}\omega^{i}$ is of the type $(1,0)$, it yields that
\[
d \Big(B^{k+1}_{k,1} \omega^{k}+\sum_{i\leq k-1}B^{k+1}_{i1}\omega^{i}\Big)
= \sum_{\begin{subarray}{c} q \geq 2 \\ q \in \mathrm{Dpt}(J) \end{subarray}} c_q \omega^1 \wedge (\omega^q + \overline{\omega}^q)+c_{11}\omega^1 \wedge \overline{\omega}^1,
\]
for $c_q \in \mathbb{C}$, by the definition of $V^2$ and Corollary \ref{basisV1}. It implies that the summands of $B^{k+1}_{k,1}d \omega^{k}$, which concerns with $\omega^{k-1}$ or $\overline{\omega}^{k-1}$, are $ c_{k-1} \omega^1 \wedge (\omega^{k-1} + \overline{\omega}^{k-1})$, where $k-1 \in \mathrm{Dpt}(J)$ and $c_{k-1} \neq 0$, which contradicts with the fact that the summands of $d \omega^{k}$, concerning with $\omega^{k-1}$ or $\overline{\omega}^{k-1}$, are of the type $(1,1)$ as in the expression \eqref{ddpt}. This shows that  $k+1 \notin \mathrm{Dpt}(J)$.
\end{proof}

Now we are ready to prove the structural result Theorem \ref{strthmn-2} stated in the introduction.

\begin{proof}[Proof of Theorem \ref{strthmn-2}]
It is easy to see that the two cases can be separated as follows:
\[n-2 \in \mathrm{Dpt}(J)\quad \text{and} \quad n-2 \notin \mathrm{Dpt}(J).\]
When $n-2 \in \mathrm{Dpt}(J)$ holds, it follows that $n-3 \notin \mathrm{Dpt}(J)$ by Proposition \ref{dpttoindpt} and thus $n-4 \in \mathrm{Dpt}(J)$ by Proposition \ref{indpttodpt} until $k=3$ is reached. Therefore, the two statements in \eqref{n-2_dept} are proved.
Similarly, the two statements in \eqref{n-2_indept} are also established.

Now let us fix an admissible coframe $\{\tau^i\}_{i=1}^n$, which satisfies, due to \eqref{omega2} in Remark \ref{convention},
\[\begin{cases} d\tau^1 =0,\\ d\tau^2 = \tau^1 \wedge \overline{\tau}^1.\end{cases}\] The following direction is clear by Theorem \ref{elementsV2}, for $3 \leq k \leq n-2$,
\[k \notin \mathrm{Dpt}(J)\Longleftarrow \tau^k \in V^2_{\mathbb{C}}.\]
Conversely, Proposition \ref{inpdtnext} enable us to modify $\tau^k$ to $\omega^k$, with other terms of the coframe left unchanged, such that
$ \omega^k \in V^2_{\mathbb{C}}$
and the new coframe $\{\omega^i\}_{i=1}^n$ is still admissible, for $k \notin \mathrm{Dpt}(J)$ where $3 \leq k \leq n-2$. Therefore, there exists an admissible coframe $\{\omega^k\}_{k=1}^n$, such that the equalities are established
\[\begin{cases} d\omega^1 =0, \\d\omega^2 = \omega^1 \wedge \overline{\omega}^1, \end{cases}\]
and, for $3 \leq k \leq n-2$, it holds that
\[k \notin \mathrm{Dpt}(J)\Longleftrightarrow \omega^k \in V^2_{\mathbb{C}}.\]
Hence the following equivalence is also established under $\{\omega^k\}_{k=1}^n$, for $3 \leq k \leq n-2$,
\[k \in \mathrm{Dpt}(J) \Longleftrightarrow \omega^k \in V^3_{\mathbb{C}} \setminus V^2_{\mathbb{C}},\]
since $\nu(\mathfrak{g})=3$.

Finally, under the admissible coframe $\{\omega^k\}_{k=1}^n$, it is obvious that, from the results above,
\[\sum_{k=1}^{n-2}(a_k \omega^k + b_k \overline{\omega}^k ) \in V^2_{\mathbb{C}} \ \, \Longleftrightarrow \
\sum_{\begin{subarray}{c} 3 \leq k \leq n-2 \\  k \in \mathrm{Dpt}(J) \\ \end{subarray}}\!(a_k \omega^k + b_k \overline{\omega}^k)
\in V^2_{\mathbb{C}},\]
since $\omega^1,\omega^2$ and $\omega^k$ all belong to $V^2_{\mathbb{C}}$, where $k \notin \mathrm{Dpt}(J)$ and $3 \leq k \leq n-2$. It follows that, by the definition of $V^2$ and Corollary \ref{basisV1},
\[ \begin{aligned} d \Big( \sum_{\begin{subarray}{c} 3 \leq k \leq n-2 \\  k \in \mathrm{Dpt}(J) \\ \end{subarray}}a_k \omega^k + b_k \overline{\omega}^k\Big) =& \sum_{\begin{subarray}{c} 3 \leq k \leq n-2 \\  k \in \mathrm{Dpt}(J) \\ \end{subarray}}(a_k-b_k) d\omega^k \\
=&\sum_{\begin{subarray}{c} 2 \leq p < q\\ p,q \in \mathrm{Dpt}(J) \end{subarray}}c_{pq} (\omega^p + \overline{\omega}^p) \wedge (\omega^q + \overline{\omega}^q) \\
& + \sum_{\begin{subarray}{c}  q \geq 2\\ q \in \mathrm{Dpt}(J) \end{subarray}}c_{1q}\omega^1 \wedge (\omega^q + \overline{\omega}^q)
+  \sum_{\begin{subarray}{c} q \geq 2 \\ q \in \mathrm{Dpt}(J) \end{subarray}}c_{q1} (\omega^q + \overline{\omega}^q) \wedge \overline{\omega}^1 \\
& + c_{11} \omega^1 \wedge \overline{\omega}^1,
\end{aligned}\]
for $c_{pq} \in \mathbb{C}$. Note that $d\omega^k$ is of the type $(1,1)$, due to $k \in \mathrm{Dpt}(J)$, which implies that
\[c_{pq}=0\quad \text{for}\quad 2 \leq p <q\quad  \text{and}\quad c_{1q}=c_{q1}=0\quad \text{for}\quad q \geq 2.\]
As $d\omega^k$ has no summand of some multiple of $\omega^1\wedge \overline{\omega}^1$ for $3 \leq k \leq n$ from \eqref{omegak} in Remark \ref{convention}, it yields that $c_{11}=0$. Therefore, the conclusion follows
\[\sum_{\begin{subarray}{c} 3 \leq k \leq n-2 \\  k \in \mathrm{Dpt}(J) \\ \end{subarray}}(a_k-b_k) d\omega^k =0,\]
which implies that $a_k = b_k$, where $k \in \mathrm{Dpt}(J)$ and $3\leq k \leq n-2$, since $\{d\omega^k\}_{k=2}^n$ is $\mathbb{C}$-linearly independent. Conversely, when the condition $a_k = b_k$ for $k \in \mathrm{Dpt}(J)$ and $3\leq k \leq n-2$ is established, it is clear that
\[\sum_{\begin{subarray}{c} 3 \leq k \leq n-2 \\  k \in \mathrm{Dpt}(J) \\ \end{subarray}}a_k \omega^k + b_k \overline{\omega}^k
\ \in \ V^2_{\mathbb{C}},\]
since $\omega^k + \overline{\omega}^k \in V^1_{\mathbb{C}}$ for $k \in \mathrm{Dpt}(J)$, from Corollary \ref{basisV1}.
This completes the proof of Theorem \ref{strthmn-2}.
\end{proof}

Theorem \ref{strthmn-2} motivates the following terminology to study the MaxN pair $(\mathfrak{g},J)$ with $\nu(\mathfrak{g})=3$.
\begin{definition}
Let $(\mathfrak{g},J)$ be MaxN with $\nu(\mathfrak{g})=3$. The admissible coframe $\{\omega^k\}_{k=1}^n$ as in Theorem \ref{strthmn-2} is called a \emph{strictly admissible} coframe.
\end{definition}
The following results can be viewed as a complementarity of Theorem \ref{strthmn-2}, refining the choice of $\omega^{n-1}$ and $\omega^{n}$ in a strictly admissible coframe $\{\omega^k\}_{k=1}^n$ above.

\begin{proposition}\label{n-2dptn-1}
Let $(\mathfrak{g},J)$ be MaxN, where $\nu(\mathfrak{g})=3$ and $\dim_{\mathbb{C}}\mathfrak{g}=n \geq 5$, with a strictly admissible coframe $\{\omega^k\}_{k=1}^n$. Assume that $n-2 \in \mathrm{Dpt}(J)$.
Then
\begin{enumerate}
\item\label{omegandpt} $n-1 \notin \mathrm{Dpt}(J)$ and $\omega^{n} \in V^3_{\mathbb{C}} \setminus V^2_{\mathbb{C}}$,
\item\label{framen-1dpt} one of the following holds:
\begin{enumerate}
\item\label{n-1inDPT} there exists a new strictly admissible coframe $\{\tilde{\omega}^k\}_{k=1}^n$, such that
    \[\tilde{\omega}^{k}=\omega^k\quad \text{for}\quad k \neq n-1\quad \text{and}\quad \tilde{\omega}^{n-1} \in V^2_{\mathbb{C}},\]
\item\label{n-1NinDPT} if \eqref{n-1inDPT} does not hold, a new strictly admissible coframe $\{\tilde{\omega}^k\}_{k=1}^n$ could still be constructed, such that
    \[\tilde{\omega}^{k}=\omega^k\quad \text{for}\quad k \leq n-2,\quad
    \tilde{\omega}^{n-1} +\overline{\tilde{\omega}}^{n-1} \in V^2_{\mathbb{C}},\]
    and at this time, the expression of $d\tilde{\omega}^n$ is
    \[
    d\tilde{\omega}^{n} = \tilde{\omega}^1 \wedge \tilde{\omega}^{n-1} +  \tilde{\omega}^1 \wedge \overline{\tilde{\omega}}^{n-1}
    + \sum_{i<j\leq n-2} \tilde{A}_{ij}^{n} \tilde{\omega}^i \wedge \tilde{\omega}^j + \sum_{i,j\leq n-2} \tilde{B}^{n}_{ij} \tilde{\omega}^i \wedge \overline{\tilde{\omega}}^j, \]
\end{enumerate}
\item\label{crnV2n-1dpt} under the strictly admissible coframe $\{\tilde{\omega}^k\}_{k=1}^n$ above, for $a_k,b_k \in \mathbb{C}$, the condition \[\sum_{k=1}^{n-1} (a_k\tilde{\omega}^k + b_k \overline{\tilde{\omega}}^k ) \, \in \, V^2_{\mathbb{C}}\] holds
    if and only if one of the following occurs, with respect to the two cases \eqref{n-1inDPT} and \eqref{n-1NinDPT} above,
\begin{enumerate}
\item\label{crn1} $a_k=b_k,\quad \text{for}\ \, k\in \mathrm{Dpt}(J)\ \text{and}\ \, 3 \leq k \leq n-2$,
\item\label{crn2} $a_{n-1}=b_{n-1},\quad a_k=b_k,\quad \text{for}\ \, k\in \mathrm{Dpt}(J)\ \text{and}\ \, 3 \leq k \leq n-2$.
\end{enumerate}
\end{enumerate}
\end{proposition}

\begin{proof}
It is clear that $n-1 \notin \mathrm{Dpt}(J)$ by the assumption $n-2 \in \mathrm{Dpt}(J)$ and Proposition \ref{dpttoindpt}.
It also follows that $\omega^{n} \in V^3_{\mathbb{C}} \setminus V^2_{\mathbb{C}}$, by Theorem \ref{elementsV2} and $\nu(\mathfrak{g})=3$.

From Lemma \ref{d^2=0}, the expression of $d \omega^{n}$ is
\[\begin{aligned}
d\omega^{n} &=\sum_{2\leq i \leq n-2} A^{n}_{i,n-1}\omega^i \wedge \omega^{n-1} + B^{n}_{n-1,i}\omega^{n-1}\wedge\overline{\omega}^i \\
              &+ A^{n}_{1,n-1}\omega^1 \wedge \omega^{n-1} + B^{n}_{n-1,1}\omega^{n-1} \wedge \overline{\omega}^1 + B^{n}_{1,n-1} \omega^1 \wedge \overline{\omega}^{n-1} \\
              & + \sum_{i<j\leq n-2} A_{ij}^{n} \omega^i \wedge \omega^j + \sum_{i,j\leq n-2} B^{n}_{ij} \omega^i \wedge \overline{\omega}^j. \\
\end{aligned}\]
It follows that \[d\omega^n \in \bigwedge^2 V^2_{\mathbb{C}},\]
by the definition of $V^3$ and $\omega^{n} \in V^3_{\mathbb{C}}$. Hence, from Lemma \ref{wcrt}, for any $\theta \in \mathfrak{g}_{\mathbb{C}}$, it yields that \[\iota_{\theta}d\omega^n \in V^2_{\mathbb{C}}.\]
We claim that there exists a new strictly admissible coframe $\{\tilde{\omega}^k\}_{k=1}^n$, such that \[ \tilde{\omega}^{k}=\omega^k\quad \text{for}\quad k \neq n-1\quad \text{and}\quad \tilde{\omega}^{n-1} \in V^2_{\mathbb{C}},\]
if one of the following is established:
\begin{enumerate}[$($A$)$]
\item\label{A} $B^{n}_{n-1,i} \neq 0$ for some $i$ satisfying $2 \leq i \leq n-2$,
\item\label{B} $B^{n}_{n-1,1} \neq 0$,
\item\label{C} $|A^{n}_{1,n-1}|\neq |B^{n}_{1,n-1}|$, when $B^{n}_{n-1,1} = 0$ and $B^{n}_{n-1,i} = 0$ for $2 \leq i \leq n-2$.
\end{enumerate}
Actually, when \eqref{A} holds, with $\{\theta_k\}_{k=1}^n$ denoted by the dual basis of $\{\omega^k\}_{k=1}^n$, it yields that
\[ \iota_{\overline{\theta}_i}d\omega^n = -B^n_{n-1,i} \omega^{n-1} -\sum_{j \leq n-2}B^n_{ji}\omega^j \in V^2_{\mathbb{C}},\]
which enables us to construct a new coframe $\{\tilde{\omega}^k\}_{k=1}^n$ such that
\[\begin{cases}
\tilde{\omega}^{n-1} = \omega^{n-1} + \frac{1}{B^n_{n-1,i}}\sum_{j \leq n-2}B^n_{ji}\omega^j, \\
\tilde{\omega}^{k}=\omega^k, \quad k \neq n-1.
\end{cases}\]
It is easy to verify that the new coframe $\{\tilde{\omega}^k\}_{k=1}^n$ is still strictly admissible by Lemma \ref{basis_change}. The claim can be proved by the same method, when \eqref{B} occurs.

When \eqref{C} holds, it yields that \[A^n_{i,n-1}=0, \quad 2 \leq i \leq n-2,\]
due to Lemma \ref{d^2=0}. The expression of $d \omega^{n}$ amounts to
\[\begin{aligned}
d\omega^{n} &= A^{n}_{1,n-1}\omega^1 \wedge \omega^{n-1} + B^{n}_{1,n-1} \omega^1 \wedge \overline{\omega}^{n-1} \\
              & + \sum_{i<j\leq n-2} A_{ij}^{n} \omega^i \wedge \omega^j + \sum_{i,j\leq n-2} B^{n}_{ij} \omega^i \wedge \overline{\omega}^j. \\
\end{aligned}\]
And thus, it follows that
\[\iota_{\theta_1} d\omega^n =A^{n}_{1,n-1} \omega^{n-1} + B^{n}_{1,n-1} \overline{\omega}^{n-1}
+ \sum_{1<j\leq n-2} A_{1j}^{n} \omega^j + \sum_{j\leq n-2} B^{n}_{1j} \overline{\omega}^j \in V^2_{\mathbb{C}},\]
which yields, after conjugation,
\[ \overline{B^{n}_{1,n-1}} \omega^{n-1} + \overline{A^{n}_{1,n-1}} \overline{\omega}^{n-1}
+ \sum_{j\leq n-2} \overline{B^{n}_{1j}} \omega^j + \sum_{1<j\leq n-2} \overline{A_{1j}^{n}} \overline{\omega}^j \in V^2_{\mathbb{C}}.\]
Since $|A^{n}_{1,n-1}|\neq |B^{n}_{1,n-1}|$, it implies that there exist $c_i,d_i\in \mathbb{C}$ such that
\[\omega^{n-1} + \sum_{i=1}^{n-2} c_i \omega^i + d_i \overline{\omega}^i \in V^2_{\mathbb{C}},\]
which is equivalent to, by Theorem \ref{strthmn-2},
\[\omega^{n-1} + \sum_{\begin{subarray}{c} 3 \leq i \leq n-2 \\ i \in \mathrm{Dpt}(J)\end{subarray}} c_i \omega^i + d_i \overline{\omega}^i \ \in \ V^2_{\mathbb{C}}.\]
Note that $\omega^i + \overline{\omega}^i \in V^1_{\mathbb{C}}$ for $i \in \mathrm{Dpt}(J)$ by Corollary \ref{basisV1}, which implies that
\[\omega^{n-1} + \sum_{\begin{subarray}{c} 3 \leq i \leq n-2 \\ i \in \mathrm{Dpt}(J)\end{subarray}} (c_i -d_i) \omega^i \in V^2_{\mathbb{C}}.\]
Therefore, a new coframe $\{\tilde{\omega}^k\}_{k=1}^n$ can be constructed as
\[\begin{cases}
\tilde{\omega}^{n-1} =\omega^{n-1} + \sum_{\begin{subarray}{c} 3 \leq i \leq n-2 \\ i \in \mathrm{Dpt}(J)\end{subarray}} (c_i -d_i) \omega^i, \\
\tilde{\omega}^{k}=\omega^k, \quad k \neq n-1,
\end{cases}\]
which is strictly admissible by Lemma \ref{basis_change}. Hence, the claim is proved and \eqref{n-1inDPT} is established.

Now suppose that the statement \eqref{n-1inDPT} does not hold. This implies that \eqref{A}, \eqref{B} and \eqref{C} all fails. It leads to \[B^{n}_{n-1,1} = 0, \quad B^{n}_{n-1,i} = 0\quad \text{for}\quad 2 \leq i \leq n-2, \quad\text{and} \quad|A^{n}_{1,n-1}| = |B^{n}_{1,n-1}|.\]
At this time, it also yields that $A^n_{i,n-1}=0$ for $2 \leq i \leq n-2$ by Lemma \ref{d^2=0}.
By the same argument as in the proof of the \eqref{C} case above, it follows that
\begin{equation}\label{domega-n}\begin{aligned}
d\omega^{n} &= A^{n}_{1,n-1}\omega^1 \wedge \omega^{n-1} + B^{n}_{1,n-1} \omega^1 \wedge \overline{\omega}^{n-1} \\
              & + \sum_{i<j\leq n-2} A_{ij}^{n} \omega^i \wedge \omega^j + \sum_{i,j\leq n-2} B^{n}_{ij} \omega^i \wedge \overline{\omega}^j, \\
\end{aligned}\end{equation}
\begin{equation}\label{left_case_1}
\iota_{\theta_1} d\omega^n =A^{n}_{1,n-1} \omega^{n-1} + B^{n}_{1,n-1} \overline{\omega}^{n-1}
+ \sum_{1<j\leq n-2} A_{1j}^{n} \omega^j + \sum_{j\leq n-2} B^{n}_{1j} \overline{\omega}^j \in V^2_{\mathbb{C}},
\end{equation}
where $|A^{n}_{1,n-1}| = |B^{n}_{1,n-1}| \neq 0$ due to the fact that the coefficients of $d\omega^n$, concerning with $\omega^{n-1}$ or $\overline{\omega}^{n-1}$, don't all vanish. By Theorem \ref{strthmn-2}, the equality \eqref{left_case_1} is equivalent to the following
\begin{equation}\label{left_case_2}
A^{n}_{1,n-1} \omega^{n-1} + B^{n}_{1,n-1} \overline{\omega}^{n-1}
+ \sum_{\begin{subarray}{c} 3\leq j\leq n-2 \\ j \in \mathrm{Dpt}(J)\end{subarray}} A_{1j}^{n} \omega^j
+ \sum_{\begin{subarray}{c} 3\leq j\leq n-2 \\ j \in \mathrm{Dpt}(J)\end{subarray}} B^{n}_{1j} \overline{\omega}^j \in V^2_{\mathbb{C}}.
\end{equation}
By the definition of $V^2$ and Corollary \ref{basisV1}, it follows that
\[\begin{aligned}
& d \Big(A^{n}_{1,n-1} \omega^{n-1} + B^{n}_{1,n-1} \overline{\omega}^{n-1}
+ \sum_{\begin{subarray}{c} 3\leq j\leq n-2 \\ j \in \mathrm{Dpt}(J)\end{subarray}} A_{1j}^{n} \omega^j
+ \sum_{\begin{subarray}{c} 3\leq j\leq n-2 \\ j \in \mathrm{Dpt}(J)\end{subarray}} B^{n}_{1j} \overline{\omega}^j\Big)\\
=\quad&\sum_{\begin{subarray}{c} 2 \leq p < q\\ p,q \in \mathrm{Dpt}(J) \end{subarray}}c_{pq} (\omega^p + \overline{\omega}^p) \wedge (\omega^q + \overline{\omega}^q) \\
& + \sum_{\begin{subarray}{c}  q \geq 2\\ q \in \mathrm{Dpt}(J) \end{subarray}}c_{1q}\omega^1 \wedge (\omega^q + \overline{\omega}^q)
+  \sum_{\begin{subarray}{c} q \geq 2 \\ q \in \mathrm{Dpt}(J) \end{subarray}}c_{q1} (\omega^q + \overline{\omega}^q) \wedge \overline{\omega}^1 \\
& + c_{11} \omega^1 \wedge \overline{\omega}^1,
\end{aligned}\]
for $c_{pq} \in \mathbb{C}$. Note that
\[ \overline{d(A^{n}_{1,n-1} \omega^{n-1} + B^{n}_{1,n-1} \overline{\omega}^{n-1})}
 \ = \ \frac{\overline{A^n_{1,n-1}}\overline{B^n_{1,n-1}}}{|B^n_{1,n-1}|^2}d (A^{n}_{1,n-1} \omega^{n-1} + B^{n}_{1,n-1} \overline{\omega}^{n-1}).\]
And $\{(\omega^p + \overline{\omega}^p) \wedge (\omega^q + \overline{\omega}^q)\}_{\!\!\!\!\!\!\begin{subarray}{c} 2 \leq p < q\\ p,q \in \mathrm{Dpt}(J) \end{subarray}}$,
$\,\{\omega^1 \wedge (\omega^q + \overline{\omega}^q)\}_{\!\!\!\!\!\!\!\begin{subarray}{c}  q \geq 2\\ q \in \mathrm{Dpt}(J) \end{subarray}}$, $\{(\omega^q + \overline{\omega}^q) \wedge \overline{\omega}^1\}_{\!\!\!\!\!\!\!\begin{subarray}{c} q \geq 2 \\ q \in \mathrm{Dpt}(J) \end{subarray}}$,
$\,\omega^1 \wedge \overline{\omega}^1$,
$\,\{d\omega^j\}_{\begin{subarray}{c} 3\leq j\leq n-2 \\ j \in \mathrm{Dpt}(J)\end{subarray}}$
are all $\mathbb{C}$-linearly independent, since $d\omega^j$ is of the type $(1,1)$ for $ j \in \mathrm{Dpt}(J)$ and $3\leq j\leq n-2 $, which has no summand of some multiple of $\omega^1 \wedge \overline{\omega}^1$ by \eqref{omegak} in Remark \ref{convention}. It implies that, for $ j \in \mathrm{Dpt}(J)$ and $3\leq j\leq n-2 $,
\[-(\overline{A^n_{1j}-B^n_{1j}}) \ = \ \frac{\overline{A^n_{1,n-1}}\overline{B^n_{1,n-1}}}{|B^n_{1,n-1}|^2}(A^n_{1j}-B^n_{1j}).\]
Then there exist complex numbers $\{e_j\}_{\begin{subarray}{c} 3\leq j\leq n-2 \\ j \in \mathrm{Dpt}(J) \end{subarray}} $ such that
\[A^n_{1,n-1}e_j + A^n_{1j} = B^n_{1,n-1}\overline{e}_j +B^n_{1j},\]
for instance, $e_j$ can be set as $-\frac{(A^n_{1j}-B^n_{1j})}{2A^n_{1,n-1}}$, which enables us to construct a new coframe
$\{ \tilde{\omega}^k\}_{k=1}^n$ as
\[\begin{cases}
\tilde{\omega}^{n-1} =\omega^{n-1} -\sum_{\begin{subarray}{c} 3 \leq j \leq n-2 \\ j \in \mathrm{Dpt}(J)\end{subarray}} e_j \omega^j, \\
\tilde{\omega}^{k}=\omega^k, \quad k \neq n-1.
\end{cases}\]
It is easy to verify that the coframe $\{\tilde{\omega}^k\}_{k=1}^n$ is strictly admissible by Lemma \ref{basis_change}
and \[A^n_{1,n-1}\tilde{\omega}^{n-1} + B^n_{1,n-1}\overline{\tilde{\omega}}^{n-1} \in V^2_{\mathbb{C}},\]
since \eqref{left_case_2} yields that
\[A^n_{1,n-1}\tilde{\omega}^{n-1} + B^n_{1,n-1}\overline{\tilde{\omega}}^{n-1}
+\sum_{\begin{subarray}{c} 3\leq j\leq n-2 \\ j \in \mathrm{Dpt}(J)\end{subarray}} (A^n_{1j}+A^n_{1,n-1}e_j) \omega^j +(B^n_{1j}+B^n_{1,n-1}\overline{e}_j) \overline{\omega}^j \ \in \ V^2_{\mathbb{C}},\]
and $\omega^j + \overline{\omega}^j \in V^1_{\mathbb{C}}$ for $j \in \mathrm{Dpt}(J)$ by Corollary \ref{basisV1}.
After a possible substitute of  $c\tilde{\omega}^{n-1}$ for $\tilde{\omega}^{n-1}$, where $|c|=1$, it yields that \[\tilde{\omega}^{n-1} + \overline{\tilde{\omega}}^{n-1} \in V^2_{\mathbb{C}}.\]
After another possible division of $\omega^n$ by a nonzero multiple, denoted by $\tilde{\omega}^n$, the expression of $d\tilde{\omega}^n$, from \eqref{domega-n}, becomes
\[\begin{aligned}
d\tilde{\omega}^{n} &= \tilde{\omega}^1 \wedge \tilde{\omega}^{n-1} +  \tilde{\omega}^1 \wedge \overline{\tilde{\omega}}^{n-1} \\
              & + \sum_{i<j\leq n-2} \tilde{A}_{ij}^{n} \tilde{\omega}^i \wedge \tilde{\omega}^j + \sum_{i,j\leq n-2} \tilde{B}^{n}_{ij} \tilde{\omega}^i \wedge \overline{\tilde{\omega}}^j, \\
\end{aligned}\]
Therefore the statement \eqref{n-1NinDPT} follows.

As to the statement \eqref{crnV2n-1dpt}, it is easy to see that, when \eqref{n-1inDPT} is established,
\[\sum_{k=1}^{n-1} (a_k\tilde{\omega}^k + b_k \overline{\tilde{\omega}}^k ) \ \in \ V^2_{\mathbb{C}} \ \, \Longleftrightarrow \ a_k =b_k,\ \, \text{for}\  k\in \mathrm{Dpt}(J)\ \text{and}\  3 \leq k \leq n-2,\]
from \eqref{crnV2n-2} in Theorem \ref{strthmn-2}. When \eqref{n-1NinDPT} is established,
\[\sum_{k=1}^{n-1} (a_k\tilde{\omega}^k + b_k \overline{\tilde{\omega}}^k ) \ \in \ V^2_{\mathbb{C}}\]
implies that $a_{n-1}=b_{n-1}$. Otherwise, there would exist $a'_k,b'_k \in \mathbb{C}$ such that
\[\tilde{\omega}^{n-1} + \sum_{k=1}^{n-2} a'_k \tilde{\omega}^k+b'_k\overline{\tilde{\omega}}^k \ \in \ V^2_{\mathbb{C}}\]
and thus
\[\tilde{\omega}^{n-1} + \sum_{\begin{subarray}{c} 3 \leq k \leq n-2 \\ k \in \mathrm{Dpt}(J) \end{subarray}} a'_k \tilde{\omega}^k+b'_k\overline{\tilde{\omega}}^k \ \in \ V^2_{\mathbb{C}}\]
by Theorem \ref{strthmn-2}, which would enable us to set a new $\hat{\omega}^{n-1}$ to be
\[\tilde{\omega}^{n-1} + \sum_{\begin{subarray}{c} 3 \leq k \leq n-2 \\ k \in \mathrm{Dpt}(J) \end{subarray}} (a'_k -b'_k) \tilde{\omega}^k\]
such that $\hat{\omega}^{n-1} \in V^2_{\mathbb{C}}$. This is a contradiction to the assumption of \eqref{n-1NinDPT}. After $a_{n-1}=b_{n-1}$ is proved, it follows easily that $a_k =b_k$ for $k\in \mathrm{Dpt}(J)$ and $3 \leq k \leq n-2$ by Theorem \ref{strthmn-2}. The converse  is rather clear, so we have completed the proof here.
\end{proof}

\begin{corollary}\label{estV2-dpt}
Let $(\mathfrak{g},J)$ be MaxN where $\nu(\mathfrak{g})=3$ and $\dim_{\mathbb{C}}\mathfrak{g}=n \geq 5$. Assume that $n-2 \in \mathrm{Dpt}(J)$ and let $\{\omega^i\}_{i=1}^n$ be a strictly admissible coframe, satisfying the property of Proposition \ref{n-2dptn-1}. Then a basis of $V^2 _{\mathbb{C}}$ can be chosen as the following
\begin{enumerate}
\item when $n \in \mathrm{Dpt}(J)$,
\[\{\omega^n+\overline{\omega}^n\!,\,\omega^{n-1}\!,\,\overline{\omega}^{n-1}\!,\,\omega^{k}+\overline{\omega}^{k}\!,\,\omega^{\ell},\overline{\omega}^{\ell},
\omega^2\!,\,\overline{\omega}^2\!,\,\omega^1\!,\,\overline{\omega}^1\},\]
\item when $n \notin \mathrm{Dpt}(J)$,
\begin{enumerate}
\item if $\omega^{n-1} \in V^2_{\mathbb{C}}$, after a new strictly admissible coframe $\{\tilde{\omega}^i\}_{i=1}^n$,
satisfying $\tilde{\omega}^i = \omega^i$ for $i\neq n$, is constructed,
\begin{gather*}
\{\tilde{\omega}^n+\overline{\tilde{\omega}}^n\!,\,\tilde{\omega}^{n-1}\!,\,\overline{\tilde{\omega}}^{n-1}\!,\,
\tilde{\omega}^{k}+\overline{\tilde{\omega}}^{k}\!,\, \tilde{\omega}^{\ell}\!,\, \overline{\tilde{\omega}}^{\ell}\!,\,
\tilde{\omega}^2\!,\,\overline{\tilde{\omega}}^2\!,\,\tilde{\omega}^1\!,\,\overline{\tilde{\omega}}^1\} \\
or \quad\{ \tilde{\omega}^{n-1}\!,\, \overline{\tilde{\omega}}^{n-1}\!,\,
\tilde{\omega}^{k}+\overline{\tilde{\omega}}^{k}\!,\, \tilde{\omega}^{\ell}\!,\, \overline{\tilde{\omega}}^{\ell}\!,\,
\tilde{\omega}^2\!,\, \overline{\tilde{\omega}}^2\!,\, \tilde{\omega}^1\!,\, \overline{\tilde{\omega}}^1 \},
\end{gather*}
\item if $\omega^{n-1} \notin V^2_{\mathbb{C}}$ but $\omega^{n-1} + \overline{\omega}^{n-1} \in V^2_{\mathbb{C}}$,
    \[\{\omega^{n-1}+\overline{\omega}^{n-1}\!,\, \omega^{k}+\overline{\omega}^{k}\!,\, \omega^{\ell},\overline{\omega}^{\ell}\!,\,
    \omega^2\!,\, \overline{\omega}^2\!,\, \omega^1\!,\, \overline{\omega}^1\},\]
\end{enumerate}
\end{enumerate}
where $k \in \mathrm{Dpt}(J)$ and $3 \leq k \leq n-2$, $\ell \notin \mathrm{Dpt}(J)$ and $3 \leq \ell \leq n-2$ for all the cases above.
\end{corollary}

\begin{proof}
Consider \begin{equation}\label{basisV2dpt}
\sum_{k=1}^n a_k\omega^k + b_k \overline{\omega}^k \in V^2_{\mathbb{C}},
\end{equation}
where $a_k,b_k \in \mathbb{C}$. The two cases can be separated as $n \in \mathrm{Dpt}(J)$ and $n \notin \mathrm{Dpt}(J)$.

When $n \in \mathrm{Dpt}(J)$ holds, it follows that the statement \eqref{n-1inDPT} of Proposition \ref{n-2dptn-1} is established and thus  $\omega^{n-1} \in V^2_{\mathbb{C}}$, otherwise the establishment of \eqref{n-1NinDPT} of Proposition \ref{n-2dptn-1} would lead to the expression
\[\begin{aligned}
    d\omega^{n} \ =  & \ \ \omega^1 \wedge \omega^{n-1} +  \omega^1 \wedge \overline{\omega}^{n-1} \\
    & + \sum_{i<j\leq n-2} A_{ij}^{n} \omega^i \wedge \omega^j + \sum_{i,j\leq n-2} B^{n}_{ij} \omega^i \wedge \overline{\omega}^j,
    \end{aligned}\]
which would imply $n \notin \mathrm{Dpt}(J)$, a contradiction. Hence, it yields that
\[\omega^n + \overline{\omega}^n \in V^1_{\mathbb{C}} \subseteq  V^2_{\mathbb{C}}, \]
and \eqref{basisV2dpt} implies $a_n =b_n$, since $a_n \neq b_n$ would lead to
\[\omega^n + \sum_{k=1}^{n-1}c_k \omega^k + d_k \overline{\omega}^k \in V^2_{\mathbb{C}},\]
for some $c_k,d_k \in \mathbb{C}$, which would contradict with $n-1 \notin \mathrm{Dpt}(J)$, by Lemma \ref{V2todpt} and Proposition \ref{n-2dptn-1}. Therefore, from \eqref{crn1} of Proposition \ref{n-2dptn-1}, \eqref{basisV2dpt} is equivalent to
\[a_n=b_n,\quad a_k=b_k,\quad for \quad k\in \mathrm{Dpt}(J)\quad and \quad 3 \leq k \leq n-2,\]
then the basis of $V^2_{\mathbb{C}}$ for $n\in\mathrm{Dpt}(J)$ in the corollary follows.

When $n \notin \mathrm{Dpt}(J)$ holds, another two cases still can be separated with respect to the statements \eqref{n-1inDPT} and \eqref{n-1NinDPT} of Proposition \ref{n-2dptn-1}
\begin{enumerate}
\item\label{inV2dpt} $\omega^{n-1} \in V^2_{\mathbb{C}}$,
\item\label{NinV2dpt} $\omega^{n-1} \notin V^2_{\mathbb{C}}$ but $\omega^{n-1} + \overline{\omega}^{n-1} \in V^2_{\mathbb{C}}$.
\end{enumerate}
If \eqref{NinV2dpt} occurs, it leads to $a_n=b_n=0$. Actually, by \eqref{n-1NinDPT} of Proposition \ref{n-2dptn-1}, it yields that
\[\begin{aligned}
    d\omega^{n}\, & = \ \omega^1 \wedge \omega^{n-1} +  \omega^1 \wedge \overline{\omega}^{n-1} \\
    &+ \sum_{i<j\leq n-2} A_{ij}^{n} \omega^i \wedge \omega^j + \sum_{i,j\leq n-2} B^{n}_{ij} \omega^i \wedge \overline{\omega}^j.
    \end{aligned}\]
The non-vanishing of $a_n$ or $b_n$ in \eqref{basisV2dpt} implies that $\omega^{n-1}$ and $\overline{\omega}^{n-1}$ appear in the expression
\[d \Big( \sum_{k=1}^n a_k\omega^k + b_k \overline{\omega}^k \Big).\]
However, \eqref{basisV2dpt} indicates that, by the definition of $V^2$ and Corollary \ref{basisV1},
\[\begin{aligned}
& d \Big(\sum_{k=1}^n a_k\omega^k + b_k \overline{\omega}^k \Big)\\
=\quad&\sum_{\begin{subarray}{c} 2 \leq p < q\\ p,q \in \mathrm{Dpt}(J) \end{subarray}}c_{pq} (\omega^p + \overline{\omega}^p) \wedge (\omega^q + \overline{\omega}^q) \\
& + \sum_{\begin{subarray}{c}  q \geq 2\\ q \in \mathrm{Dpt}(J) \end{subarray}}c_{1q}\omega^1 \wedge (\omega^q + \overline{\omega}^q)
+  \sum_{\begin{subarray}{c} q \geq 2 \\ q \in \mathrm{Dpt}(J) \end{subarray}}c_{q1} (\omega^q + \overline{\omega}^q) \wedge \overline{\omega}^1 \\
& + c_{11} \omega^1 \wedge \overline{\omega}^1,
\end{aligned}\]
for $c_{pq} \in \mathbb{C}$. Note that this would imply $n-1 \in \mathrm{Dpt}(J)$, which is a contradiction. Then it follows, from \eqref{crn2} of Proposition \ref{n-2dptn-1}, that \eqref{basisV2dpt} is equivalent to
\[a_{n-1}=b_{n-1},\quad a_k=b_k,\quad for \quad k\in \mathrm{Dpt}(J)\quad and \quad 3 \leq k \leq n-2.\]
Hence, the basis of $V^2_{\mathbb{C}}$ for this case is
\[\{\omega^{n-1}+\overline{\omega}^{n-1}\!,\, \omega^{k}+\overline{\omega}^{k}\!,\, \omega^{\ell}\!,\, \overline{\omega}^{\ell}\!,\,
    \omega^2\!,\, \overline{\omega}^2\!,\, \omega^1\!,\, \overline{\omega}^1\},\]
where $k \in \mathrm{Dpt}(J)$ and $3 \leq k \leq n-2$, $\ell \notin \mathrm{Dpt}(J)$ and $3 \leq \ell \leq n-2$.

If \eqref{inV2dpt} holds, \eqref{basisV2dpt}
implies that \[a_n=b_n=0\quad or\quad |a_{n}| =| b_{n}| \neq 0,\]
since $a_n\neq0,b_n=0$ or $a_n=0,b_n\neq0$ would lead to $n-1 \in \mathrm{Dpt}(J)$, by Lemma \ref{V2todpt}, which is absurd from Proposition \ref{n-2dptn-1}, while $|a_n|\neq|b_n|$ would yield \[\omega^n + \sum_{k=1}^{n-1} a'_k \omega^k + b'_k \overline{\omega}^k \in V^2_{\mathbb{C}},\]
for some $a'_k, b'_k \in \mathbb{C}$, which reduces to the situation above.
Therefore, when \eqref{basisV2dpt} implies $a_n=b_n=0$, the basis of $V^2_{\mathbb{C}}$ for this case is clearly, from \eqref{crn1} of Proposition \ref{n-2dptn-1},
\[\{\omega^{n-1}\!,\, \overline{\omega}^{n-1}\!,\,
\omega^{k}+\overline{\omega}^{k}\!,\, \omega^{\ell}\!,\, \overline{\omega}^{\ell}\!,\,
\omega^2\!,\, \overline{\omega}^2\!,\, \omega^1,\overline{\omega}^1\},\]
where $k\in \mathrm{Dpt}(J)$ and $3 \leq k \leq n-2$, $\ell \notin \mathrm{Dpt}(J)$ and $3 \leq \ell \leq n-2$.
When \eqref{basisV2dpt} implies $|a_n|=|b_n|\neq0$, it reduces to
\[a_n \omega^n + b_n \overline{\omega}^n + \sum_{k=1}^{n-2}a_k \omega^k + b_k \overline{\omega}^k \in V^2_{\mathbb{C}},\]
by the argument after \eqref{left_case_2} in the proof of \eqref{n-1NinDPT} of Proposition \ref{n-2dptn-1},
which enable us to modify $\omega^n$ to $\tilde{\omega}^n$, with others left unchanged, such that the new coframe $\{\tilde{\omega}^k\}_{k=1}^n$ is still strictly admissible and
\[\tilde{\omega}^n + \overline{\tilde{\omega}}^n \in V^2_{\mathbb{C}}.\]
Under this new coframe $\{\tilde{\omega}^k\}_{k=1}^n$,
\[\sum_{k=1}^n a_k \tilde{\omega}^k + b_k \overline{\tilde{\omega}}^k \in V^2_{\mathbb{C}}\]
is equivalent to
\[a_n = b_n,\quad a_k=b_k,\quad for \quad k\in \mathrm{Dpt}(J)\quad and \quad 3 \leq k \leq n-2.\]
Therefore, the basis of $V^2_{\mathbb{C}}$ for this case in the corollary follows.
\end{proof}

\begin{proposition}\label{n-2indptn-1}
Let $(\mathfrak{g},J)$ be MaxN, where $\nu(\mathfrak{g})=3$ and $\dim_{\mathbb{C}}\mathfrak{g}=n \geq 5$, with a strictly admissible coframe $\{\omega^k\}_{k=1}^n$. Assume that $n-2 \notin \mathrm{Dpt}(J)$.
Then
\begin{enumerate}
\item\label{omeganindpt} $n \notin \mathrm{Dpt}(J)$.
\item\label{framen-1indpt} there exists no strictly admissible coframe $\{\hat{\omega}^k\}_{k=1}^n$ satisfying $\hat{\omega}^{n-1} \in V^2_{\mathbb{C}}$, but there does exist a strictly admissible coframe $\{\tilde{\omega}^k\}_{k=1}^n$ such that
      \[\tilde{\omega}^{k}=\omega^k\quad \text{for}\quad k \leq n-2,\quad \tilde{\omega}^{n-1} + \overline{\tilde{\omega}}^{n-1} \in V^2_{\mathbb{C}},\]
      and
    \[
    d\tilde{\omega}^{n} \ = \ \tilde{\omega}^1 \wedge \tilde{\omega}^{n-1} +  \tilde{\omega}^1 \wedge \overline{\tilde{\omega}}^{n-1}
    + \sum_{i<j\leq n-2} \tilde{A}_{ij}^{n} \tilde{\omega}^i \wedge \tilde{\omega}^j + \sum_{i,j\leq n-2} \tilde{B}^{n}_{ij} \tilde{\omega}^i \wedge \overline{\tilde{\omega}}^j. \]
\item\label{crnV2n-1indpt} under the strictly admissible coframe $\{\tilde{\omega}^k\}_{k=1}^n$ above, for $a_k,b_k \in \mathbb{C}$, it yields that \[\sum_{k=1}^{n-1}a_k\tilde{\omega}^k + b_k \overline{\tilde{\omega}}^k \in V^2_{\mathbb{C}}\] holds
    if and only if
\[a_{n-1}=b_{n-1},\quad a_k=b_k,\quad \text{for}\quad k\in \mathrm{Dpt}(J)\quad \text{and}\quad 3 \leq k \leq n-2.\]
\end{enumerate}
\end{proposition}

\begin{proof}
It follows that, from Proposition \ref{inpdtnext}, the expression $d\omega^{n}$ is
\begin{equation}\label{domegan-indpt}\begin{aligned}
d\omega^{n} &= A^{n}_{1,n-1}\omega^1 \wedge \omega^{n-1} + B^{n}_{1,n-1} \omega^1 \wedge \overline{\omega}^{n-1} \\
              & + \sum_{i<j\leq n-2} A_{ij}^{n} \omega^i \wedge \omega^j + \sum_{i,j\leq n-2} B^{n}_{ij} \omega^i \wedge \overline{\omega}^j, \\
\end{aligned}\end{equation}
where $|A^{n}_{1,n-1}|=|B^{n}_{1,n-1}|\neq 0$ and $\omega^{n-1}$ satisfies
\begin{equation}\label{left_case_3}
A_{1,n-1}^{n} \omega^{n-1} + B_{1,n-1}^{n} \overline{\omega}^{n-1} + \sum_{1<j\leq n-2}A_{1j}^{n} \omega^j + \sum_{j \leq n-2}B^{n}_{1j} \overline{\omega}^j \in V^2_{\mathbb{C}}.
\end{equation}
Hence $n \notin \mathrm{Dpt}(J)$, since $d\omega^n$ is not of type $(1,1)$. The existence of a strictly admissible coframe $\{\tilde{\omega}^k\}_{k=1}^n$, satisfying $\tilde{\omega}^{n-1} \in V^2_{\mathbb{C}}$,
would lead to $n-2 \in \mathrm{Dpt}(J)$, by Theorem \ref{elementsV2}, which is a contradiction. By the argument after \eqref{left_case_2} in the proof of \eqref{n-1NinDPT} of Proposition \ref{n-2dptn-1}, \eqref{left_case_3} implies that there exists a strictly admissible coframe $\{\tilde{\omega}^k\}_{k=1}^n$, such that
\[\tilde{\omega}^{k}=\omega^k\quad \text{for}\quad k \leq n-2,\quad
\tilde{\omega}^{n-1} + \overline{\tilde{\omega}}^{n-1} \in V^2_{\mathbb{C}},\]
and at this time, by \eqref{domegan-indpt}, the expression of $d\tilde{\omega}^n$ becomes
\[\begin{aligned}
d\tilde{\omega}^{n} &= \tilde{\omega}^1 \wedge \tilde{\omega}^{n-1} +  \tilde{\omega}^1 \wedge \overline{\tilde{\omega}}^{n-1} \\
&+ \sum_{i<j\leq n-2} \tilde{A}_{ij}^{n} \tilde{\omega}^i \wedge \tilde{\omega}^j + \sum_{i,j\leq n-2} \tilde{B}^{n}_{ij} \tilde{\omega}^i \wedge \overline{\tilde{\omega}}^j. \end{aligned}\]

As to the statement \eqref{crnV2n-1indpt},
\[\sum_{k=1}^{n-1}a_k\tilde{\omega}^k + b_k \overline{\tilde{\omega}}^k \in V^2_{\mathbb{C}}\]
implies that $a_{n-1}=b_{n-1}$, by the same proof of the statement \eqref{crn2} in Proposition \ref{n-2dptn-1}, where the nonexistence of a strictly admissible coframe $\{\tilde{\omega}^k\}_{k=1}^n$, satisfying $\tilde{\omega}^{n-1} \in V^2_{\mathbb{C}}$, is used. Then it follows easily that $a_k =b_k$ for $k\in \mathrm{Dpt}(J)$ and $3 \leq k \leq n-2$ by Theorem \ref{strthmn-2}.
The opposite direction is rather clear. Therefore, the proof is completed.
\end{proof}

\begin{corollary}\label{estV2-indpt}
Let $(\mathfrak{g},J)$ be MaxN, where $\nu(\mathfrak{g})=3$ and $\dim_{\mathbb{C}}\mathfrak{g}=n\geq 5$. Assume that $n-2 \notin \mathrm{Dpt}(J)$ and let $\{\omega^i\}_{i=1}^n$ be a strictly admissible coframe satisfying the property of Proposition \ref{n-2indptn-1}. Then a basis of $V^2_{\mathbb{C}}$ can be chosen as
\begin{enumerate}
\item when $n-1 \in \mathrm{Dpt}(J)$, after a new strictly admissible coframe $\{\tilde{\omega}^i\}_{i=1}^n$,
satisfying $\tilde{\omega}^i = \omega^i$ for $i\neq n$, is applied,
\begin{gather*}
\{ \tilde{\omega}^{n-1}+\overline{\tilde{\omega}}^{n-1}\!,\,
\tilde{\omega}^{k}+\overline{\tilde{\omega}}^{k}\!,\, \tilde{\omega}^{\ell}\!,\, \overline{\tilde{\omega}}^{\ell}\!,\,
\tilde{\omega}^2\!,\, \overline{\tilde{\omega}}^2\!,\, \tilde{\omega}^1\!,\, \overline{\tilde{\omega}}^1 \} \\
\text{or} \quad \{\tilde{\omega}^n\!,\, \overline{\tilde{\omega}}^n\!,\, \tilde{\omega}^{n-1}+\overline{\tilde{\omega}}^{n-1}\!,\,
\tilde{\omega}^{k}+\overline{\tilde{\omega}}^{k}\!,\, \tilde{\omega}^{\ell}\!,\, \overline{\tilde{\omega}}^{\ell}\!,\,
\tilde{\omega}^2\!,\, \overline{\tilde{\omega}}^2\!,\, \tilde{\omega}^1\!,\, \overline{\tilde{\omega}}^1\}\\
\text{or}\quad \{\tilde{\omega}^n+\overline{\tilde{\omega}}^n\!,\, \tilde{\omega}^{n-1}+\overline{\tilde{\omega}}^{n-1}\!,\,
\tilde{\omega}^{k}+\overline{\tilde{\omega}}^{k}\!,\, \tilde{\omega}^{\ell}\!,\, \overline{\tilde{\omega}}^{\ell}\!,\,
\tilde{\omega}^2\!,\, \overline{\tilde{\omega}}^2\!,\, \tilde{\omega}^1\!,\, \overline{\tilde{\omega}}^1\}\!,\,
\end{gather*}
\item when $n-1 \notin \mathrm{Dpt}(J)$,
\[\{ \omega^{n-1}+\overline{\omega}^{n-1}\!,\, \omega^k+\overline{\omega}^k\!,\, \omega^{\ell}\!,\, \overline{\omega}^{\ell}\!,\,
\omega^2\!,\, \overline{\omega}^2\!,\, \omega^1\!,\, \overline{\omega}^1 \},\]
\end{enumerate}
where $k \in \mathrm{Dpt}(J)$ and $3 \leq k \leq n-2$, $\ell \notin \mathrm{Dpt}(J)$ and $3 \leq \ell \leq n-2$ for all the cases above.
\end{corollary}

\begin{proof}
Consider \begin{equation}\label{basisV2indpt}
\sum_{k=1}^n a_k\omega^k + b_k \overline{\omega}^k \in V^2_{\mathbb{C}},
\end{equation}
where $a_k,b_k \in \mathbb{C}$. The two cases can be separated as $n-1 \in \mathrm{Dpt}(J)$ and $n-1 \notin \mathrm{Dpt}(J)$.

When $n-1 \notin \mathrm{Dpt}(J)$, it follows that $a_n=b_n=0$. Actually, from \eqref{framen-1indpt} of Proposition \ref{n-2indptn-1},
it yields that
\[\begin{aligned}
    d\omega^{n} &= \omega^1 \wedge \omega^{n-1} +  \omega^1 \wedge \overline{\omega}^{n-1} \\
    &+ \sum_{i<j\leq n-2} A_{ij}^{n} \omega^i \wedge \omega^j + \sum_{i,j\leq n-2} B^{n}_{ij} \omega^i \wedge \overline{\omega}^j.
    \end{aligned}\]
The non-vanishing of $a_n$ or $b_n$ in \eqref{basisV2indpt} implies that $\omega^{n-1}$ and $\overline{\omega}^{n-1}$ appear in the expression
\[d \Big( \sum_{k=1}^n a_k\omega^k + b_k \overline{\omega}^k \Big).\]
However, \eqref{basisV2indpt} indicates that, by the definition of $V^2$ and Corollary \ref{basisV1},
\[\begin{aligned}
& d \Big(\sum_{k=1}^n a_k\omega^k + b_k \overline{\omega}^k \Big)\\
=\quad&\sum_{\begin{subarray}{c} 2 \leq p < q\\ p,q \in \mathrm{Dpt}(J) \end{subarray}}c_{pq} (\omega^p + \overline{\omega}^p) \wedge (\omega^q + \overline{\omega}^q) \\
& + \sum_{\begin{subarray}{c}  q \geq 2\\ q \in \mathrm{Dpt}(J) \end{subarray}}c_{1q}\omega^1 \wedge (\omega^q + \overline{\omega}^q)
+  \sum_{\begin{subarray}{c} q \geq 2 \\ q \in \mathrm{Dpt}(J) \end{subarray}}c_{q1} (\omega^q + \overline{\omega}^q) \wedge \overline{\omega}^1 \\
& + c_{11} \omega^1 \wedge \overline{\omega}^1,
\end{aligned}\]
for $c_{pq} \in \mathbb{C}$. This would imply $n-1 \in \mathrm{Dpt}(J)$, which is a contradiction. Therefore, from \eqref{crnV2n-1indpt} of Proposition \ref{n-2indptn-1}, the basis of $V^2_{\mathbb{C}}$ for this case is
\[\{ \omega^{n-1}+\overline{\omega}^{n-1},\omega^k+\overline{\omega}^k,\omega^{\ell},\overline{\omega}^{\ell},
\omega^2,\overline{\omega}^2,\omega^1,\overline{\omega}^1 \},\]
where $k \in \mathrm{Dpt}(J)$ and $3 \leq k \leq n-2$, $\ell \notin \mathrm{Dpt}(J)$ and $3 \leq \ell \leq n-2$.

When $n-1 \in \mathrm{Dpt}(J)$, it follows, by \eqref{framen-1indpt} of Proposition \ref{n-2indptn-1}, that
\[\omega^{n-1} \notin V^2_{\mathbb{C}} \quad \text{but}\quad \omega^{n-1} + \overline{\omega}^{n-1} \in V^1_{\mathbb{C}}.\]
Then \eqref{basisV2indpt} may imply that
\[a_n=b_n=0\quad \text{or}\quad |a_n| \neq |b_n|\quad \text{or}\quad |a_n|=|b_n|\neq0.\]
If $a_n=b_n=0$ is established, the basis of $V^2_{\mathbb{C}}$ for this case is, from \eqref{crnV2n-1indpt} of Proposition \ref{n-2indptn-1},
\[\{\omega^{n-1}+\overline{\omega}^{n-1},\omega^{k}+\overline{\omega}^k,\omega^{\ell},\overline{\omega}^{\ell},
\omega^{2},\overline{\omega}^{2},\omega^{1},\overline{\omega}^{1}\},\]
where $k\in \mathrm{Dpt}(J)$ and $3 \leq k \leq n-2$, $\ell \notin \mathrm{Dpt}(J)$ and $3 \leq \ell \leq n-2$.

If $|a_n|\neq |b_n|$ holds, \eqref{basisV2indpt} reduces to
\[ \omega^n + \sum_{k=1}^{n-1}c_k\omega^{k} + d_k \overline{\omega}^{k} \in V^2_{\mathbb{C}}, \]
for some $c_k,d_k \in \mathbb{C}$, which is equivalent to, from Theorem \ref{strthmn-2} and $n-1\in \mathrm{Dpt}(J)$,
\[ \omega^n + \sum_{\begin{subarray}{c} 3 \leq k \leq n-1 \\ k \in \mathrm{Dpt}(J)\end{subarray}}
c_k\omega^{k} + d_k \overline{\omega}^{k} \in V^2_{\mathbb{C}}. \]
Therefore, a new coframe $\{\tilde{\omega}^k\}_{k=1}^n$ can be constructed as
\[\begin{cases}
\tilde{\omega}^{n} =\omega^{n} + \sum_{\begin{subarray}{c} 3 \leq i \leq n-1 \\ i \in \mathrm{Dpt}(J)\end{subarray}} (c_i -d_i) \omega^i, \\
\tilde{\omega}^{k}=\omega^k, \quad k \neq n,
\end{cases}\]
which is still strictly admissible by Lemma \ref{basis_change}, such that $\tilde{\omega}^n \in V^2_{\mathbb{C}}$.
Under this new coframe $\{\tilde{\omega}^k\}_{k=1}^n$,
\[\sum_{k=1}^n a_k \tilde{\omega}^k + b_k \overline{\tilde{\omega}}^k \in V^2_{\mathbb{C}}\]
is equivalent to
\[a_{n-1} = b_{n-1},\quad a_k=b_k,\quad for \quad k\in \mathrm{Dpt}(J)\quad and \quad 3 \leq k \leq n-2,\]
and the basis of $V^2_{\mathbb{C}}$ for this case is \[\{\tilde{\omega}^n,\overline{\tilde{\omega}}^n,\tilde{\omega}^{n-1}+\overline{\tilde{\omega}}^{n-1},
\tilde{\omega}^{k}+\overline{\tilde{\omega}}^{k},\tilde{\omega}^{\ell},\overline{\tilde{\omega}}^{\ell},
\tilde{\omega}^2,\overline{\tilde{\omega}}^2,\tilde{\omega}^1,\overline{\tilde{\omega}}^1\},\]
where $k \in \mathrm{Dpt}(J)$ and $3 \leq k \leq n-2$, $\ell \notin \mathrm{Dpt}(J)$ and $3 \leq \ell \leq n-2$.

If $|a_n|=|b_n|\neq0$ holds, \eqref{basisV2indpt} reduces to
\[ a_n\omega^n + b_n \overline{\omega}^{n} + \sum_{k=1}^{n-1} a_k \omega^k + b_k \overline{\omega}^k \in V^2_{\mathbb{C}},\]
which is equivalent to, from Theorem \ref{strthmn-2} and $n-1 \in \mathrm{Dpt}(J)$,
\[ a_n\omega^n + b_n \overline{\omega}^{n} + \sum_{\begin{subarray}{c} 3 \leq k \leq n-1 \\ k \in \mathrm{Dpt}(J)\end{subarray}}
a_k \omega^k + b_k \overline{\omega}^k \in V^2_{\mathbb{C}}.\]
By the argument after \eqref{left_case_2} in the proof of \eqref{n-1NinDPT} of Proposition \ref{n-2dptn-1},
it enables us to modify $\omega^n$ to $\tilde{\omega}^n$, with others left unchanged, such that the new coframe $\{\tilde{\omega}^k\}_{k=1}^n$ is still strictly admissible and
\[\tilde{\omega}^n + \overline{\tilde{\omega}}^n \in V^2_{\mathbb{C}}.\]
Without loss of generality, let us assume the non-existence of a strictly admissible coframe $\{\hat{\omega}^k\}_{k=1}^n$ such that $\hat{\omega}^k = \tilde{\omega}^k$ for $k \neq n$ and $\hat{\omega}^n \in V^2_{\mathbb{C}}$ here, since it would reduce to the case $|a_n|\neq |b_n|$ above.
Under this assumption and the new coframe $\{\tilde{\omega}^k\}_{k=1}^n$,
\[\sum_{k=1}^n a_k \tilde{\omega}^k + b_k \overline{\tilde{\omega}}^k \in V^2_{\mathbb{C}}\]
is equivalent to
\[a_n = b_n,\quad a_k=b_k,\quad for \quad k\in \mathrm{Dpt}(J)\quad and \quad 3 \leq k \leq n-2.\]
Therefore, the basis of $V^2_{\mathbb{C}}$ for this case is
\[ \{\tilde{\omega}^n+\overline{\tilde{\omega}}^n,\tilde{\omega}^{n-1}+\overline{\tilde{\omega}}^{n-1},
\tilde{\omega}^{k}+\overline{\tilde{\omega}}^{k},\tilde{\omega}^{\ell},\overline{\tilde{\omega}}^{\ell},
\tilde{\omega}^2,\overline{\tilde{\omega}}^2,\tilde{\omega}^1,\overline{\tilde{\omega}}^1\}\]
where $k \in \mathrm{Dpt}(J)$ and $3 \leq k \leq n-2$, $\ell \notin \mathrm{Dpt}(J)$ and $3 \leq \ell \leq n-2$.
\end{proof}

\begin{proof}[Proof of Corollary \ref{firstB-MaxN}]
Note that
\[\dim_{\mathbb{R}}\mathfrak{g}^i = \dim_{\mathbb{C}}\mathfrak{g}^i_{\mathbb{C}}
= 2n - \dim_{\mathbb{C}} V_{\mathbb{C}}^i \quad \text{for} \quad 1 \leq i \leq 2.\]
It is easy to figure out that
\[|P|=\lfloor\frac{n-3}{2}\rfloor \quad \text{and} \quad |Q|=\lfloor \frac{n-4}{2} \rfloor,\]
where \[\begin{split} P = \{ k \in \mathbb{Z} \big| 3 \leq k \leq n-2,\, k\equiv n-2 \mod 2\}, \\
Q = \{ k \in \mathbb{Z} \big| 3 \leq k \leq n-2,\, k\not\equiv n-2\mod 2\},
\end{split}\]
and the number of elements in the set $P$ is denoted by $|P|$.

The two cases will be separated as $n-2 \in \mathrm{Dpt}(J)$ and $n-2 \notin \mathrm{Dpt}(J)$ to consider the problem. When $n-2 \in \mathrm{Dpt}(J)$, Remark \ref{dimV1}, Theorem \ref{strthmn-2}, Proposition \ref{n-2dptn-1} and Corollary \ref{estV2-dpt} imply that
\[
\lfloor \frac{n+3}{2}\rfloor \leq \dim_{\mathbb{C}}V^1_{\mathbb{C}} \leq  \lfloor \frac{n+5}{2}\rfloor,\quad
\lfloor \frac{3n-2}{2}\rfloor \leq \dim_{\mathbb{C}}V^2_{\mathbb{C}} \leq  \lfloor \frac{3n+2}{2}\rfloor.\]
In the case when $n-2 \notin \mathrm{Dpt}(J)$, Remark \ref{dimV1}, Theorem \ref{strthmn-2}, Proposition \ref{n-2indptn-1} and Corollary \ref{estV2-indpt} imply that
\[\lfloor \frac{n+2}{2}\rfloor \leq \dim_{\mathbb{C}}V^1_{\mathbb{C}} \leq  \lfloor \frac{n+4}{2}\rfloor,\quad
\lfloor \frac{3n-1}{2}\rfloor \leq \dim_{\mathbb{C}}V^2_{\mathbb{C}} \leq  \lfloor \frac{3n+3}{2}\rfloor.\]
It follows that \[\lfloor \frac{n+2}{2}\rfloor \leq \dim_{\mathbb{C}}V^1_{\mathbb{C}} \leq  \lfloor \frac{n+5}{2}\rfloor,\quad
\lfloor \frac{3n-2}{2}\rfloor \leq \dim_{\mathbb{C}}V^2_{\mathbb{C}} \leq  \lfloor \frac{3n+3}{2}\rfloor,\]
and thus
\[  \lfloor \frac{3n-4}{2}\rfloor \leq \dim_{\mathbb{R}} \mathfrak{g}^1 \leq \lfloor \frac{3n-1}{2}\rfloor ,
\quad \lfloor \frac{n-2}{2}\rfloor \leq \dim_{\mathbb{R}} \mathfrak{g}^2 \leq  \lfloor \frac{n+3}{2}\rfloor.\]
A  celebrated theorem of Nomizu \cite{Nom} asserts that
\[ b_1(M) = \dim_{\mathbb{R}} V^1 =\dim_{\mathbb{C}} V^1_{\mathbb{C}}, \]
so the conclusion of the corollary follows.
\end{proof}

\textbf{Acknowledgement}: The first and second named authors are grateful to
the Mathematics Department of Ohio State University for the nice research environment and the warm hospitality during their stay,
and also to Professor Bo Guan for his help and interest. The authors would like to thank Professors Fino and Ugarte for their suggestion
and comments regarding to this paper.

\end{document}